\documentclass[reqno]{amsart}
\usepackage{amsmath}
\usepackage{amsthm}
\usepackage{amssymb}
\usepackage{amsfonts}
\usepackage[foot]{amsaddr}
\usepackage{mathrsfs}
\usepackage{bm}
\usepackage{bbm}
\usepackage{color}
\usepackage{geometry}
\usepackage{graphicx}
\usepackage{hyperref}
\hypersetup{hypertex=true,
	colorlinks=true,
	linkcolor=blue,
	anchorcolor=blue,
	citecolor=blue}
\usepackage{aligned-overset}
\usepackage{yhmath}

\usepackage{tikz}
\usepackage{scalerel}
\usepackage{tkz-euclide}
\usetikzlibrary{patterns,arrows.meta}

\usepackage{pgfplots}

\usepackage{cite}

\numberwithin{equation}{section}
\numberwithin{figure}{section}

\graphicspath{{./figure/}}
\DeclareGraphicsExtensions{.png}

\theoremstyle{definition}
\newtheorem{definition}{Definition}[section]
\newtheorem{proposition}[definition]{Proposition}

\newtheorem{lemma}[definition]{Lemma}
\newtheorem{theorem}[definition]{Theorem}
\theoremstyle{remark}

\newcommand{\ovl}{\overline}
\allowdisplaybreaks

\title[Shock formation for 2D Isentropic Euler equations with self-similar variables]{Shock formation for 2D Isentropic Euler Equations\\with self-similar variables}
\author{Wenze Su}
\address{School of Mathematical Sciences, Fudan University, Shanghai, China. }
\email{19110180013@fudan.edu.cn}

\keywords{2D Isentropic Euler Equations; Shock Formation; Self-similar solution}
\subjclass[2010]{35Q31, 35L67, 35B44}%

\begin{document}
\maketitle

\begin{abstract}
	We study the 2D isentropic Euler equations with the ideal gas law. We exhibit a set of smooth initial data that give rise to shock formation at a single point near the planar symmetry. These solutions are associated with non-zero vorticity at the shock and have uniform-in-time 1/3-H\"older bound. Moreover, these point shocks are of self-similar type and share the same profile, which is a solution to the 2D self-similar Burgers equation. Our proof, following Buckmaster, Shkoller and Vicol \cite{buckmaster2022formation}, is based on the stable 2D self-similar Burgers profile and the modulation method. 
\end{abstract}

\section{Introduction}
The two-dimensional compressible isentropic Euler equations read
\begin{equation}
\left\{\begin{aligned}
	&\partial_\mathrm{t} \mathrm{\rho}+\nabla_\mathrm{x}\cdot(\rho u)=0\\
	&\partial_\mathrm{t}(\rho u)+\mathrm{div}_\mathrm{x}(\rho u\otimes u)+\nabla_\mathrm{x}p=0\\
	&p=\frac{1}{\gamma}\rho^\gamma,
\end{aligned}\right.
\label{original Euler equation}
\end{equation}
where $\mathrm{x}=(\mathrm{x}_1,\mathrm{x}_2)\in\mathbb{R}^2$ and $\mathrm{t}\in\mathbb{R}$ are space and time coordinates respectively, the unknown scalar $\rho$ is the fluid density, $u=(u_1,u_2)$ is the velocity field of fluid, $p=\frac{1}{\gamma}\rho^\gamma$ is the pressure with adiabatic index $\gamma>1$. This system describes the evolution of a two-dimensional compressible ideal gas without viscosity. 

We define the vorticity $\omega=\partial_{\mathrm{x}_1}u_2-\partial_{\mathrm{x}_2}u_1$ and the specific vorticity $\zeta=\omega/\rho$ at those points where $\rho>0$. One can deduce from (\ref{original Euler equation}) that $\zeta$ is purely transported by the velocity field: 
\begin{equation}
	\partial_\mathrm{t}\zeta+u\cdot\nabla_{\mathrm{x}}\zeta=0.
\end{equation}

Our main result can be stated roughly as follows:
\begin{theorem}[Rough statement of the main theorem]
	\textit{There exists a set of initial data $(u_0,\rho_0)$ with $|\nabla(u_0,\rho_0)|=\mathcal{O}(1/\varepsilon)$, such that their corresponding solutions to (\ref{original Euler equation}) develop a shock-type singularity within time $\mathcal{O}(\varepsilon)$. }
\end{theorem}

It is well known that governing by compressible Euler equations, shock can develop from smooth initial data. In the one-dimensional case, this fact can be studying the dynamics of the Riemann variables, which were first introduced in Riemann's foundational work \cite{riemann1858uber}. See the discussion in John \cite{john1974formation}, Liu \cite{liu1979development}, and Majda \cite{majda2012compressible}. 

In multi-dimensional cases, Sideris \cite{sideris1985formation} first proved a blow-up result. However,  the shock formation remained open. In 2007, Christodoulou \cite{christodoulou2007formation} studied the relativistic fluids and he found an open set of irrotational initial data that will eventually develop shock-type singularity, this was considered to be the first proof of shock formation for the compressible Euler equations in multi-dimensional cases. Later in Christodoulou-Miao \cite{christodoulou2014compressible} the authors established the shock formation for non-relativistic and irrotational flow. In the case of irrotational flow, one can rewrite the isentropic Euler equations as a scalar second-order quasilinear wave equation. Alinhac \cite{alinhac1999blowup2,alinhac1999blowup} proved the first blow-up result for 2D quasilinear wave equations, which do not satisfy Klainerman's null condition. Using geometric method, shock formation for the 3D quasilinear wave equations was studied in Speck-Holzegel-Luk-Wong\cite{speck2016stable}, Speck\cite{speck2016shock}, Miao-Yu\cite{miao2017formation}. The first result on shock formation that admits non-zero vorticity for the compressible Euler equations was given by Luk-Speck \cite{luk2018shock}. They use the geometric framework and developed new methods to study the vorticity transport. Later in \cite{luk2021stability}, they proved shock formation for full compressible Euler equations in 3D with non-trivial vorticity and variable entropy. In An-Chen-Yin \cite{an2020low,an2021low,an2022cauchy,an2022h,an2022low}, the authors proved the low regularity ill-posedness for elastic waves and MHD equations and showed that the ill-posedness is driven by shock formation. As to the shock development problem for the compressible Euler equations, one could refer to the ‘discussions in Christodoulou-Lisibach\cite{christodoulou2016shock}, Christodoulou\cite{christodoulou2019shock},  Abbrescia-Speck\cite{abbrescia2022emergence} and  Buckmaster-Drivas-Shkoller-Vicol\cite{buckmaster2022simultaneous}. 

In \cite{buckmaster2019-2d-formation}, Buckmaster, Shkoller, and Vicol utilized the modulation method to construct shock solutions to the 2D Euler equations with azimuthal symmetry. Later in \cite{buckmaster2022formation}, they extend this method to the 3D case with no symmetry assumptions. After a dynamical rescaling, the solutions they constructed are close to a profile $\ovl W$, which solves the self-similar Burgers equation. By a singular coordinate transformation controlled by several modulation variables, proving shock formation is equivalent to showing global existence in the self-similar coordinate. This approach, known as the modulation method or dynamical rescaling, was successfully applied in \cite{merle1996asymptotics,merle2005blow,merle2020strongly} for the blow-up of Schr\"odinger equations and in \cite{merle1997stability} for the nonlinear heat equation. The proof in \cite{buckmaster2019-2d-formation} is $L^\infty$ based since there is no derivative in the forcing term, whereas in \cite{buckmaster2022formation}, an additional $L^2$ based energy estimate was used to overcome the derivative loss in the $L^\infty$-type argument. They also analyzed the non-isentropic case in \cite{buckmaster2020shock}. 

Following the work in \cite{buckmaster2022formation}, we utilize the self-similar Burgers ansatz to construct shock solutions. To keep track of the curvature of shock formation while maintaining the solution's stationarity in the far field, we make a minor modification to the construction in \cite{buckmaster2022formation}.  Different from the construction in \cite{buckmaster2019-2d-formation}, we consider shock solutions without any symmetry. 

The shock we attempt to construct is of self-similar type. We introduce a self-similar coordinate transformation $(\mathrm{t},\mathrm{x})\mapsto (s,y)$, where $(\mathrm{t},\mathrm{x})$ is the original Cartesian coordinate and $(s,y)$ is the self-similar coordinate. The new coordinate is aligned to the shock formation and will become singular when $t$ approaches the blow-up time $T_*$. Roughly speaking, we have that $$y_1\approx(T_*-t)^{-3/2}\mathrm{x}_1,\ \ y_2\approx(T_*-t)^{-1/2}\mathrm{x}_2.$$
Thus $y$ is a zoom-in version of $\mathrm{x}$. In self-similar coordinates, the Riemann invariant $W$(will be defined in the next subsection) will converge to a profile $\ovl W$, uniformly on any compact set of $y$. Moreover, $\ovl W$ solves the self-similar Burgers equation:
$$-\frac{1}{2}\overline{W}+\left(\frac{3}{2}y_1+\overline{W}\right)\partial_{y_1}\overline{W}+\frac{1}{2}y_2\partial_{y_2}\overline{W}=0.$$
In this sense, the constructed blow-up solution of the Euler equations is close to a fixed shape on a smaller and smaller scale. 

To better understand what happens, we shall examine the simplest 1D inviscid Burgers model, whose well-localized solutions are proved to become singular in finite time. It is pointed out explicitly in \cite{eggers2008role,collot2018singularity} that as we are approaching the blow-up point, the blow-up solution can be well modeled by a dynamically rescaled version of a fixed profile, which belongs to a countable family $\mathcal{F}$ of functions, and the members in $\mathcal{F}$ are solutions to the self-similar Burgers equation. The choice of profile only depends on the derivatives of initial data at the point that achieves the minimum negative slope. Thus the family $\mathcal{F}$ of solutions to the self-similar Burgers equation plays an important role in the blow-up phenomenon of the Burgers equation. For a detailed discussion, see \cite{collot2018singularity}, or the toy model in appendix \ref{universality}. 

After the asymptotic blow-up behavior of the inviscid Burgers equation was clarified systematically in \cite{collot2018singularity}, the self-similar Burgers profiles have been used to explore blow-up phenomena in various systems, see \cite{buckmaster2022formation,buckmaster2022-2d-unstableformation,yang2021shock,qiu2021shock,pasqualotto2022gradient}. The modulation method that was developed in the context of nonlinear dispersive equations, is the suitable way to utilize the self-similar Burgers profiles. 

\section{Preliminaries}
In this section we introduce a seris of coordinate transformations and the Riemann variables. 
\subsection{Coordinates adapted to the shock}
We introduce the sound speed $\sigma=\frac{1}{\alpha}\rho^\alpha$, where $\alpha=\frac{\gamma-1}{2}>0$, then the system of $(u,\rho)$ is transformed into a system of $(u,\sigma)$, which reads
\begin{equation}
	\left\{\begin{aligned}
		&\partial_\mathrm{t} \sigma+u\cdot\nabla_\mathrm{x}\sigma+\alpha\sigma\nabla_\mathrm{x}\cdot u=0\\
		&\partial_\mathrm{t}u+(u\cdot\nabla)u+\alpha\sigma\nabla_\mathrm{x}\sigma=0,\\
	\end{aligned}\right.
\end{equation}

By defining $t=\frac{1+\alpha}{2}\mathrm{t}$, the equations become
\begin{equation}
	\left\{\begin{aligned}
		&\frac{1+\alpha}{2}\partial_t \sigma+u\cdot\nabla_\mathrm{x}\sigma+\alpha\sigma\nabla_\mathrm{x}\cdot u=0\\
		&\frac{1+\alpha}{2}\partial_tu+(u\cdot\nabla)u+\alpha\sigma\nabla_\mathrm{x}\sigma=0,\\
	\end{aligned}\right.
\end{equation}

The vorticity is defined as
\begin{equation}
	\omega=\partial_{\mathrm{x}_1}u_2-\partial_{\mathrm{x}_2}u_1.
\end{equation}
We also introduce the specific vorticity $\zeta:=\omega/\rho$, which satisfies
\begin{equation}
	\frac{1+\alpha}{2}\partial_t\zeta+u\cdot\nabla_{\mathrm{x}}\zeta=0.
\end{equation}

To keep track of shock formation, we introduce six time-dependent modulation variables $\xi=(\xi_1,\xi_2)\in\mathbb{R}^2,n=(n_1,n_2)\in\mathbb{S}^1,\tau\in\mathbb{R},\phi\in\mathbb{R}$, and $\kappa\in\mathbb{R}$. $\xi$ records the location of the shock; $n$ records the direction of the shock; $\tau$ records the slope of the Riemann invariant $w$; $\phi$ measures the "curvature" of the shock; $\kappa$ records the value of the Riemann invariant at $\mathrm{x}=\xi(t)$.

Using modulation variables $\xi,n,\tau,\phi$, we define the coordinate adapted to the shock formation. 

\subsubsection{Tracing the location and direction of the shock}
With the time-dependent vector $\xi(t)=(\xi_1(t),\xi_2(t))$, and the normal vector $n(t)=(n_1(t),n_2(t))$, we define a coordinate transformation $\tilde x=R(t)^T(\mathrm{x}-\xi(t))$, where
\begin{equation}
	R(t)=\begin{bmatrix}
		n_1 & -n_2\\
		n_2 & n_1
	\end{bmatrix}\in SO(2)
\end{equation}
The origin of the $\tilde{x}$ coordinate coincides with $\xi(t)$, which dynamically tracks the spatial location of the shock formation, and $\tilde{e}_1$ aligns with $n(t)$, direction of the shock.

The functions should also be rewritten in the new coordinate: 
\begin{equation}
	\left\{\begin{aligned}
		&\tilde{u}(\tilde{x},t)=R(t)^Tu(\mathrm{x},t)\\
		&\tilde \rho(\tilde{x},t)=\rho(\mathrm{x},t)\\
		&\tilde \sigma(\tilde{x},t)=\sigma(\mathrm{x},t)\\
		&\tilde{\zeta}(\tilde{x},t)=\zeta(\mathrm{x},t).
	\end{aligned}\right.
\end{equation}
Then $(\tilde{u},\tilde{\sigma})$ satisfies
\begin{equation}
	\left\{\begin{aligned}
		&\frac{1+\alpha}{2}\partial_t \tilde{\sigma}+(\tilde{u}+\tilde{v})\cdot\nabla_{\tilde{x}}\tilde{\sigma}+\alpha\tilde{\sigma}\nabla_{\tilde{x}}\cdot\tilde{u}=0\\
		&\frac{1+\alpha}{2}\partial_t\tilde{u}-\frac{1+\alpha}{2}Q\tilde{u}+\left[(\tilde{u}+\tilde{v})\cdot\nabla_{\tilde{x}}\right]\tilde{u}+\alpha\tilde{\sigma}\nabla_{\tilde{x}}\tilde{\sigma}=0,
	\end{aligned}\right.
\end{equation}
where $Q(t)=\frac{dR(t)^T}{dt}R(t)=\dot R(t)^TR(t)$,  and $\tilde{v}(\tilde{x},t)=\frac{1+\alpha}{2}(Q\tilde{x}-R^T\dot\xi)$. 

The equation of specific vorticity is transformed into
\begin{equation}
	\frac{1+\alpha}{2}\partial_t \tilde{\zeta}+(\tilde{u}+\tilde{v})\cdot\nabla_{\tilde{x}}\tilde{\zeta}=0.
	\label{evolution of specific vorticity in tilde x coordinate}
\end{equation}

\subsubsection{Tracking the curvature of shock front}
In order to track the curvature of the shock, we introduce a time-dependent scalar function $\tilde f(\tilde{x}_1,\tilde{x}_2,t)$. 

We denote $\phi(t)\in\mathbb{R}$ as the ``curvature" of the ``wavefront of the shock formation" at the origin, and we assume that $\tilde f$ satisfies $\partial_{\tilde{x}_2}^2\tilde f(0,0,t)=\phi(t)$. 
In particular, we construct $\tilde f$ as follows. Let $\theta\in C_c^\infty(-\frac{5}{4},\frac{5}{4})$ be a bump function such that $\theta(\tilde x_2)\equiv1$ when $|\tilde x_2|\le1$, then we define
\begin{equation}
	\tilde f(\tilde{x}_1,\tilde{x}_2,t)=\theta(\varepsilon^{-\frac{1}{2}}\tilde{x}_1)\int_0^{\tilde x_2}\phi(t) \tilde x_2'\theta(\varepsilon^{-\frac{1}{6}}\tilde x_2')d\tilde x_2',
	\label{construction of f}
\end{equation}
where $\varepsilon$ is a small constant to be specified. Note that $\tilde{f}(\tilde{x}_1,\tilde{x}_2,t)=\frac{1}{2}\phi\tilde{x}_2^2$ when $|\tilde{x}|$ is small. This gurantees that in the forcing terms of $W,Z,A$ (to be defined in (\ref{definition of W,Z,A})) those related to the coordinate transformation vanish when $y$ is far from the origin, while not affecting the computation near the origin.  

Now we introduce the coordinate transformation that adapted to the shock front:
\begin{equation}
	\left\{
	\begin{aligned}
		&x_1=\tilde{x}_1-\tilde f(\tilde x_1,\tilde{x}_2,t)\\
		&x_2=\tilde{x}_2.
	\end{aligned}\right.
\end{equation}
Let $f(x_1,x_2,t):=\tilde f(\tilde x_1,\tilde{x}_2,t)$, then we have
\begin{equation}
	\left\{
	\begin{aligned}
		&\tilde x_1=x_1+f(x_1,x_2,t)\\
		&\tilde x_2=x_2.
	\end{aligned}\right.
\end{equation}

We define
\begin{equation}
	J(\tilde x_1,\tilde x_2,t)=|\nabla_{\tilde x}x_1|=\sqrt{(1-\tilde f_{\tilde x_1})^2+\tilde f_{\tilde x_2}^2}=\frac{\sqrt{1+f_{x_2}^2}}{1+f_{x_1}},
\end{equation}
\begin{equation}
	N=J^{-1}\nabla_{\tilde{x}}x_1=\frac{(1-\tilde f_{\tilde x_1},-\tilde f_{\tilde x_2})}{\sqrt{(1-\tilde f_{\tilde x_1})^2+\tilde f_{\tilde x_2}^2}}=\frac{1}{\sqrt{1+f_{x_2}^2}}(1,-f_{x_2}),
\end{equation}
\begin{equation}
	T=N^\perp=\frac{(\tilde f_{\tilde x_2},1-\tilde f_{\tilde x_1})}{\sqrt{(1-\tilde f_{\tilde x_1})^2+\tilde f_{\tilde x_2}^2}}=\frac{1}{\sqrt{1+f_{x_2}^2}}(f_{x_2},1).
\end{equation}
Note that $\{N,T\}$ forms an orthonormal basis. 

$J,N,T$ can also be viewed as functions of $(x_1,x_2,t)$ and we overload their names for the sake of convenience. One can verify that
\begin{equation}
	\mathrm{supp}_x(N-\tilde e_1,T-\tilde e_2)\subset\left\{|x_1|\le\frac{3}{2}\varepsilon^{\frac{1}{2}},|x_2|\le\frac{3}{2}\varepsilon^{\frac{1}{6}}\right\}.
	\label{support of DN and DT}
\end{equation}

Now the functions are redefined as
$$\left\{\begin{aligned}
	&\mathring{u}(x,t)=\tilde{u}(\tilde{x},t)\\
	&\mathring{\rho}(x,t)=\tilde\rho(\tilde{x},t)\\
	&\mathring{\sigma}(x,t)=\tilde \sigma(\tilde{x},t)\\
	&\mathring{\zeta}(x,t)=\tilde{\zeta}(\tilde{x},t)\\
	&v(x,t)=\tilde{v}(\tilde{x},t),
\end{aligned}\right.$$
and the system can be written as
\begin{equation}
	\left\{\begin{aligned}
		&\partial_t\mathring{u}-Q\mathring{u}+\left[-\frac{\partial_t f}{1+f_{x_1}}+2\beta_1(\mathring{u}+v)\cdot JN\right]\partial_{x_1}\mathring{u}+2\beta_1(\mathring{u}_2+v_2)\partial_{x_2}\mathring{u}=-2\beta_3JN\mathring{\sigma}\partial_{x_1}\mathring{\sigma}-2\beta_3\mathring{\sigma}\partial_{x_2}\mathring{\sigma}\tilde{e}_2\\
		&\partial_t\mathring{\sigma}+\left[-\frac{\partial_t f}{1+f_{x_1}}+2\beta_1(\mathring{u}+v)\cdot JN\right]\partial_{x_1}\mathring{\sigma}+2\beta_1(\mathring{u}_2+v_2)\partial_{x_2}\mathring{\sigma}=-2\beta_3\mathring{\sigma} JN\cdot\partial_{x_1}\mathring{u}-2\beta_3\mathring{\sigma}\partial_{x_2}\mathring{u}_2,
	\end{aligned}
	\right.
\end{equation}
where 
\begin{equation}
	\beta_1=\frac{1}{1+\alpha},\  \beta_2=\frac{1-\alpha}{1+\alpha},\  \beta_3=\frac{\alpha}{1+\alpha}.
\end{equation}

We can also deduce he equation governing the evolution of $\mathring{\zeta}$:
\begin{equation}
	\partial_t\mathring{\zeta}+\left[-\frac{\partial_t f}{1+f_{x_1}}+2\beta_1(\mathring{u}+v)\cdot JN\right]\partial_{x_1}\mathring{\zeta}+2\beta_1(\mathring{u}_2+v_2)\partial_{x_2}\mathring{\zeta}=0.
\end{equation}

\subsubsection{Riemann variables}
We define the Riemann variables by
\begin{equation}
	\left\{\begin{aligned}
		&w(x,t)=\mathring{u}(x,t)\cdot N+\mathring{\sigma}(x,t)\\
		&z(x,t)=\mathring{u}(x,t)\cdot N-\mathring{\sigma}(x,t)\\
		&a(x,y)=\mathring{u}(x,t)\cdot T.
	\end{aligned}
	\right.
\end{equation}
Then the system of $(\mathring{u},\mathring{\sigma})$ can be rewritten in terms of $(w,z,a)$ as
\begin{equation}
	\begin{aligned}
		\partial_tw&+\left(-\frac{\partial_t f}{1+f_{x_1}}+2\beta_1v\cdot JN+Jw+\beta_2Jz\right)\partial_{x_1}w+\left(2\beta_1v_2+N_2w+\beta_2N_2z+2\beta_1aT_2\right)\partial_{x_2}w\\
		=&-2\beta_3\mathring{\sigma}\partial_{x_2}aT_2+aT\cdot(\partial_t)_x N+aQ_{ij}T_jN_i+2\beta_1(\mathring{u}\cdot NN_2+aT_2+v_2)aT\cdot\partial_{x_2}N\\
		&-2\beta_3\sigma(a\partial_{x_2}T_2+\mathring{u}\cdot N\partial_{x_2}N_2)-\left(-\frac{\partial_t f}{1+f_{x_1}}+2\beta_1v\cdot JN+2\beta_1J\mathring{u}\cdot N\right)a\partial_{x_1}T\cdot N
	\end{aligned}		
\end{equation}
\begin{equation}
\begin{aligned}
	\partial_tz&+\left(-\frac{\partial_t f}{1+f_{x_1}}+2\beta_1v\cdot JN+\beta_2Jw+Jz\right)\partial_{x_1}z+\left(2\beta_1v_2+\beta_2N_2w+N_2z+2\beta_1aT_2\right)\partial_{x_2}w\\
	=&2\beta_3\mathring{\sigma}\partial_{x_2}aT_2+aT\cdot(\partial_t)_x N+aQ_{ij}T_jN_i+2\beta_1(\mathring{u}\cdot NN_2+aT_2+v_2)aT\cdot\partial_{x_2}N\\
	&+2\beta_3\mathring{\sigma}(a\partial_{x_2}T_2+\mathring{u}\cdot N\partial_{x_2}N_2)-\left(-\frac{\partial_t f}{1+f_{x_1}}+2\beta_1v\cdot JN+2\beta_1J\mathring{u}\cdot N\right)a\partial_{x_1}T\cdot N
\end{aligned}
\end{equation}	
\begin{equation}
	\begin{aligned}
		\partial_ta&+\left(-\frac{\partial_t f}{1+f_{x_1}}+2\beta_1v\cdot JN+\beta_1Jw+\beta_1Jz\right)\partial_{x_1}a+2\beta_1\left(v_2+\frac{w+z}{2}N_2+aT_2\right)\partial_{x_2}a\\
		=&-2\beta_3\mathring{\sigma} T_2\partial_{x_2}\mathring{\sigma}+\mathring{u}\cdot T N\cdot(\partial_t)_x T+\mathring{u}\cdot NQ_{ij}N_jT_i\\
		&+2\beta_1(\mathring{u}\cdot NN_2+aT_2+v_2)\mathring{u}\cdot N N\cdot\partial_{x_2}T-\left(-\frac{\partial_t f}{1+f_{x_1}}+2\beta_1v\cdot JN+2\beta_1J\mathring{u}\cdot N\right)\mathring{u}\cdot N\partial_{x_1}N\cdot T.
	\end{aligned}
\end{equation}

\subsubsection{Self-similar transformation}
We introduce self-similar variables as follows
\begin{equation}
	\left\{\begin{aligned}
		&s(t)=-\log(\tau(t)-t)\\
		&y_1=\frac{x_1}{(\tau-t)^{3/2}}=x_1e^{\frac{3}{2}s}\\
		&y_2=\frac{x_2}{(\tau-t)^{1/2}}=x_2e^{\frac{s}{2}},
	\end{aligned}\right.
\end{equation}
where $\tau(t)$ is a parameter to be determined. 

Now the original time $t$ is transformed into the self-similar time $\tau$, and the space variable $x$ is transformed into the self-similar space variable $y$. At each fixed time $t$, $y$ is a dilation of $x$. In the $y$ coordinate, we can closely observe the behavior of the solution around the shock location. 

Now we assume that
\begin{equation}
	\left\{\begin{aligned}
		&w(x,t)=e^{-\frac{s}{2}}W(y,s)+\kappa(t)\\
		&z(x,t)=Z(y,s)\\
		&a(x,t)=A(y,s),
	\end{aligned}\right.
	\label{definition of W,Z,A}
\end{equation}
where $\kappa$ is also a modulation parameter to be determined. 

In the self-similar variables, the system becomes
\begin{equation}
	\left\{\begin{aligned}
		&\left(\partial_s-\frac{1}{2}\right)W+\left(\frac{3}{2}y_1+g_W\right)\partial_1W+\left(\frac{1}{2}y_2+h_W\right)\partial_2W=F_W\\
		&\partial_sZ+\left(\frac{3}{2}y_1+g_Z\right)\partial_1Z+\left(\frac{1}{2}y_2+h_Z\right)\partial_2Z=F_Z\\
		&\partial_sA+\left(\frac{3}{2}y_1+g_A\right)\partial_1A+\left(\frac{1}{2}y_2+h_A\right)\partial_2A=F_A.
	\end{aligned}\right.
	\label{evolution of W,Z,A}
\end{equation}
Here and throughout the papar we use the notation $\partial_j=\partial_{y_j}$, and $\beta_\tau:=\frac{1}{1-\dot\tau}$. The transport terms and and the forcing terms are given by
\begin{equation}
	\left\{\begin{aligned}
		&g_W=\beta_\tau JW+\beta_\tau e^{\frac{s}{2}}\left[-\frac{\partial_t f}{1+f_{x_1}}+J\left(\kappa+\beta_2Z+2\beta_1V\cdot N\right)\right]=\beta_\tau JW+G_W\\
		&g_Z=\beta_2\beta_\tau JW+\beta_\tau e^{\frac{s}{2}}\left[-\frac{\partial_t f}{1+f_{x_1}}+J\left(\beta_2\kappa+Z+2\beta_1V\cdot N\right)\right]=\beta_2\beta_\tau JW+G_Z\\
		&g_A=\beta_1\beta_\tau JW+\beta_\tau e^{\frac{s}{2}}\left[-\frac{\partial_t f}{1+f_{x_1}}+J\left(\beta_1\kappa+\beta_1Z+2\beta_1V\cdot N\right)\right]=\beta_1\beta_\tau JW+G_A,
	\end{aligned}\right.
	\label{transport terms of W,Z,A}
\end{equation}
\begin{equation}
	\left\{\begin{aligned}
		&h_W=\beta_\tau e^{-s}N_2W+\beta_\tau e^{-\frac{s}{2}}\left(2\beta_1V_2+N_2\kappa+\beta_2N_2Z+2\beta_1AT_2\right)\\
		&h_Z=\beta_2\beta_\tau e^{-s}N_2W+\beta_\tau e^{-\frac{s}{2}}\left(2\beta_1V_2+\beta_2N_2\kappa+N_2Z+2\beta_1AT_2\right)\\
		&h_A=\beta_1\beta_\tau e^{-s}N_2W+\beta_\tau e^{-\frac{s}{2}}\left(2\beta_1V_2+\beta_1N_2\kappa+\beta_1N_2Z+2\beta_1AT_2\right),
	\end{aligned}\right.
\end{equation}
and
\begin{equation}
	\left\{\begin{aligned}
		F_W=&-2\beta_3\beta_\tau S\partial_2AT_2+\beta_\tau e^{-\frac{s}{2}}AT\cdot(\partial_t)_xN+\beta_\tau e^{-\frac{s}{2}}Q_{ij}AT_jN_i\\
		&+2\beta_1\beta_\tau \left(V_2+U\cdot NN_2+AT_2\right)AT\cdot\partial_2N-2\beta_3\beta_\tau S(U\cdot N\partial_2N_2+A\partial_2T_2)\\
		&-\beta_\tau e^{s}\left(-\frac{\partial_tf}{1+f_{x_1}}+2\beta_1V\cdot JN+2\beta_1JU\cdot N\right)A\partial_1T\cdot N-\beta_\tau e^{-\frac{s}{2}}\dot{\kappa}\\
		F_Z=&2\beta_3\beta_\tau e^{-\frac{s}{2}}S\partial_2AT_2+\beta\tau e^{-s}AT\cdot(\partial_t)_xN+\beta e^{-s}Q_{ij}AT_jN_i\\
		&+2\beta_1\beta_\tau e^{-\frac{s}{2}}(V_2+U\cdot NN_2+AT_2)AT\cdot\partial_2N+2\beta_3\beta_\tau e^{-\frac{s}{2}}(A\partial_2T_2+U\cdot N\partial_2N_2)\\
		&-\beta_\tau e^{\frac{s}{2}}\left(-\frac{\partial_tf}{1+f_{x_1}}+2\beta_1V\cdot JN+2\beta_1JU\cdot N\right)A\partial_1T\cdot N\\
		F_A=&-2\beta_3\beta_\tau e^{-\frac{s}{2}}ST_2\partial_2S+\beta_\tau e^{-s}U\cdot NN\cdot(\partial_t)_xT+\beta_\tau e^{-s}Q_{ij}(U\cdot NN_j+AT_j)T_i\\
		&+2\beta_1\beta_\tau e^{-\frac{s}{2}}(V_2+U\cdot NN_2+AT_2)U\cdot NN\cdot\partial_2T\\
		&-\beta_\tau e^{\frac{s}{2}}\left(-\frac{\partial_tf}{1+f_{x_1}}+2\beta_1V\cdot JN+2\beta_1JU\cdot N\right)U\cdot N\partial_1N\cdot T,
	\end{aligned}\right.
\label{forcing terms of W,Z,A}
\end{equation}
where $U$, $V$, $S$ are the self-similar versions of $\mathring{u}$, $v$, $\mathring{\sigma}$, for example $S(y,s)=\mathring{\sigma}(x,t)$. 

If we write the transport terms as 
\begin{equation}
	\left\{\begin{aligned}
		&\mathcal{V}_W=\left(\frac{3}{2}y_1+g_W,\frac{1}{2}y_2+h_W\right)\\
		&\mathcal{V}_Z=\left(\frac{3}{2}y_1+g_Z,\frac{1}{2}y_2+h_Z\right)\\
		&\mathcal{V}_A=\left(\frac{3}{2}y_1+g_A,\frac{1}{2}y_2+h_A\right),
	\end{aligned}\right.
\end{equation}
then the equation of $(W,Z,A)$ can be written in a compact form:
\begin{equation}
	\left\{\begin{aligned}
		&\partial_sW-\frac{1}{2}W+\mathcal{V}_W\cdot\nabla W=F_W\\
		&\partial_sZ+\mathcal{V}_Z\cdot\nabla Z=F_Z\\
		&\partial_sA+\mathcal{V}_A\cdot\nabla A=F_A.
	\end{aligned}
	\right.
\end{equation}

We also deduce the equations of $(U,S)$: 
\begin{equation}
	\left\{\begin{aligned}
		&\partial_sU_i-\beta_\tau e^{-s}Q_{ij}U_j+\mathcal{V}_A\cdot\nabla U=-2\beta_3\beta_\tau e^{\frac{s}{2}}S\partial_1SJN_i-2\beta_3\beta_\tau e^{-\frac{s}{2}}S\partial_2S\delta_{i2}\\
		&\partial_sS+\mathcal{V}_A\cdot\nabla S=-2\beta_3\beta_\tau e^{\frac{s}{2}}S\partial_1U\cdot JN-2\beta_3\beta_\tau e^{-\frac{s}{2}}S\partial_2U_2.
	\end{aligned}
	\right.
	\label{equation of U,S}
\end{equation}
We can see that $(U,S)$ are transported in the same way as $A$. The transport terms $g_A$, $h_A$ in the equation of $A$ can also be expressed in terms of $U$, $S$: 
\begin{equation}
		\left\{
	\begin{aligned}
		&g_A=\beta_\tau e^{\frac{s}{2}}\left[2\beta_1(U+V)\cdot JN-\frac{\partial_t f}{1+f_{x_1}}\right]\\
		&h_A=2\beta_1\beta_\tau e^{-\frac{s}{2}}(U_2+V_2).
	\end{aligned}\right.
\end{equation}

Here we record the relation between $(U,S)$ and $(W,Z,A)$:
\begin{equation}
	\left\{
	\begin{aligned}
		&U=\frac{1}{2}\left(e^{-\frac{s}{2}}W+Z+\kappa\right)N+AT\\
		&S=\frac{1}{2}\left(e^{-\frac{s}{2}}W-Z+\kappa\right),
	\end{aligned}\right.
\label{U,S in terms of W,Z,A}
\end{equation}
and
\begin{equation}
	\left\{
	\begin{aligned}
		&W=e^{\frac{s}{2}}(U\cdot N+S-\kappa)\\
		&Z=U\cdot N-S\\
		&A=U\cdot T.
	\end{aligned}\right.
\label{W,Z,A in terms of U,S}
\end{equation}

Although we introduce the self-similar version of functions like $V(y,s)$ of $v(x,t)$, we overload the functions $f$, $J$, $N$, $T$ as functions of $(y,s)$. For example, in the self-similar coordinates, we view $N$ as the map $y\mapsto N(x(y),t(s))$, and $\partial_2N(y)=\partial_{y_2}[N(x(y),t(s))]$. 

\subsection{Self-similar 2D Burgers profile}
We first introduce the 1D self-similar Burgers profile
\begin{equation}
	W_{1d}(y_1)=\left(-\frac{y_1}{2}+\left(\frac{1}{27}+\frac{y_1^2}{4}\right)^{\frac{1}{2}}\right)^{\frac{1}{3}}-\left(\frac{y_1}{2}+\left(\frac{1}{27}+\frac{y_1^2}{4}\right)^{\frac{1}{2}}\right)^{\frac{1}{3}},
\end{equation}
which solves the 1D self-similar Burgers equation(cf. \cite{collot2018singularity}): 
\begin{equation}
	-\frac{1}{2}W_{1d}+\left(\frac{3}{2}y_1+W_{1d}\right)\partial_{y_1}W_{1d}=0.
\end{equation}

Moreover, we introduce
\begin{equation}
	\overline{W}(y_1,y_2)=\langle y_2\rangle W_{1d}(\langle y_2\rangle^{-3}y_1),
\end{equation}
where $\langle y_2\rangle=\sqrt{1+y_2^2}$. One an verify that $\overline{W}$ is a solution to the 2D self-similar Burgers equation:
\begin{equation}
	-\frac{1}{2}\overline{W}+\left(\frac{3}{2}y_1+\overline{W}\right)\partial_{y_1}\overline{W}+\frac{1}{2}y_2\partial_{y_2}\overline{W}=0.
\end{equation}

\subsubsection{Properties of $\overline{W}$}
It can be checked via the explicit formula of $W_{1d}$ that
\begin{equation}
	\left|W_{1d}(y_1)\right|\le\min\left(|y_1|,\frac{|y_1|}{\frac{1}{3}+|y_1|^{\frac{2}{3}}}\right)\le\min\left(|y_1|,|y_1|^{\frac{1}{3}}\right),
\end{equation}
\begin{equation}
	|W_{1d}'(y_1)|\le\langle y_1\rangle^{-\frac{2}{3}},\ |W_{1d}''(y_1)|\le\langle y_1\rangle^{-\frac{5}{3}},
\end{equation}
\begin{equation}
	\left|W_{1d}(y_1)W_{1d}'(y_1)\right|\le\frac{1}{3}\langle y_1\rangle^{-\frac{1}{3}},\ \left|(W_{1d}W_{1d}')'(y_1)\right|\le\min\left(\langle y_1\rangle^{-\frac{4}{3}},\frac{1}{7}|y_1|^{-1}\langle y_1\rangle^{-\frac{1}{3}}\right).
\end{equation}

Define $\eta(y)=1+y_1^2+y_2^6$, $\tilde{\eta}(y)=1+|y|^2+y_2^6$, then the above inequalities imply that
\begin{equation}
	\left|\overline{W}\right|\le(1+y_1^2)^{\frac{1}{6}}\le\eta^{\frac{1}{6}},
	\label{estimate of ovl W}
\end{equation}
\begin{equation}
	\left|\partial_1\ovl{W}\right|\le\tilde{\eta}^{-\frac{1}{3}}, \ \left|\partial_2\ovl{W}\right|\le\frac{2}{3},
	\label{estimates of D ovl W}
\end{equation}
\begin{equation}
	\left|\partial_{11}\ovl{W}\right|\le\tilde\eta^{-\frac{5}{6}},\ 
	\left|\partial_{12}\ovl{W}\right|\le2\eta^{-\frac{1}{2}},\ \left|\partial_{22}\ovl{W}\right|\le\frac{6}{7}\eta^{-\frac{1}{6}}.
	\label{estimates of D^2 ovl W}
\end{equation}
At the origin we can check by the expression of $\ovl{W}$ that
\begin{equation}
	\ovl{W}(0)=0,\ \ \ \nabla\ovl{W}(0)=
	\begin{pmatrix}
		-1\\0
	\end{pmatrix},\ \ \ 
	\nabla^2\ovl{W}(0)=
	\begin{pmatrix}
		0&0\\
		0&0
	\end{pmatrix},\ \ \ 
	\partial_1\nabla^2\ovl{W}(0)=
	\begin{pmatrix}
		6&0\\
		0&2
	\end{pmatrix}.
	\label{evaluation of ovl W at 0}
\end{equation}

\subsection{Evolution of $\tilde W$ and higher order derivatives of the unknowns}
If we define $\widetilde{W}=W-\ovl{W}$, then $\widetilde{W}$ satisfies
\begin{equation}
	\left(\partial_s-\frac{1}{2}+\beta_\tau J\partial_1\ovl{W}\right)\widetilde{W}+\mathcal{V}_W\cdot\nabla\widetilde{W}=\widetilde{F}_W,
	\label{equation of tilde W}
\end{equation}
where
\begin{equation}
	\widetilde{F}_W=F_W+\left[(1-\beta_\tau J)\ovl{W}-G_W\right]\partial_1\ovl{W}-h_W\partial_2\ovl{W}.
\end{equation}

For a multi-index $\gamma=(\gamma_1,\gamma_2)$ satisfying $|\gamma|\ge1$, we have the evolution equation for $(\partial^\gamma W,\partial^\gamma Z,\partial^\gamma A)$:
\begin{equation}
	\left\{\begin{aligned}
		&\left(\partial_s+\frac{3\gamma_1+\gamma_2-1}{2}+\beta_\tau(1+\gamma_1\mathbbm{1}_{\gamma_1\ge2}J\partial_1W)\right)\partial^\gamma W+\mathcal{V}_W\cdot\nabla\partial^\gamma W=F_W^{(\gamma)}\\
		&\left(\partial_s+\frac{3\gamma_1+\gamma_2}{2}+\beta_2\beta_\tau\gamma_1J\partial_1W\right)\partial^\gamma Z+\mathcal{V}_Z\cdot\nabla\partial^\gamma Z=F_Z^{(\gamma)}\\
		&\left(\partial_s+\frac{3\gamma_1+\gamma_2}{2}+\beta_2\beta_\tau\gamma_1J\partial_1W\right)\partial^\gamma A+\mathcal{V}_A\cdot\nabla\partial^\gamma A=F_A^{(\gamma)},
	\end{aligned}
	\right.
\end{equation}
where the forcing terms are
\begin{equation}
	\begin{aligned}
		F_W^{(\gamma)}=&\partial^\gamma F_W-\beta_\tau\partial_1W[\partial^\gamma,J]W-\beta_\tau\mathbbm{1}_{|\gamma|\ge2}\sum_{\substack{|\beta|=|\gamma|-1\\\beta_1=\gamma_1}}\binom{\gamma}{
			\beta}\partial^{\gamma-\beta}(JW)\partial_1\partial^\beta W\\
		&-\beta_\tau\mathbbm{1}_{|\gamma|\ge3}\sum_{\substack{1\le|\beta|\le|\gamma|-1\\\beta\le\gamma}}\binom{\gamma}{\beta}\partial^{\gamma-\beta}(JW)\partial_1\partial^\beta W-\sum_{0\le\beta<\gamma}\binom{\gamma}{\beta}\left(\partial^{\gamma-\beta}G_W\partial_1\partial^\beta W+\partial^{\gamma-\beta}h_W\partial_2\partial^\beta W\right),
	\end{aligned}
	\label{forcing terms of derivatives of W}
\end{equation}
\begin{equation}
	\begin{aligned}
		F_Z^{(\gamma)}=&\partial^\gamma F_Z-\beta_2\beta_\tau\sum_{\substack{|\beta|=|\gamma|-1\\\beta_1=\gamma_1}}\binom{\gamma}{\beta}\partial^{\gamma-\beta}(JW)\partial_1\partial^\beta Z\\
		&-\beta_2\beta_\tau\mathbbm{1}_{|\gamma|\ge2}\sum_{\substack{0\le|\beta|\le|\gamma|-2\\\beta\le\gamma}}\binom{\gamma}{\beta}\partial^{\gamma-\beta}(JW)\partial_1\partial^\beta Z-\sum_{0\le\beta<\gamma}\binom{\gamma}{\beta}\left(\partial^{\gamma-\beta}G_Z\partial_1\partial^\beta Z+\partial^{\gamma-\beta}h_Z\partial_2\partial^\beta Z\right),
	\end{aligned}
	\label{forcing terms of derivative of Z}
\end{equation}
\begin{equation}
	\begin{aligned}
		F_A^{(\gamma)}=&\partial^\gamma F_A-\beta_2\beta_\tau\sum_{\substack{|\beta|=|\gamma|-1\\\beta_1=\gamma_1}}\binom{\gamma}{\beta}\partial^{\gamma-\beta}(JW)\partial_1\partial^\beta A\\
		&-\beta_2\beta_\tau\mathbbm{1}_{|\gamma|\ge2}\sum_{\substack{0\le|\beta|\le|\gamma|-2\\\beta\le\gamma}}\binom{\gamma}{\beta}\partial^{\gamma-\beta}(JW)\partial_1\partial^\beta A-\sum_{0\le\beta<\gamma}\binom{\gamma}{\beta}\left(\partial^{\gamma-\beta}G_A\partial_1\partial^\beta A+\partial^{\gamma-\beta}h_A\partial_2\partial^\beta A\right).
	\end{aligned}
	\label{forcing terms of derivatives of A}
\end{equation}

Similarly we can deduce the equation of $\partial^\gamma\widetilde{W}$:
\begin{equation}
	\left[\partial_s+\frac{3\gamma_1+\gamma_2-1}{2}+\beta_\tau J(\partial_1\ovl{W}+\gamma_1\partial_1W)\right]\partial^\gamma\widetilde{W}+\mathcal{V}_W\cdot\nabla\partial^\gamma\widetilde{W}=\widetilde{F}_W^{(\gamma)},
	\label{evolution of derivatives of tilde W}
\end{equation}
where
\begin{equation}
	\begin{aligned}
		\widetilde{F}_W^{(\gamma)}=&\partial^\gamma\widetilde{F}_W-\sum_{0\le\beta<\gamma}\binom{\gamma}{\beta}\left[\partial^{\gamma-\beta}G_W\partial_1\partial^\beta \widetilde{W}+\partial^{\gamma-\beta}h_W\partial_2\partial^\beta \widetilde{W}+\beta_\tau\partial^{\gamma-\beta}(J\partial_1\ovl{W})\partial^\beta \widetilde{W}\right]\\
		&-\beta_\tau\gamma_2\partial_2(JW)\partial_1^{\gamma_1+1}\partial_2^{\gamma_2-1}\widetilde{W}-\beta_\tau\mathbbm{1}_{|\gamma|\ge2}\sum_{\substack{0\le|\beta|\le|\gamma|-2\\\beta\le\gamma}}\binom{\gamma}{\beta}\partial^{\gamma-\beta}(JW)\partial_1\partial^\beta\widetilde{W}.
	\end{aligned}
\end{equation}

\section{Main result}
In this section, we state the initial condition and the main shock formation result of the 2D compressible Euler equations. The proof of the main theorem will be given in section \ref{proof of the main theorem}. 
\subsection{Initial data in physical variables}\label{Initial data in physical variables}
We assume that the initial time is $t=-\varepsilon$ with $\varepsilon$ to be determined. 
For modulation variables, we assume that
\begin{equation}
	\kappa(-\varepsilon)=\kappa_0,\ \ \xi(-\varepsilon)=0,\ \ n_2(-\varepsilon)=0,\ \ \tau(-\varepsilon)=0,\ \ \ \phi(-\varepsilon)=\phi_0=0,
	\label{initial condition of modulation variables}
\end{equation}
where
\begin{equation}
	\kappa_0\ge\frac{3}{1-\max(\beta_1,\beta_2)}.
	\label{initial condition for kappa}
\end{equation}
Since $n_2(-\varepsilon)=0$ and $\xi(-\varepsilon)=0$, $\mathrm{x}$-coordinate and $\tilde{x}$-coordinate coincide at $t=-\varepsilon$, and
\begin{equation}
	\left\{
	\begin{aligned}
		x_1&=\tilde{f}(\mathrm{x}_1,\mathrm{x}_2,-\varepsilon)\\
		x_2&=\mathrm{x}_2.
	\end{aligned}
	\right.
\end{equation}
Now we prescribe the initial data:
\begin{equation}
	u_0(\mathrm{x}):=u(\mathrm{x},-\varepsilon),\ \ \ \rho_0(\mathrm{x}):=\rho(\mathrm{x},-\varepsilon),\ \ \ \sigma_0:=\frac{\rho_0^\alpha}{\alpha}.
\end{equation}
We choose $u_0$ and $\rho_0$ such that the corresponding Riemann varibles satisfy the conditions stated in this section. The initial data of the Riemann variables are denoted as
\begin{equation}
	\begin{aligned}
		\widetilde{w}_0(\mathrm{x}):&=u_0(\mathrm{x})\cdot N(\mathrm{x},-\varepsilon)+\sigma_0(\mathrm{x})=:w_0(x)\\
		\widetilde{z}_0(\mathrm{x}):&=u_0(\mathrm{x})\cdot N(\mathrm{x},-\varepsilon)-\sigma_0(\mathrm{x})=:z_0(x)\\
		\widetilde{a}_0(\mathrm{x}):&=u_0(\mathrm{x})\cdot T(\mathrm{x},-\varepsilon)=:w_0(x).
	\end{aligned}
\end{equation}

First we assume that
\begin{equation}
	\mathrm{supp}_{\mathrm{x}}\ (\widetilde{w}_0-\kappa_0,\widetilde{z}_0,\widetilde{a}_0)\subset{\scriptstyle\mathcal{X}}_0:=\left\{|\mathrm{x}_1|\le\frac{1}{2}\varepsilon^{\frac{1}{2}},|\mathrm{x}_2|\le\varepsilon^{\frac{1}{6}}\right\}.
	\label{spatial support of initial data in physical variable, tilde x}
\end{equation}
This implies that
\begin{equation}
	\mathrm{supp}_{x}\ (w_0-\kappa_0,z_0,a_0)\subset\left\{|x_1|\le\varepsilon^{\frac{1}{2}},|x_2|\le\varepsilon^{\frac{1}{6}}\right\}
	\label{spatial support of initial data in physical variable}.
\end{equation}

The function $\widetilde{w}_0(\mathrm{x})$ is chosen such that 
\begin{equation}
	\begin{aligned}
		&\text{the minimum negative slope of } \widetilde{w}_0\text{ occurs in the $\mathrm{x}_1$ direction, }\\
		&\partial_{\mathrm{x}_1}\widetilde{w}_0\text{ attains its global minimum at }\mathrm{x}=0.
	\end{aligned}
	\label{initial slope condition for w, 1}
\end{equation}
and
\begin{equation}
	\nabla_{\mathrm{x}}\partial_{\mathrm{x}_1}\widetilde{w}_0(0)=0.
	\label{initial slope condition for w, 2}
\end{equation}
We also assume that
\begin{equation}
	\widetilde{w}_0(0)=\kappa_0,\ \ \ \partial_{\mathrm{x}_1}\widetilde{w}_0(0)=-\frac{1}{\varepsilon},\ \ \ \partial_{\mathrm{x}_2}\widetilde{w}_0(0)=0,
	\label{initial value of w at 0}
\end{equation}
Define 
\begin{equation}
	\overline{w}_\varepsilon(x):=\varepsilon^{\frac{1}{2}}\overline{W}(\varepsilon^{-\frac{3}{2}}x_1,\varepsilon^{-\frac{1}{2}}x_2),
\end{equation}
and we set
\begin{equation}
	\wideparen{w}_0(\mathrm{x}):=\widetilde{w}_0(\mathrm{x})-\overline{w}_\varepsilon(\mathrm{x}_1-\tilde{f}(\mathrm{x},t),\mathrm{x}_2)=w_0(x)-\overline{w}_\varepsilon(x)=\varepsilon^{\frac{1}{2}}\widetilde{W}(y,-\log\varepsilon)+\kappa_0.
\end{equation}

We assume that for $\mathrm{x}$ such that $\left|(\varepsilon^{-\frac{3}{2}}\mathrm{x}_1,\varepsilon^{-\frac{1}{2}}\mathrm{x}_2)\right|\le2\varepsilon^{-\frac{1}{10}}$, the following bounds hold:
\begin{equation}
	\begin{aligned}
		|\wideparen{w}_0(\mathrm{x})-\kappa_0|&\le\varepsilon^{\frac{1}{10}}\left(\varepsilon^3+\mathrm{x}_1^2+\mathrm{x}_2^6\right)^\frac{1}{6},\\
		|\partial_{\mathrm{x}_1}\wideparen{w}_0(\mathrm{x})|&\le\varepsilon^{\frac{1}{11}}\left(\varepsilon^3+\mathrm{x}_1^2+\mathrm{x}_2^6\right)^{-\frac{1}{3}},\\
		|\partial_{\mathrm{x}_2}\wideparen{w}_0(\mathrm{x})|&\le\frac{1}{2}\varepsilon^{\frac{1}{12}}
		\label{initial data for difference in physical variable, 0th and 1st order}.
	\end{aligned}
\end{equation}
For $\mathrm{x}$ such that $\left|(\varepsilon^{-\frac{3}{2}}\mathrm{x}_1,\varepsilon^{-\frac{1}{2}}\mathrm{x}_2)\right|\le1$, we assume that
\begin{equation}
	|\partial_{\mathrm{x}}^\gamma\wideparen{w}_0(\mathrm{x})|\overset{|\gamma|=4}{\le}\frac{1}{2}\varepsilon^{\frac{5}{8}-\frac{1}{2}(3\gamma_1+\gamma_2)}.
	\label{initial data for difference in physical variable, 4th order}
\end{equation}
At $\mathrm{x}=0$, we assume that
\begin{equation}
	|\partial_{\mathrm{x}}^\gamma\wideparen{w}_0(\mathrm{0})|\overset{|\gamma|=3}{\le}\frac{1}{2}\varepsilon^{1-\frac{1}{2}(3\gamma_1+\gamma_2)-\frac{4}{2k-7}}.
	\label{initial data for difference in physical variable, 3rd order}
\end{equation}
For $\mathrm{x}\in{\scriptstyle\mathcal{X}}_0$ such that $\left|(\varepsilon^{-\frac{3}{2}}\mathrm{x}_1,\varepsilon^{-\frac{1}{2}}\mathrm{x}_2)\right|\ge\frac{1}{2}\varepsilon^{-\frac{1}{10}}$, we assume that
\begin{equation}
	\begin{aligned}
		|\widetilde{w}_0(\mathrm{x})-\kappa_0|&\le(1+\varepsilon^{\frac{1}{11}})\left(\varepsilon^4+\mathrm{x}_1^2+\mathrm{x}_2^6\right)^{\frac{1}{6}},\\
		|\partial_{\mathrm{x}_1}\widetilde{w}_0(\mathrm{x})|&\le(1+\varepsilon^{\frac{1}{12}})\left(\varepsilon^4+\mathrm{x}_1^2+\mathrm{x}_2^6\right)^{-\frac{1}{3}},\\
		|\partial_{\mathrm{x}_2}\widetilde{w}_0(\mathrm{x})|&\le\frac{2}{3}+\varepsilon^{\frac{1}{13}}.
	\end{aligned}
	\label{initial data for w in physical variable, 0th and 1st order, outside}
\end{equation}
For all $\mathrm{x}\in{\scriptstyle\mathcal{X}}_0$, we assume that
\begin{equation}
	\begin{aligned}
		|\partial_{\mathrm{x}_1}^2\widetilde{w}_0(\mathrm{x})|&\le\varepsilon^{-\frac{3}{2}}\left(\varepsilon^3+\mathrm{x}_1^2+\mathrm{x}_2^6\right)^{-\frac{1}{3}},\\
		|\partial_{\mathrm{x}_1\mathrm{x}_2}\widetilde{w}_0(\mathrm{x})|&\le\frac{1}{2}\varepsilon^{-\frac{1}{2}}\left(\varepsilon^3+\mathrm{x}_1^2+\mathrm{x}_2^6\right)^{-\frac{1}{3}},\\
		|\partial_{\mathrm{x}_2}^2\widetilde{w}_0(\mathrm{x})|&\le\frac{1}{2}\left(\varepsilon^3+\mathrm{x}_1^2+\mathrm{x}_2^6\right)^{-\frac{1}{6}}.
	\end{aligned}
	\label{initial data for w in physical variable, 2nd order}
\end{equation}
Also att $\mathrm{x}=0$ we assume that
\begin{equation}
	\left|\partial_{\mathrm{x}_2}^2\widetilde{w}_0(0)\right|\le1.
	\label{initial data for w in physical variable, 2nd order at 0}
\end{equation}

For the initial data of $\widetilde{z}_0$ and $\widetilde{a}_0$ we assume that
\begin{equation}
	\begin{aligned}
		&|\widetilde{z}_0(\mathrm{x})|\le\varepsilon,\ \ \ |\partial_{\mathrm{x}_1}\widetilde{z}_0(\mathrm{x})|\le1,\ \ \ |\partial_{\mathrm{x}_2}\widetilde{z}_0(\mathrm{x})|\le\frac{1}{2}\varepsilon^{\frac{1}{2}},\\
		&|\partial_{\mathrm{x}_1}^2\widetilde{z}_0(\mathrm{x})|\le\varepsilon^{-\frac{3}{2}},\ \ \ |\partial_{\mathrm{x}_1\mathrm{x}_2}\widetilde{z}_0(\mathrm{x})|\le\frac{1}{2}\varepsilon^{-\frac{1}{2}},\ \ \ |\partial_{\mathrm{x}_2}^2\widetilde{z}_0(\mathrm{x})|\le\frac{1}{2},
	\end{aligned}
	\label{initial data for z in physical variable}
\end{equation}
and
\begin{equation}
	|\widetilde{a}_0(\mathrm{x})|\le\varepsilon,\ \ \ |\partial_{\mathrm{x}_1}\widetilde{a}_0(\mathrm{x})|\le1,\ \ \ |\partial_{\mathrm{x}_2}\widetilde{a}_0(\mathrm{x})|\le\frac{1}{2}\varepsilon^{\frac{1}{2}},\ \ \  |\partial_{\mathrm{x}_2}^2\widetilde{a}_0(\mathrm{x})|\le\frac{1}{2}.
	\label{initial data for a in physical variable}
\end{equation}
For the initial specific vorticity, we assume that
\begin{equation}
	\left\|\frac{\operatorname{curl}u_0(\mathrm{x})}{\rho_0(\mathrm{x})}\right\|_{L^\infty}\le1
	\label{initial data for vorticity in physical variable}.
\end{equation}
Finally for the Sobolev norm of the initial data we assume that for a fixed $k$ with $k\ge18$ the following holds:
\begin{equation}
	\sum_{\gamma=k}\varepsilon^2\|\partial_\mathrm{x}^\gamma \widetilde{w}_0\|_{L^2}^2+\|\partial_\mathrm{x}^\gamma \widetilde{z}_0\|_{L^2}^2+\|\partial_\mathrm{x}^\gamma \widetilde{a}_0\|_{L^2}^2\le\frac{1}{2}\varepsilon^{\frac{7}{2}-3\gamma_1-\gamma_2}.
	\label{initial sobolev norm in physical variable}
\end{equation}

\begin{theorem}[\emph{Main result in physical variables}]
	If 
	\begin{itemize}
		\item the initial values of the modulation variables satisfy (\ref{initial condition of modulation variables})(\ref{initial condition for kappa});
		\item the initial data $(u_0,\rho_0)$ of the Euler equations is smooth, and it gurantees that the corresponding Riemann variables $(w_0,z_0,a_0)$ satisfies the initial conditions (\ref{initial slope condition for w, 1})-(\ref{initial sobolev norm in physical variable}),
	\end{itemize}
	then the cooresponding solution $(u,\rho)$ to (\ref{original Euler equation}) blows up in finite time $-\varepsilon<T_*=O(\varepsilon^2)<+\infty$. Moreover, we have the following description of the shock: 
	\begin{enumerate}
		\item\emph{Blow-up speed}. We have the following inequalities for $(u,\sigma)$:
		\begin{equation}
			\frac{c}{T_*-t}\le\|\nabla_\mathrm{x} u(t)\|_{L^\infty}\le\frac{C}{T_*-t},
			\label{blow up speed for u}
		\end{equation}
		\begin{equation}
			\frac{c}{T_*-t}\le\|\nabla_\mathrm{x} \sigma(t)\|_{L^\infty}\le\frac{C}{T_*-t}.
			\label{blow up speed for sigma}
		\end{equation}
		\item\emph{Blow-up location}. For arbitrary $\delta\in(0,1)$, there holds that
		 \begin{equation}
		 	\|\nabla _\mathrm{x} u(t)\|_{L^\infty(B_\delta^c(\xi(t)))}+\|\nabla_\mathrm{x} \sigma(t)\|_{L^\infty(B_\delta^c(\xi(t)))}\le C(\delta),
		 	\label{boundedness of u and sigma away from origin}
		 \end{equation}
	 	while we have the unboundedness of gradient along $\xi(t)$:
	 	\begin{equation}
	 		|\nabla _\mathrm{x} u(\xi(t),t)|\ge\frac{c}{T_*-t},\ \ |\nabla _\mathrm{x} \sigma(\xi(t),t)|\ge\frac{c}{T_*-t}.
	 		\label{unboundednesss of u and sigma at the origin}
	 	\end{equation}
 		Moreover, the limit of $\xi(t)$ exists:
 		\begin{equation}
 			\lim_{t\rightarrow T_*}\xi(t)=\xi_*\in\mathbb{R}^2.
 		\end{equation}
 		\item\emph{Direction of the shock}. The gradient of $(u,\sigma)$ blows up only in one direction:
 		\begin{equation}
 			|[(R(t)N)\cdot\nabla _{\mathrm{x}}] u(\xi(t),t)|\ge\frac{c}{T_*-t},\ \ |(R(t)N)\cdot\nabla _{\mathrm{x}} \sigma(\xi(t),t)|\ge\frac{c}{T_*-t};
 			\label{unboundedness at the N direction}
 		\end{equation}
 		\begin{equation}
 			\|[(R(t)T)\cdot\nabla _{\mathrm{x}}] u(t)\|_{L^\infty}+ \|[(R(t)T)\cdot\nabla _{\mathrm{x}}] \sigma(t)\|_{L^\infty}\le C.
 			\label{boundedness at the T direction}
 		\end{equation}
 		Moreover, we have $n(t)=R(t)N(0,t)$, and the limit of $n(t)$ exists:
 		\begin{equation}
 			\lim_{t\rightarrow T_*}n(t)=n_*\in\mathbb{S}^1.
 			\label{limit of n}
 		\end{equation}
		\item \emph{1/3-H\"older continuity}. The solution has a uniform-in-time $C^{1/3}$ bound. More precisely, we have that
		\begin{equation}
			(u,\sigma)\in L^\infty_t([-\varepsilon,T_*),C^{1/3}_\mathrm{x}).
		\end{equation}
	\end{enumerate}
\end{theorem}

Proof of the main result will be given in section \ref{proof of the main theorem}. 

\subsection{Initial data in self-similar variables}
Since $\tau(-\varepsilon)=0$, we have that the initial self-similar time is $s=-\log\varepsilon$. 

When $s=-\log\varepsilon$, $y_1=x_1\varepsilon^{-\frac{3}{2}}$, $y_2=x_2\varepsilon^{-\frac{1}{2}}$, from (\ref{spatial support of initial data in physical variable}) we have that the initial data of $W,Z,A$ are supported in
\begin{equation}
	\mathcal{X}_0=\{|y_1|\le\varepsilon^{-1},|y_2|\le\varepsilon^{-\frac{1}{3}}\}. 
	\label{initial spatial support}
\end{equation}

Now we introduce a large constant $M=M(\alpha,\kappa_0,k)$ to absorb universal constants, here $k$ is the order of energy estimate to be established later in section \ref{energy estimate}. In terms of $M$ and $\varepsilon$, we define a small scale $l$ and a large scale $L$ by
\begin{subequations}
	\begin{align}
		\begin{split}
			l=(\log M)^{-5},
			\label{definition of l}
		\end{split}\\
		\begin{split}
			L=\varepsilon^{-\frac{1}{10}}.
			\label{definition of L}
		\end{split}
	\end{align}
\end{subequations}

From (\ref{initial data for difference in physical variable, 0th and 1st order})(\ref{initial data for difference in physical variable, 4th order})(\ref{initial data for difference in physical variable, 3rd order}) we know that $\widetilde W(y,-\log\varepsilon)$ satisfies
\begin{subequations}
	\begin{align}
		\begin{split}
			\eta^{-\frac{1}{6}}\left|\widetilde{W}(y,-\log\varepsilon)\right|\mathbbm{1}_{|y|\le L}&\le\varepsilon^{\frac{1}{10}},
		\end{split}\\
		\begin{split}
			\eta^{\frac{1}{3}}\left|\partial_1\widetilde{W}(y,-\log\varepsilon)\right|\mathbbm{1}_{|y|\le L}&\le\varepsilon^{\frac{1}{11}},
		\end{split}\\
		\begin{split}
			\left|\partial_2\widetilde{W}(y,-\log\varepsilon)\right|\mathbbm{1}_{|y|\le L}&\le\varepsilon^{\frac{1}{12}},
		\end{split}\\
		\begin{split}
			\left|\partial^\gamma\widetilde{W}(y,-\log\varepsilon)\right|\mathbbm{1}_{|y|\le l}\overset{|\gamma|=4}&{\le}\varepsilon^{\frac{1}{8}}
		\end{split}\\
		\begin{split}
			\left|\partial^\gamma\widetilde{W}(0,-\log\varepsilon)\right|\overset{|\gamma|=3}&{\le}\varepsilon^{\frac{1}{2}-\frac{1}{k-3}}.
		\end{split}	
	\end{align}
	\label{initial condition of tilde W}
\end{subequations}

For $W(y,-\log\varepsilon)$, from (\ref{initial data for w in physical variable, 0th and 1st order, outside}) that for all $y\in\mathcal{X}_0\in\{|y|\ge L\}$, there hold that 
\begin{equation}
	\begin{aligned}
		\eta^{-\frac{1}{6}}\left|{W}(y,-\log\varepsilon)\right|&\le1+\varepsilon^{\frac{1}{11}},\\
		\eta^{\frac{1}{3}}\left|\partial_1{W}(y,-\log\varepsilon)\right|&\le1+\varepsilon^{\frac{1}{12}},\\
		\left|\partial_2{W}(y,-\log\varepsilon)\right|&\le\frac{3}{4}.
	\end{aligned}
	\label{initial condition of W, 0th and 1st order}
\end{equation}
and from (\ref{initial data for w in physical variable, 2nd order}) we have that for all $y\in\mathcal{X}_0$ there hold that
\begin{equation}
	\begin{aligned}
		\eta^{\frac{1}{3}}\left|\partial_{11}{W}(y,-\log\varepsilon)\right|&\le1,\\
		\eta^{\frac{1}{3}}\left|\partial_{12}{W}(y,-\log\varepsilon)\right|&\le1,\\
		\eta^{\frac{1}{6}}\left|\partial_{22}{W}(y,-\log\varepsilon)\right|&\le1.\\
	\end{aligned}
	\label{initial condition of W, 2nd order}
\end{equation}

From (\ref{initial data for z in physical variable})(\ref{initial data for a in physical variable}), we have that the initial data of $Z$ and $A$ satisfy
\begin{equation}
	\left|\partial^\gamma Z(y,-\log\varepsilon)\right|\le\left\{
	\begin{aligned}
		&\varepsilon^{\frac{3}{2}},\ \ \ \gamma_1>0,\ |\gamma|=2\\
		&\varepsilon,\ \ \ \ \ \gamma_1=0,\ |\gamma|\le2,
	\end{aligned}\right.
	\label{initial condition of Z}
\end{equation}
\begin{equation}
	\left|\partial^\gamma A(y,-\log\varepsilon)\right|\le\left\{
	\begin{aligned}
		&\varepsilon^{\frac{3}{2}},\ \ \ \gamma=(1,0)\\
		&\varepsilon,\ \ \ \ \ \gamma_1=0,\ |\gamma|\le2.
	\end{aligned}\right.
	\label{initial condition of A}
\end{equation}

Furthermore, from (\ref{initial data for vorticity in physical variable}) we know the specific vorticity satisfies
\begin{equation}
	\left\|\Omega(\cdot,-\log\varepsilon)\right\|_{L^\infty}\le1.
\end{equation}

Finally from (\ref{initial sobolev norm in physical variable}) we have
\begin{equation}
	\varepsilon\|W(\cdot,-\log\varepsilon)\|_{\dot H^k}^2+\|Z(\cdot,-\log\varepsilon)\|_{\dot H^k}^2+\|A(\cdot,-\log\varepsilon)\|_{\dot H^k}^2\le\varepsilon.
	\label{initial sobolev norm}
\end{equation}

\begin{theorem}[Main theorem in self-similar coordinate]
	Suppose $W(y,-\log\varepsilon)$,$Z(y,-\log\varepsilon)$, $A(y,-\log\varepsilon)\in H^k(\mathbb{R}^2)$ with integer $k$ large enough, and they satisfy (\ref{initial spatial support})-(\ref{initial sobolev norm}), and the initial data of modulation variables $(\kappa,\xi,n_2,\tau,\phi)$ satisfy (\ref{initial condition of modulation variables})(\ref{initial condition for kappa}), then there exists a choice of $\varepsilon\ll1$, such that the system (\ref{evolution of W,Z,A}) coupled with (\ref{evolution of kappa})(\ref{evolution of phi})(\ref{evolution of tau})(\ref{evolution of xi}) admits a global solution, and the solution $(W,Z,A,\kappa,\phi,\tau,\xi)$ satisfies the bootstrap assumptions(which are stated in the next section) for all time. 
\end{theorem}

\section{Bootstrap argument}
To establish global existence in self-similar coordinate, we set up bootstrap argument. 
\subsection{Bootstrap assumption}
We first state the bootstrap assumptions. 
\begin{enumerate}
	\item[(1)]\emph{Assumptions on modulation variables}. For the modulation variables, we assume that
	\begin{equation}\tag{B-M}
		\left\{
		\begin{aligned}
			&\frac{1}{2}\kappa_0\le\kappa\le2\kappa_0,   &&|\dot\kappa|\le M\\
			&|\tau|\le M\varepsilon^2,   &&|\dot{\tau}|\le Me^{-s}\\
			&|\xi|\le M^{\frac{1}{4}}\varepsilon, &&|\dot\xi|\le M^{\frac{1}{4}}\\
			&|n_2|\le M^2\varepsilon^{\frac{3}{2}}, &&|\dot n_2|\le M^2\varepsilon^{\frac{1}{2}}\\
			&|\phi|\le M^2\varepsilon, &&|\dot\phi|\le M^2.
		\end{aligned}
		\right.
		\label{bootstrap assumptions of dynamic variables}
	\end{equation}
	
	\item[(2)]\emph{Assumptions on Spatial support bootstrap}. We define  $\mathcal{X}(s):=\left\{|y_1|\le2\varepsilon^{\frac{1}{2}}e^{\frac{3}{2}s},|y_2|\le2\varepsilon^{\frac{1}{6}}e^{\frac{s}{2}}\right\}$, and assume that
	\begin{equation}\tag{B-S}
		\mathrm{supp}(DW,DZ,DA)\subset\mathcal{X}(s).
		\label{bootstrap assumptions for the spatial support}
	\end{equation}
	We will show that this assumption together with (\ref{support of DN and DT}) imply that $\mathrm{supp}(DU,DS)\subset\mathcal{X}(s)$ in Lemma \ref{spatial support of DU,DS}. 
	
	\item[(3)]\emph{Assumptions on $W$ and $\widetilde{W}$}. For $|\gamma|\le2$, we assume that either $\partial^\gamma W$ is close to $\partial^\gamma\ovl W$ , or it behaves like $\partial^\gamma\ovl W$. More precisely, we assume that
	\begin{equation}\tag{B-$W$}
		\left\{\begin{aligned}
			&|W|\le(1+\varepsilon^{\frac{1}{20}})\eta^{\frac{1}{6}},\ \ 
			&|\partial_1W|\le2\eta^{-\frac{1}{3}},\ \ 
			&|\partial_2W|\le1\\
			&|\partial_{11}W|\le M^{\frac{1}{3}}\eta^{-\frac{1}{3}},\ \ 
			&|\partial_{12}W|\le M^{\frac{2}{3}}\eta^{-\frac{1}{3}},\ \ 
			&|\partial_{22}W|\le M\eta^{-\frac{1}{6}}.
		\end{aligned}\right.
		\label{bootstrap assumptions of W}
	\end{equation}
	Noting that by $\mathrm{supp}\ DW\subset\mathcal{X}(s)$ and $W(0)=\ovl{W}(0)=0$, we have
	\begin{equation}
		|W(y)|\le \int_0^{y_1}2\eta^{-\frac{1}{3}}(y_1',0)dy_1'+\|\partial_2W\|_{L^\infty}|y_2|\lesssim \varepsilon^{\frac{1}{6}}e^{\frac{s}{2}}.
		\label{estimate of W}
	\end{equation}
	
	For $\widetilde{W}$ we assume that
	\begin{equation}\tag{B-$\widetilde{W}$-1}
		\left\{
		\begin{aligned}
			&\left|\widetilde{W}\right|\mathbbm{1}_{|y|\le L}\le\varepsilon^{\frac{1}{11}}\eta^{\frac{1}{6}}\\
			&\left|\partial_1\widetilde{W}\right|\mathbbm{1}_{|y|\le L}\le\varepsilon^{\frac{1}{12}}\eta^{-\frac{1}{3}}\\
			&\left|\partial_2\widetilde{W}\right|\mathbbm{1}_{|y|\le L}\le\varepsilon^{\frac{1}{13}}.
		\end{aligned}
		\right.
		\label{bootstrap assumptions of tilde W when |y|<L}
	\end{equation}
	where $L=\varepsilon^{-\frac{1}{10}}$, and
	\begin{equation}\tag{B-$\widetilde{W}$-2}
		\left|\partial^\gamma\widetilde{W}\right|\mathbbm{1}_{|y|\le l}\le\log^4 M\varepsilon^{\frac{1}{10}}|y|^{4-|\gamma|}+M\varepsilon^{\frac{1}{4}}|y|^{3-|\gamma|}\ \ \ \ \ (\forall|\gamma|\le3),
		\label{bootstrap assumptions of tilde W when |y|<l, |gamma|<4}
	\end{equation}
	\begin{equation}\tag{B-$\widetilde{W}$-3}
		\left|\partial^\gamma\widetilde{W}\right|\mathbbm{1}_{|y|\le l}\le\frac{1}{2}\log^{|\check\gamma|} M\varepsilon^{\frac{1}{10}}\ \ \ \ \ (\forall|\gamma|=4).
		\label{bootstrap assumptions of tilde W when |y|<l, |gamma|=4}
	\end{equation}
	where $l=(\log M)^{-5}$, and
	\begin{equation}\tag{B-$\widetilde{W}^0$}
		\left|\partial^\gamma\widetilde{W}(0,s)\right|\le\varepsilon^{\frac{1}{4}}\ \ \ \ \ (\forall|\gamma|=3,\forall s\ge s_0).
		\label{bootstrap assumptions of tilde W when y=0, |gamma|=3}
	\end{equation}

	\item [(4)]\emph{Assumptions on $Z$ and $A$}. For $Z$, $A$ and their derivatives up to second order, we assume they are small or have decay properties. More precisely, we assume that
	\begin{equation}\tag{B-$Z$}
		\left\{
		\begin{aligned}
			&|Z|\le M\varepsilon,\ \ 
			&|\partial_1Z|\le M^{\frac{1}{2}}e^{-\frac{3}{2}s},\ \ 
			&|\partial_2Z|\le M\varepsilon^{\frac{1}{2}}e^{-\frac{s}{2}}\\
			&|\partial_{11}Z|\le M^{\frac{1}{2}}e^{-\frac{3}{2}s},\ \ 
			&|\partial_{12}Z|\le Me^{-\frac{3}{2}s},\ \ 
			&|\partial_{22}Z|\le Me^{-s}.
		\end{aligned}
		\right.
		\label{bootstrap assumptions of Z}
	\end{equation}
	and
	\begin{equation}\tag{B-$A$}
		\left\{
		\begin{aligned}
			&|A|\le M\varepsilon,\ \ 
			&|\partial_1A|\le Me^{-\frac{3}{2}s}\\
			&|\partial_2A|\le M\varepsilon^{\frac{1}{2}}e^{-\frac{s}{2}},\ \ 
			&|\partial_{22}A|\le Me^{-s}.
		\end{aligned}
		\right.
		\label{bootstrap assumptions of A}
	\end{equation}
\end{enumerate}

\subsection{Bootstrap procedure}
Now we state the improved bootstrap inequality (IB), which supposedly can be deduced from the bootstrap assumptions and the initial conditions: 
\begin{equation}\tag{IB-M}
	\left\{
	\begin{aligned}
		&\frac{3}{4}\kappa_0\le\kappa\le\frac{5}{4}\kappa_0,   &&|\dot\kappa|\le \frac{1}{2}M\\
		&|\tau|\le \frac{1}{4}M\varepsilon^2,   &&|\dot{\tau}|\le \frac{1}{4}Me^{-s}\\
		&|\xi|\le \frac{1}{2}M^{\frac{1}{4}}\varepsilon, &&|\dot\xi|\le \frac{1}{2}M^{\frac{1}{4}}\\
		&|n_2|\le \frac{1}{2}M\varepsilon, &&|\dot n_2|\le \frac{1}{2}M^2\varepsilon^{\frac{1}{2}}\\
		&|\phi|\le \frac{1}{2}M^2\varepsilon, &&|\dot\phi|\le \frac{1}{10}M^2,
	\end{aligned}
	\right.
	\label{refined bootstrap inequality of dynamic variables}
\end{equation}

\begin{equation}\tag{IB-S}
	\mathrm{supp}(DW,DZ,DA)\subset\frac{7}{8}\mathcal{X}(s),
	\label{refined spatial support}
\end{equation}

\begin{equation}\tag{IB-$W$}
	\left\{\begin{aligned}
		&|W|\le(1+\varepsilon^{\frac{1}{19}})\eta^{\frac{1}{6}},\ \ 
		&|\partial_1W|\le\left(1+\varepsilon^{\frac{1}{13}}\right)\eta^{-\frac{1}{3}},\ \ 
		&|\partial_2W|\le\frac{5}{6}\\
		&|\partial_{11}W|\le \frac{1}{2}M^{\frac{1}{3}}\eta^{-\frac{1}{3}},\ \ 
		&|\partial_{12}W|\le \frac{1}{2}M^{\frac{2}{3}}\eta^{-\frac{1}{3}},\ \ 
		&|\partial_{22}W|\le \frac{1}{2}M\eta^{-\frac{1}{6}},
	\end{aligned}\right.
	\label{refined bootstrap inequality of W}
\end{equation}

\begin{equation}\tag{IB-$\widetilde{W}$-1}
	\left\{
	\begin{aligned}
		&\left|\widetilde{W}\right|\mathbbm{1}_{|y|\le L}\le\frac{1}{2}\varepsilon^{\frac{1}{11}}\eta^{\frac{1}{6}}\\
		&\left|\partial_1\widetilde{W}\right|\mathbbm{1}_{|y|\le L}\le\frac{1}{2}\varepsilon^{\frac{1}{12}}\eta^{-\frac{1}{3}}\\
		&\left|\partial_2\widetilde{W}\right|\mathbbm{1}_{|y|\le L}\le\frac{1}{2}\varepsilon^{\frac{1}{13}},
	\end{aligned}
	\right.
	\label{refined bootstrap inequality of tilde W when |y|<L}
\end{equation}

\begin{equation}\tag{IB-$\widetilde{W}$-2}
	\left|\partial^\gamma\widetilde{W}\right|\mathbbm{1}_{|y|\le l}\le\frac{1}{2}\log^4 M\varepsilon^{\frac{1}{10}}|y|^{4-|\gamma|}+\frac{1}{2}M\varepsilon^{\frac{1}{4}}|y|^{3-|\gamma|}\ \ \ \ \ (\forall|\gamma|\le3),
	\label{refined bootstrap inequality of tilde W when |y|<l, |gamma|<4}
\end{equation}

\begin{equation}\tag{IB-$\widetilde{W}$-3}
	\left|\partial^\gamma\widetilde{W}\right|\mathbbm{1}_{|y|\le l}\le\frac{1}{4}\log^{|\check\gamma|} M\varepsilon^{\frac{1}{10}}\ \ \ \ \ (\forall|\gamma|=4),
	\label{refined bootstrap inequality of tilde W when |y|<l, |gamma|=4}
\end{equation}

\begin{equation}\tag{IB-$\widetilde{W}^0$}
	\left|\partial^\gamma\widetilde{W}(0,s)\right|\le\frac{1}{10}\varepsilon^{\frac{1}{4}}\ \ \ \ \ (\forall|\gamma|=3,\forall s\ge s_0),
	\label{refined bootstrap inequality of tilde W when y=0, |gamma|=3}
\end{equation}

\begin{equation}\tag{IB-$Z$}
	\left\{
	\begin{aligned}
		&|Z|\le \frac{1}{2}M\varepsilon,\ \ 
		&|\partial_1Z|\le \frac{1}{2}M^{\frac{1}{2}}e^{-\frac{3}{2}s},\ \ 
		&|\partial_2Z|\le \frac{1}{2}M\varepsilon^{\frac{1}{2}}e^{-\frac{s}{2}}\\
		&|\partial_{11}Z|\le \frac{1}{2}M^{\frac{1}{2}}e^{-\frac{3}{2}s},\ \ 
		&|\partial_{12}Z|\le \frac{1}{2}Me^{-\frac{3}{2}s},\ \ 
		&|\partial_{22}Z|\le \frac{1}{2}Me^{-s},
	\end{aligned}
	\right.
	\label{refined bootstrap inequality of Z}
\end{equation}

\begin{equation}\tag{IB-$A$}
	\left\{
	\begin{aligned}
		&|A|\le M\varepsilon,\ \ 
		&|\partial_1A|\le \frac{1}{2}Me^{-\frac{3}{2}s}\\
		&|\partial_2A|\le \frac{1}{2}M\varepsilon^{\frac{1}{2}}e^{-\frac{s}{2}},\ \ 
		&|\partial_{22}A|\le \frac{1}{2}Me^{-s}.
	\end{aligned}
	\right.
	\label{refined bootstrap inequality of A}
\end{equation}


Compare to the 3d case in \cite{buckmaster2022formation}, we carefully close the bootstrap argument of spatial support in subsection \ref{upper bound of the trajectories}. To prove $W,Z,A$ are constant outside $\frac{7}{8}\mathcal{X}(s)$, we define two rectangles $Q_{big}=\{|y_1|\le M',|y_2|\le M'\}$ and $Q_{small}(s)$ satisfying 
	$$\frac{3}{4}\mathcal{X}(s)\subset Q_{small}(s)\subset\frac{7}{8}\mathcal{X}(s)\subset Q_{big},$$
	where $M'$ can be chosen arbitrarily large. Then we consider the quantity
		$$\int_{Q_{big}\setminus Q_{small}} E(y,s)dy,$$ where $E(y,s)=\frac{1}{2}\left(e^{-s}(W-W_\infty)^2+(Z-Z_\infty)+2(A-A_\infty)^2\right)$. 
		From the equations of $W,Z,A$ and bootstrap assumptions, we find that
		$$\frac{d}{ds}\int_{Q_{big}\setminus Q_{small}} E\le C\int_{Q_{big}\setminus Q_{small}} E.$$
		By Gronwall's inequality and the initial conditions, we can deduce that $W,Z,A$ are constant outside $Q_{small}$.

\section{Immediate corollaries of bootstrap assumptions}
\subsection{Blow-up time}
By the definition of $s$, we have $t=\tau-e^{-s}$. From the bootstrap assumption of $\tau$ and $s\ge-\log\varepsilon$, we can see that if the bootstrap assumptions hold on the interval $[t_0,t]=[-\varepsilon,t]$, then $t$ satisfies
\begin{equation}
	|t-t_0|=|t+\varepsilon|\le\varepsilon+M\varepsilon^2+e^{\log\varepsilon}\le3\varepsilon.
	\label{bootstrap time}
\end{equation}
The blow-up time $T_*$ is defined to be $T_*=\tau(T_*)$. 
\subsection{Closure of bootstrap argument for $W$,$\widetilde{W}$ near the origin}
From estimates (\ref{estimate of ovl W})(\ref{estimates of D ovl W}) of $\ovl{W}$ and bootstrap assumptions (\ref{bootstrap assumptions of tilde W when |y|<L}), we have
\begin{equation}
	\left\{
	\begin{aligned}
		&\left|W\right|\mathbbm{1}_{|y|\le L}\le(1+\varepsilon^{\frac{1}{11}})\eta^{\frac{1}{6}}\\
		&\left|\partial_1W\right|\mathbbm{1}_{|y|\le L}\le(1+\varepsilon^{\frac{1}{12}})\eta^{-\frac{1}{3}}\\
		&\left|\partial_2W\right|\mathbbm{1}_{|y|\le L}\le\frac{2}{3}+\varepsilon^{\frac{1}{13}}.
	\end{aligned}
	\right.
	\label{bootstrap estimate of W and DW}
\end{equation}
Thus we closed the bootstrap argument for $W$ and $DW$ in the region $\{|y|\le L\}$, and by $D^2\widetilde{W}(0,s)=0$, the bootstrap argument for $D^2W$ in $\{|y|\le l\}$ is automatically closed. 

Note that by (\ref{bootstrap estimate of W and DW})(\ref{bootstrap assumptions of W}), for $\varepsilon$ taken small enough, we have
\begin{equation}
	\left|\partial_1W\right|\le(1+\varepsilon^{\frac{1}{12}})\eta^{-\frac{1}{3}}\mathbbm{1}_{|y|\le L}+2\eta^{-\frac{1}{3}}\mathbbm{1}_{|y|>L}\le1+\varepsilon^{\frac{1}{12}}.
	\label{estimate of D1W that will appear in damping terms}
\end{equation}
This bound will be used in the estimate of the damping terms. 

Now we prove (\ref{refined bootstrap inequality of tilde W when |y|<l, |gamma|<4}) for $\widetilde{W}$:  
\begin{equation}
	\begin{aligned}
		\left|\partial^\gamma\widetilde{W}\right|\mathbbm{1}_{|y|\le l}\overset{|\gamma|=3}&{\le}\left|\partial^\gamma\widetilde{W}(0,s)\right|+\left\|D\partial^{\gamma}\widetilde{W}\right\|_{L^\infty(|y|\le l)}|y|\\
		&{\le}\varepsilon^{\frac{1}{4}}+\frac{1}{2}\log^4 M\varepsilon^{\frac{1}{10}}|y|;
	\end{aligned}
\end{equation}
if $|\gamma|\le2$, we have that
\begin{equation}
	\left|\partial^\gamma\widetilde{W}\right|\mathbbm{1}_{|y|\le l}\overset{|\gamma|\le2}{\le}\left\|D\partial^{\gamma}\widetilde{W}(\cdot,s)\right\|_{L^\infty(|\cdot|\le |y|)}|y|.
\end{equation}

\subsection{Spatial support of unknowns}
For the support of unknowns, we have the following lemma. 
\begin{lemma}\label{spatial support of DU,DS}
	$\mathrm{supp}\ (DU,DS)\subset \mathcal{X}(s)$. 
\end{lemma}
\begin{proof}
	According to the spatial support assumption of $(DW,DZ,DA)$, it suffices to show $\mathrm{supp}_x(D_xN,D_xT)\subset\{|x_1|\le2\varepsilon^{\frac{1}{2}},|x_2|\le2\varepsilon^{\frac{1}{6}}\}$. Now by the expression of $N,T$, we only need to show that $\mathrm{supp}_x\ f_{x_2}\subset\{|x_1|\le2\varepsilon^{\frac{1}{2}},|x_2|\le2\varepsilon^{\frac{1}{6}}\}$. Note that $f_{x_2}=\tilde f_{\tilde x_2}(1+\frac{\tilde f_{\tilde x_1}}{1-\tilde f_{\tilde x_1}})$, and $\mathrm{supp}_{\tilde x}\tilde f_{\tilde x_2}\subset\{|\tilde x_1|\le\frac{5}{4}\varepsilon^{\frac{1}{2}},|\tilde x_2|\le\frac{5}{4}\varepsilon^{\frac{1}{6}}\}$, thus we have $\mathrm{supp}_x(D_xN,D_xT)\subset\{|x_1|\le\frac{3}{2}\varepsilon^{\frac{1}{2}},|x_2|\le\frac{3}{2}\varepsilon^{\frac{1}{6}}\}$ by choosing $\varepsilon$ small enough in terms of $M$. 
\end{proof}

From (\ref{spatial support of initial data in physical variable}), we know that in the original $\mathrm{x}$ coordinate, we have
\begin{equation}
	\lim\limits_{|\mathrm{x}|\rightarrow\infty}u(\mathrm{x},-\varepsilon)=\frac{\kappa_0}{2}e_1,\ \lim\limits_{|\mathrm{x}|\rightarrow\infty}\sigma(\mathrm{x},-\varepsilon)=\frac{\kappa_0}{2}.
\end{equation}
From the finite speed propagation of the Euler equations, we have that for all $t\in[-\varepsilon,T_*)$, there hold
\begin{equation}
	\lim\limits_{|\mathrm{x}|\rightarrow\infty}u(\mathrm{x},t)=\frac{\kappa_0}{2}e_1,\ \lim\limits_{|\mathrm{x}|\rightarrow\infty}\sigma(\mathrm{x},t)=\frac{\kappa_0}{2}.
\end{equation}
Note that the coordinate transformation is determined by the modulation variables, and from bootstrap assumptions we can deduce that
\begin{equation}
	y\notin\mathcal{X}(s)\text{ implies that }
	\left\{\begin{aligned}
		&W(y,s)=W_\infty(s)\\
		&Z(y,s)=Z_\infty(s)\\
		&A(y,s)=A_\infty(s)\\
		&S(y,s)=S_\infty(s)\\
		&U(y,s)=U_\infty(s),
	\end{aligned}\right.
\end{equation}
where
\begin{equation}
	\left\{
	\begin{aligned}
		W_\infty(s)&:=\left[\frac{\kappa_0}{2}(n_1+1)-\kappa\right]e^{\frac{s}{2}}\\
		Z_\infty(s)&:=\frac{\kappa_0}{2}(n_1-1)\\
		A_\infty(s)&:=-\frac{\kappa_0}{2}n_2\\
		S_\infty(s)&:=\frac{e^{-\frac{s}{2}}W_\infty+\kappa-Z_\infty}{2}=\frac{\kappa_0}{2}\\
		U_\infty(s)&:=\frac{e^{-\frac{s}{2}}W_\infty+\kappa+Z_\infty}{2}\tilde{e}_1+A_\infty\tilde{e}_2=\frac{\kappa_0n_1}{2}\tilde{e}_1-\frac{\kappa_0n_2}{2}\tilde{e}_2.
	\end{aligned}
	\right.
\end{equation}

\subsection{Estimates related to coordinate transformation}
In this section we will estimate the functions $f$, $J$, $N$, $T$, $Q$, $V$, which only depend on modulation variables. 

\begin{lemma}\label{estimates for functions of coordinate transformation, lemma}
	For any multi-index $\gamma\in \mathbb{Z}_{\ge0}^2$, we have
	\begin{equation}
		\left\{
		\begin{aligned}
			&|\partial_x^\gamma f|\le C_\gamma M^2\varepsilon^{\frac{4}{3}-\frac{\gamma_1}{2}-\frac{\gamma_2}{6}}\\
			&|\partial_x^\gamma(J-1)|\le C_\gamma M^2\varepsilon^{\frac{5}{6}-\frac{\gamma_1}{2}-\frac{\gamma_2}{6}}\\
			&|\partial_x^\gamma (N-\tilde e_1)|\le C_\gamma M^2\varepsilon^{\frac{7}{6}-\frac{\gamma_1}{2}-\frac{\gamma_2}{6}}\\
			&|\partial_x^\gamma (T-\tilde e_2)|\le C_\gamma M^2\varepsilon^{\frac{7}{6}-\frac{\gamma_1}{2}-\frac{\gamma_2}{6}}\\
			&|\partial_x^\gamma (JN-\tilde e_1)|\le C_\gamma M^2\varepsilon^{\frac{5}{6}-\frac{\gamma_1}{2}-\frac{\gamma_2}{6}}\\
			&|\partial_x^\gamma\partial_tf|\le C_\gamma M^2\varepsilon^{\frac{1}{3}-\frac{\gamma_1}{2}-\frac{\gamma_2}{6}}\\
			&\left|\partial_x^\gamma\frac{\partial_tf}{1+f_{x_1}}\right|\le C_\gamma M^2\varepsilon^{\frac{1}{3}-\frac{\gamma_1}{2}-\frac{\gamma_2}{6}}\\
			&\left|\partial_x^\gamma\partial_t N\right|\le C_\gamma M^2\varepsilon^{\frac{1}{6}-\frac{\gamma_1}{2}-\frac{\gamma_2}{6}}\\
			&\left|\partial_x^\gamma\partial_t T\right|\le C_\gamma M^2\varepsilon^{\frac{1}{6}-\frac{\gamma_1}{2}-\frac{\gamma_2}{6}}.\\
		\end{aligned}
		\right.
		\label{estimates of f-dependent functions, inequalities}
	\end{equation}
\end{lemma}
\begin{proof}
	From the expression of $\tilde f$ and the bootstrap assumption for $\phi$ and $\dot\phi$, it is not hard to see that $|\partial_{\tilde x}^\gamma\tilde f|\le C_\gamma M^2\varepsilon^{\frac{4}{3}-\frac{\gamma_1}{2}-\frac{\gamma_2}{6}}$, $|\partial_{\tilde x}^\gamma\partial_t\tilde f|\le C_\gamma M^2\varepsilon^{\frac{1}{3}-\frac{\gamma_1}{2}-\frac{\gamma_2}{6}}$.
	
	Using chain rule, one can see that
	\begin{equation}
		\left\{
		\begin{aligned}
			&\partial_{x_1}=\frac{\partial_{\tilde x_1}}{1-\tilde f_{\tilde x_1}}\\
			&\partial_{x_2}=\frac{\tilde f_{\tilde x_2}}{1-\tilde f_{\tilde x_1}}\partial_{\tilde x_1}+\partial_{\tilde x_2}.
		\end{aligned}\right.
	\end{equation}
	By Faà di Bruno's formula, we have
	\begin{equation}
		\begin{aligned}
			\left|\partial_{\tilde x}^\gamma\left(\frac{1}{1-\tilde f_{\tilde x_1}}\right)\right|\overset{\gamma>0}&{\lesssim}_{|\gamma|}\sum\limits_{\sum\limits_{\beta\le\gamma}\beta m_\beta=\gamma}\left|1-\tilde f_{\tilde x_1}\right|^{-1-\sum\limits_{\beta\le\gamma}m_\beta}\prod_{\beta\le\gamma}\left|\partial_{\tilde x}^\beta\tilde f_{\tilde x_1}\right|^{m_\beta}\\
			\overset{\varepsilon\ll1}&{\lesssim}\sum\limits_{\sum\limits_{\beta\le\gamma}\beta m_\beta=\gamma}\left(1-\varepsilon^{\frac{1}{2}}\right)^{-\sum\limits_{\beta\le\gamma}m_\beta}\prod_{\beta\le\gamma}\left(M^2\varepsilon^{\frac{4}{3}-\frac{\beta_1+1}{2}-\frac{\beta_2}{6}}\right)^{m_\beta}\\
			&\lesssim\varepsilon^{-\frac{\gamma_1}{2}-\frac{\gamma_2}{6}}\sum\limits_{\sum\limits_{\beta\le\gamma}\beta m_\beta=\gamma}\left((1-\varepsilon^{\frac{1}{2}})M^2\varepsilon^{\frac{5}{6}}\right)^{\sum\limits_{\beta\le\gamma}m_\beta}\lesssim M^2\varepsilon^{\frac{5}{6}-\frac{\gamma_1}{2}-\frac{\gamma_2}{6}}.
		\end{aligned}
	\end{equation}
	
	And by Leibniz rule, we have
	\begin{equation}
		\begin{aligned}
			\left|\partial_{\tilde x}^\gamma\left(\frac{\tilde f_{\tilde x_2}}{1-\tilde f_{\tilde x_1}}\right)\right|&\lesssim\sum_{0<\beta\le\gamma}\left|\partial^{\gamma-\beta}\tilde f_{\tilde x_2}\right|\left|\partial_{\tilde x}^\beta\left(\frac{1}{1-\tilde f_{\tilde x_1}}\right)\right|+M^2\varepsilon^{\frac{4}{3}-\frac{\gamma_1}{2}-\frac{\gamma_2+1}{6}}\lesssim M^2\varepsilon^{\frac{7}{6}-\frac{\gamma_1}{2}-\frac{\gamma_2}{6}}.
		\end{aligned}
	\end{equation}
	
	Note that
	\begin{equation}
		\begin{aligned}
			\partial_{x_2}^k&=\left(\frac{\tilde f_{\tilde x_2}}{1-\tilde f_{\tilde x_1}}\partial_{\tilde x_1}+\partial_{\tilde x_2}\right)^k=\sum_{\substack{\sum\limits_{|\beta|\le k}n_\beta=\gamma_1+\sum\limits_{|\beta|\le k}\beta_1n_\beta\\\gamma_1+\gamma_2+\sum\limits_{|\beta|\le k}n_\beta|\beta|=k}}C(k,\gamma,n_\beta)\prod_{|\beta|\le k}\left(\partial_{\tilde x}^\beta\left(\frac{\tilde f_{\tilde x_2}}{1-\tilde f_{\tilde x_1}}\right)\right)^{n_\beta}\partial_{\tilde x}^\gamma.
		\end{aligned}
	\end{equation}
	Thus we have
	\begin{equation}
		\begin{aligned}
			\left|\partial_{\tilde x_1}^j(\partial_{x_2}^kf)\right|&\lesssim\left|\partial_{\tilde x_1}^j\sum_{\substack{\sum\limits_{|\beta|\le k}n_\beta=\gamma_1+\sum\limits_{|\beta|\le k}\beta_1n_\beta\\\gamma_1+\gamma_2+\sum\limits_{|\beta|\le k}n_\beta|\beta|=k}}C(k,\gamma,n_\beta)\prod_{|\beta|\le k}\left(\partial_{\tilde x}^\beta\left(\frac{\tilde f_{\tilde x_2}}{1-\tilde f_{\tilde x_1}}\right)\right)^{n_\beta}\partial_{\tilde x}^\gamma\tilde f\right|\\
			&\lesssim_{j,k}\sum_{\substack{\sum\limits_{|\beta|\le k+j}n_\beta+j=\gamma_1+\sum\limits_{|\beta|\le k+j}\beta_1n_\beta\\\gamma_1+\gamma_2+\sum\limits_{|\beta|\le k+j}n_\beta|\beta|=k+j}}\ \prod_{|\beta|\le k+j}\left|\partial_{\tilde x}^\beta\left(\frac{\tilde f_{\tilde x_2}}{1-\tilde f_{\tilde x_1}}\right)\right|^{n_\beta}|\partial_{\tilde x}^\gamma f|\\
			&\lesssim\sum_{\substack{\sum\limits_{|\beta|\le k+j}n_\beta+j=\gamma_1+\sum\limits_{|\beta|\le k+j}\beta_1n_\beta\\\gamma_1+\gamma_2+\sum\limits_{|\beta|\le k+j}n_\beta|\beta|=k+j}}\ \prod_{|\beta|\le k+j}\left(M^2\varepsilon^{\frac{7}{6}-\frac{\beta_1}{2}-\frac{\beta_2}{6}}\right)^{n_\beta}M^2\varepsilon^{\frac{4}{3}-\frac{\gamma_1}{2}-\frac{\gamma_2}{6}}\\
			&\lesssim M^2\varepsilon^{\frac{4}{3}-\frac{j}{2}-\frac{k}{6}}\sum_{\substack{\sum\limits_{|\beta|\le k+j}n_\beta+j=\gamma_1+\sum\limits_{|\beta|\le k+j}\beta_1n_\beta\\\gamma_1+\gamma_2+\sum\limits_{|\beta|\le k+j}n_\beta|\beta|=k+j}}\left(M^2\varepsilon\right)^{\sum\limits_{|\beta|\le k+j}n_\beta}\lesssim M^2\varepsilon^{\frac{4}{3}-\frac{j}{2}-\frac{k}{6}}.
		\end{aligned}
	\end{equation}
	Finally, we have
	\begin{equation}
		\begin{aligned}
			\left|\partial_x^\gamma f\right|&=\left|\left(\frac{\partial_{\tilde x_1}}{1-\tilde f_{\tilde x_1}}\right)^{\gamma_1}\partial_{x_2}^{\gamma_2}f\right|\\
			\overset{\gamma_1\ge1}&{\lesssim}\sum_{j=1}^{\gamma_1}\sum_{\substack{n_1+2n_2+\cdots+\gamma_1n_{\gamma_1}=\gamma_1-j\\ n_0+n_1+\cdots+n_{\gamma_1}=\gamma_1}}\left|\frac{1}{1-\tilde f_{\tilde x_1}}\right|^{n_0}\left|\partial_{\tilde x_1}\left(\frac{1}{1-\tilde f_{\tilde x_1}}\right)\right|^{n_1}\cdots\left|\partial_{\tilde x_1}^{\gamma_1}\left(\frac{1}{1-\tilde f_{\tilde x_1}}\right)\right|^{n_{\gamma_1}}\left|\partial_{\tilde x_1}^j\partial_{x_2}^{\gamma_2}f\right|\\
			&\lesssim\sum_{j=1}^{\gamma_1}\sum_{\substack{n_1+2n_2+\cdots+\gamma_1n_{\gamma_1}=\gamma_1-j\\ n_0+n_1+\cdots+n_{\gamma_1}=\gamma_1}}\left(1-\varepsilon^{\frac{1}{2}}\right)^{-n_0}\left(M^2\varepsilon^{\frac{5}{6}-\frac{1}{2}}\right)^{n_1}\cdots\left(M^2\varepsilon^{\frac{5}{6}-\frac{\gamma_1}{2}}\right)^{n_{\gamma_1}}\left|\partial_{\tilde x_1}^j\partial_{x_2}^{\gamma_2}f\right|\\
			&\lesssim\sum_{j=1}^{\gamma_1}\varepsilon^{-\frac{\gamma_1-j}{2}}\left|\partial_{\tilde x_1}^j\partial_{x_2}^{\gamma_2}f\right|\sum_{\substack{n_1+2n_2+\cdots+\gamma_1n_{\gamma_1}=\gamma_1-j\\ n_0+n_1+\cdots+n_{\gamma_1}=\gamma_1}}\left[(1-\varepsilon^{\frac{1}{2}})M^2\varepsilon^{\frac{5}{6}}\right]^{\gamma_1-n_0}\\
			&\lesssim\varepsilon^{-\frac{\gamma_1}{2}}\sum_{j=1}^{\gamma_1}\varepsilon^{\frac{j}{2}}M^2\varepsilon^{\frac{4}{3}-\frac{j}{2}-\frac{\gamma_2}{6}}\lesssim M^2\varepsilon^{\frac{4}{3}-\frac{\gamma_1}{2}-\frac{\gamma_2}{6}}.
		\end{aligned}
	\end{equation}
	One can check the same estimate holds when $\gamma_1=0$. 
	
	Also from Faà di Bruno's formula one can see that for $\alpha\in\mathbb{R}$ and $\gamma>0$, we have  $\left|\partial_x^\gamma(1+f_{x_2}^2)^\alpha\right|\lesssim_{\alpha,\gamma} M^4\varepsilon^{\frac{7}{3}-\frac{\gamma_1}{2}-\frac{\gamma_2}{6}}$, this estimate combining with Leibniz rule gives that $|\partial_x^\gamma N|\lesssim M^2\varepsilon^{\frac{7}{6}-\frac{\gamma_1}{2}-\frac{\gamma_2}{6}}$ for $\gamma>0$. $|N-\tilde e_1|\lesssim M^4\varepsilon^{\frac{7}{6}}$ can be checked separately. The estimates of $N$ implies $|\partial_x^\gamma (T-\tilde e_2)|\lesssim M^2\varepsilon^{\frac{7}{6}-\frac{\gamma_1}{2}-\frac{\gamma_2}{6}}$ for $\gamma\ge0$ since $T=N^\perp$. The estimate of $JN$ is similar. 
	
	As for $J=\frac{\sqrt{1+f_{x_2}^2}}{1+f_{x_1}}$, we use Leibniz rule to deduce that $|\partial^\gamma J|\lesssim M^2\varepsilon^{\frac{5}{6}-\frac{\gamma_1}{2}-\frac{\gamma_2}{6}}$ holds for $\gamma>0$, then one can check $|J-1|\lesssim M^2\varepsilon^{\frac{5}{6}}$. 
	
	The estimates of $\partial_t f$ and $\frac{\partial_t f}{1+f_{x_1}}$ is much the same and rely on the facts that $|\partial_{\tilde x}^\gamma\partial_t\tilde f|\lesssim M^2\varepsilon^{\frac{1}{3}-\frac{\gamma_1}{2}-\frac{\gamma_2}{6}}$ and $(\partial_t)_{x}f=\frac{(\partial_t)_{\tilde x}\tilde f}{1-\tilde f_{\tilde x_1}}$.  
\end{proof}

Here we emphasize that $C_\gamma$ in Lemma \ref{estimates for functions of coordinate transformation, lemma} grows at least exponentially since $f$ is compactly supported and cannot be analytic. 

\begin{lemma}
	For $\varepsilon\ll1$ small enough and $M\gg1$ large enough we have 
	\begin{equation}
		|Q|\le M^2\varepsilon^{\frac{1}{2}}.
		\label{estimate of Q}
	\end{equation}
\end{lemma}
\begin{proof}Since we have
	\begin{equation}
		Q=\dot R^TR=
		\begin{bmatrix}
			0 &-n_1\dot n_2+n_2\dot n_1\\
			-n_2\dot n_1+n_1\dot n_2 &0
		\end{bmatrix},
	\end{equation}
the rest is appealing to $n_1=\sqrt{1-n_2^2}$ and the bootstrap assumptions (\ref{bootstrap assumptions of dynamic variables}) for $n_2$ and $\dot n_2$. 
\end{proof}

\begin{lemma}
	For $y\in10\mathcal{X}(s)=\{|y_1|\le20\varepsilon^{\frac{1}{2}}e^{\frac{3}{2}s},|y_2|\le20\varepsilon^{\frac{1}{6}}e^{\frac{s}{2}}\}$, we have
	\begin{equation}
		|V|\lesssim M^{\frac{1}{4}}
		\label{estimate of of V}
	\end{equation}
	and for $\forall y\in\mathbb{R}^2$, it holds that
	\begin{equation}
		\left\{
		\begin{aligned}
			&|\partial_1V|\lesssim M^2\varepsilon^{\frac{1}{2}}e^{-\frac{3}{2}s}\\
			&|\partial_2V|\lesssim M^2\varepsilon^{\frac{1}{2}}e^{-\frac{s}{2}}\\
			&|\partial_{11}V|\lesssim M^4\varepsilon^{\frac{5}{6}}e^{-3s}\\
			&|\partial_{12}V|\lesssim M^4\varepsilon^{\frac{7}{6}}e^{-2s}\\
			&|\partial_{22}V|\lesssim M^4\varepsilon^{\frac{3}{2}}e^{-s}\\
			&|\partial^\gamma V|\overset{|\gamma|\ge3}{\lesssim} M^4\varepsilon^{\frac{11}{6}}e^{-(\gamma_1+\gamma_2/3)s}\\
			&|\partial^\gamma V|\overset{|\gamma|\ge 1}{\lesssim} M^4\varepsilon^{\frac{2}{3}}e^{-(\gamma_1+\gamma_2/3)s}.
		\end{aligned}
		\right.
		\label{estimates of derivatives of V}
	\end{equation}
\end{lemma}
\begin{proof}
	Note that
	\begin{equation}
		V(y,s)=\frac{1+\alpha}{2}\left(Q\begin{bmatrix}
			y_1e^{-\frac{3}{2}s}+f\\ y_2e^{-\frac{s}{2}}
		\end{bmatrix}-R^T\dot\xi\right).
	\end{equation}
	By $R\in SO(2)$ and (\ref{bootstrap assumptions of dynamic variables})(\ref{estimates of f-dependent functions, inequalities}), we have the above estimates. 
\end{proof}

\subsection{Estimates for $U,S$}
\begin{lemma}
	For $U\cdot N$ and $S$, we have that
	\begin{equation}
		|\partial^\gamma(U\cdot N)|+|\partial^\gamma S|\lesssim\left\{
		\begin{aligned}
			&M^{\frac{1}{4}}\ \ &\gamma=(0,0)\\
			&e^{-\frac{s}{2}}\eta^{-\frac{1}{3}}\ \  &\gamma=(1,0)\\
			&e^{-\frac{s}{2}}\ \  &\gamma=(0,1)\\
			&M^{\frac{1}{3}}e^{-\frac{s}{2}}\eta^{-\frac{1}{3}}\ \  &\gamma=(2,0)\\
			&M^{\frac{2}{3}}e^{-\frac{s}{2}}\eta^{-\frac{1}{3}}\ \  &\gamma=(1,1)\\
			&Me^{-\frac{s}{2}}\eta^{-\frac{1}{6}}\ \  &\gamma=(0,2).
		\end{aligned}
		\right.
		\label{estimates of U dot N,S}
	\end{equation}
\end{lemma}
\begin{proof}
	One can express $U\cdot N$, $S$ in terms of $W$, $Z$, $A$ as in (\ref{U,S in terms of W,Z,A}). Then by directly appealing to the bootstrap assumptions we obtain the desired estimates. 
\end{proof}

\begin{lemma}
	By taking $\varepsilon$ sufficiently small, we have
	\begin{equation}
		\left\{
		\begin{aligned}
			&|U|\lesssim M^{\frac{1}{4}}\\
			&|\partial_1U|\le\left(1+\varepsilon^{\frac{3}{4}}\right)e^{-\frac{s}{2}}\\
			&|\partial_2U|\le e^{-\frac{s}{2}}\\
			&|\partial_1S|\le(1+\varepsilon)e^{-\frac{s}{2}}\\
			&|\partial_2S|\le\left(\frac{1}{2}+\varepsilon^{\frac{1}{2}}\right)e^{-\frac{s}{2}}.
		\end{aligned}
		\right.
		\label{estimates of U,S}
	\end{equation}
\end{lemma}
\begin{proof}
	Express $U$ in terms of $W$, $Z$, $A$, then use bootstrap assumptions and the estimates (\ref{estimates of f-dependent functions, inequalities}) of $N$, $T$. 
\end{proof}

\subsection{Transport estimates}
\begin{lemma}
	For $\varepsilon\ll1$ and $\forall y\in10\mathcal{X}(s)$, we have
	\begin{equation}
		\left\{
		\begin{aligned}
			&|\partial_1G_A|\lesssim M^2e^{-\frac{5}{6}s},\ \ 
			&|\partial_2G_A|\lesssim M^2\varepsilon^{\frac{1}{6}}\\
			&|\partial_{11}G_A|\lesssim M^\frac{1}{2}e^{-s},\ \ 
			&|\partial_{12}G_A|\lesssim Me^{-s},\ \ 
			&|\partial_{22}G_A|\lesssim M^2e^{-\frac{s}{2}}.
		\end{aligned}
		\right.
		\label{estimates of derivatives of G_A}
	\end{equation}
\end{lemma}
\begin{proof}
	We first deal with $\partial_1G_A$. 
	Using the definition (\ref{transport terms of W,Z,A}) of $G_A$, the estimates (\ref{estimates of f-dependent functions, inequalities}) for functions of coordinate transformation, estimates (\ref{estimate of of V})(\ref{estimates of derivatives of V}) for $V$, and the bootstrap assumptions, and by Leibniz rule, we have that
	\begin{equation}
		\begin{aligned}
			|\partial_1G_A|&\lesssim e^{\frac{s}{2}}\left|\partial_1\frac{\partial_tf}{1+f_{x_1}}\right|+e^{\frac{s}{2}}|\partial_1J|(\kappa_0+|Z|+|V|)+e^{\frac{s}{2}}|\partial_1Z|+e^{\frac{s}{2}}|\partial_1(V\cdot N)|\\
			&\lesssim e^{\frac{s}{2}}M^2\varepsilon^{-\frac{1}{6}}e^{-\frac{3}{2}s}+e^{\frac{s}{2}}\varepsilon^{\frac{1}{3}}e^{-\frac{3}{2}s}M^{\frac{1}{4}}+e^{\frac{s}{2}}(M^{\frac{1}{2}}e^{-\frac{3}{2}s}+M^2\varepsilon^{\frac{1}{2}}e^{-\frac{3}{2}s}+M^{2+\frac{1}{4}}\varepsilon e^{-\frac{3}{2}s})\\
			&\lesssim M^2\varepsilon^{-\frac{1}{6}}e^{-s}\lesssim M^2e^{-\frac{5}{6}s}.
		\end{aligned}
	\end{equation}
The other derivatives of $G_A$ are estimated in the similar way. 
\end{proof}

\begin{lemma}
	For $\varepsilon\ll1$ and $\forall y\in\mathcal{X}(s)$, we have
	\begin{equation}
		\left\{
		\begin{aligned}
			&|g_A|\lesssim M^{\frac{1}{4}}e^{\frac{s}{2}}\\
			&|\partial_1g_A|\le3\\
			&|\partial_2g_A|\le2\\
			&|D^2g_A|\lesssim M\eta^{-\frac{1}{6}}+M^2e^{-\frac{s}{2}}\\
			&|\partial_1h_A|\lesssim e^{-s}\\
			&|\partial_2h_A|\lesssim e^{-s}.
		\end{aligned}
		\right.
		\label{estimates of transport terms for energy estimate}
	\end{equation}
\end{lemma}
\begin{proof}
	Use the definition (\ref{transport terms of W,Z,A}) and the estimates (\ref{bootstrap assumptions of W})(\ref{estimates of f-dependent functions, inequalities})(\ref{estimates of derivatives of G_A}), estimate similarly as we did in the proof of (\ref{estimates of derivatives of G_A}) with more care since there is no room of a universal constant here. 
\end{proof}

\section{Energy estimate}\label{energy estimate}
To overcome the loss of derivative in $L^\infty$ estimates of $W$, $Z$, and $A$, we will establish an additional energy estimate to control the $\dot H^k$($k\ll1$) norms of $W$, $Z$, and $A$. It is crutial that in the proof of energy estimate we only use the bootstrap assumptions, not requiring any information on higher order derivatives. 
\begin{proposition}[Energy estimate for $W$, $Z$, $A$]
	For an integer $k\ge 18$, and a constant $\lambda=\lambda(k)$, 
	\begin{equation}
		\|Z(\cdot,s)\|_{\dot H^k}^2+\|A(\cdot,s)\|_{\dot H^k}^2\le2\lambda^{-k}e^{-s}+M^{4k}e^{-s}(1-\varepsilon^{-s}e^{-s})\lesssim M^{4k}e^{-s},
	\end{equation}
	\begin{equation}
		\|W(\cdot,s)\|_{\dot H^k}^2\le2\lambda^{-k}\varepsilon^{-1}e^{-s}+M^{4k}(1-\varepsilon^{-s}e^{-s}).
	\end{equation}
\end{proposition}

We will prove this by using the $\dot H^k$ bound for $(U,S)$, and the fact that the $\dot H^k$ norm of $(W,Z,A)$ can be controlled by the $\dot H^k$ norm of $(U,S)$. More precisely, we have: 
\begin{lemma}\label{W,Z,A controlled by U,S}The following inequalities hold: 
	\begin{equation}
		\begin{aligned}
			\|W\|_{\dot H^k}\lesssim_k e^{\frac{s}{2}}\left(\|U\|_{\dot H^k}+\|S\|_{\dot H^k}+M^{\frac{9}{4}}\varepsilon^{\frac{3}{2}}e^{-\frac{k-3}{3}s}\right),\\
			\|Z\|_{\dot H^k}+\|A\|_{\dot H^k}\lesssim_k \|U\|_{\dot H^k}+\|S\|_{\dot H^k}+M^{\frac{9}{4}}\varepsilon^{\frac{3}{2}}e^{-\frac{k-3}{3}s}.
		\end{aligned}
	\end{equation}
\end{lemma}
\begin{proof}
	We first estimate $\|W\|_{\dot H^k}$. Note that by (\ref{W,Z,A in terms of U,S}), $\mathrm{supp}(DU,DS)\subset\mathcal{X}(s)$, we have
	\begin{equation}
		\begin{aligned}
			e^{-\frac{s}{2}}\|\partial^\gamma W\|_{L^2(\mathbb{R}^2)}\overset{|\gamma|=k}&{\lesssim_k}\|\partial^\gamma S\|_{L^2}+\sum_{\beta\le\gamma}\|\partial^{\gamma-\beta}U\cdot\partial^\beta N\|_{L^2(\mathcal{X}(s))}\\
			&\lesssim \|S\|_{\dot{H}^k}+\|U\|_{L^\infty}\|\partial^\gamma N\|_{L^\infty}|\mathcal{X}(s)|^{\frac{1}{2}}+\|\partial^\gamma U\|_{L^2}+\sum_{0<\beta<\gamma}\|\partial^{\gamma-\beta}U\|_{L^2}\|\partial^\beta N\|_{L^\infty}\\
			\overset{\text{Poincaré}}&{\lesssim_k}\|S\|_{\dot H^k}+\|U\|_{\dot H^k}+M^{\frac{1}{4}}M^2\varepsilon^{\frac{7}{6}-\frac{\gamma_1}{2}-\frac{\gamma_2}{6}}e^{-(\frac{3}{2}\gamma_1+\frac{1}{2}\gamma_2)s}\varepsilon^{\frac{1}{3}}e^s\\
			&\ \ \ \ \  +\sum_{0<\beta<\gamma}(\varepsilon^{\frac{1}{6}}e^{\frac{s}{2}})^{|\beta|}\|D^kU\|_{L^2}M^2\varepsilon^{\frac{7}{6}-\frac{\beta_1}{2}-\frac{\beta_2}{6}}e^{-(\frac{3}{2}\beta_1+\frac{1}{2}\beta_2)s}\\
			&\lesssim\|S\|_{\dot H^k}+\|U\|_{\dot H^k}+M^{\frac{9}{4}}\varepsilon^{\frac{3}{2}}e^{-\frac{|\gamma|-3}{3}s}.
		\end{aligned}
	\end{equation}
The estimates of $Z$ and $A$ are similar. 
\end{proof}

\begin{definition}[Modified $\dot H^k$ norm]We define
	\begin{equation}
		E_k^2(s):=\sum_{|\gamma|=k}\lambda^{\gamma_2}\left(\|\partial^\gamma U(\cdot,s)\|_{L^2}^2+\|\partial^\gamma S(\cdot,s)\|_{L^2}^2\right),
	\end{equation}
	where $\lambda\in(0,1)$ is to be specified below. Clearly we have the norm equivalence: 
	\begin{equation}
		\lambda^k\left(\|U\|_{\dot H^k}^2+\|S\|_{\dot H^k}^2\right)\le E_k^2\le\|U\|_{\dot H^k}^2+\|S\|_{\dot H^k}^2.
	\end{equation}
\end{definition}

\subsection{Evolution of derivatives of $(U,S)$}
Applying $\partial^\gamma$ to both sides of the $(U,S)$ equation (\ref{equation of U,S}), we see that
\begin{subequations}
	\begin{equation}
			\begin{aligned}
			\partial_s\partial^\gamma U_i&-\beta_\tau e^{-s}Q_{ij}\partial^\gamma U_j+\mathcal{V}_A\cdot\nabla\partial^\gamma U_i+D_\gamma\partial^\gamma U_i+\beta_3\beta_\tau(1+\gamma_1)JN_i\partial^\gamma S\partial_1W,\\
			&+2\beta_3\beta_\tau S\left(e^{\frac{s}{2}}JN_i\partial_1\partial^\gamma S+e^{-\frac{s}{2}}\delta_{i2}\partial_2\partial^\gamma S\right)=F_{U_i}^{(\gamma)},
			\end{aligned}
		\label{equation of derivatives of U_i}
	\end{equation}
	\begin{equation}
		\begin{aligned}
			\partial_s\partial^\gamma S&+\mathcal{V}_A\cdot\nabla\partial^\gamma S+D_\gamma\partial^\gamma S+\beta_\tau(\beta_1+\beta_3\gamma_1)JN\cdot\partial^\gamma U\partial_1W\\
			&+2\beta_3\beta_\tau S\left(e^{\frac{s}{2}}JN\cdot\partial_1\partial^\gamma U+e^{-\frac{s}{2}}\partial_2\partial^\gamma U_2\right)=F_{S}^{(\gamma)}
		\end{aligned}
		\label{equation of derivatives of S}
	\end{equation}	
\end{subequations}
where $D_\gamma=\frac{1}{2}|\gamma|+\gamma_1(1+\partial_1g_U)$, and the forcing terms are  $F_{U_i}^{(\gamma)}=F_{U_i}^{(\gamma,U)}+F_{U_i}^{(\gamma-1,U)}+F_{U_i}^{(\gamma,S)}+F_{U_i}^{(\gamma-1,S)}$, $F_{S}^{(\gamma)}=F_{S}^{(\gamma,U)}+F_{S}^{(\gamma-1,U)}+F_{S}^{(\gamma,S)}+F_{S}^{(\gamma-1,S)}$. Here
\begin{subequations}
	\begin{align}
		\begin{split}
			F_{U_i}^{(\gamma,U)}=&-2\beta_1\beta_\tau\left(e^{\frac{s}{2}}JN_j\partial^\gamma U_j\partial_1 U_i+e^{-\frac{s}{2}}\partial^\gamma U_2\partial_2 U_i\right)\\
			&-\gamma_2\partial_2g_A\partial_1\partial^{\gamma-e_2} U_i-\sum_{\substack{|\beta|=|\gamma|-1\\\beta\le\gamma}}\binom{\gamma}{\beta}\partial^{\gamma-\beta}h_A\partial_2\partial^\beta U_i\\
			=&F_{U_i,(1)}^{(\gamma,U)}+F_{U_i,(2)}^{(\gamma,U)}+F_{U_i,(3)}^{(\gamma,U)},
			\label{Ui top order forcing terms involving Ui}
		\end{split}\\
		\begin{split}
			F_{U_i}^{(\gamma-1,U)}=&-\sum_{\substack{{1\le|\beta|\le|\gamma|-2}\\{\beta\le\gamma}}}\binom{\gamma}{\beta}\left(\partial^{\gamma-\beta}g_A\partial_1\partial^\beta U_i+\partial^{\gamma-\beta}h_A\partial_2\partial^\beta U_i\right)\\
			&-2\beta_1\beta_\tau e^{\frac{s}{2}}[\partial^\gamma,JN]\cdot U\partial_1 U_i-\beta_\tau e^{\frac{s}{2}}\partial^\gamma\left(2\beta_1V\cdot JN-\frac{\partial_tf}{1+f_{x_1}}\right)\partial_1U_i-2\beta_1\beta_\tau e^{-\frac{s}{2}}\partial^\gamma V_2\partial_2U_i\\
			=&F_{U_i,(1)}^{(\gamma-1,U)}+F_{U_i,(2)}^{(\gamma-1,U)}+F_{U_i,(3)}^{(\gamma-1,U)}+F_{U_i,(4)}^{(\gamma-1,U)},
			\label{Ui lower order forcing terms involving Ui}
		\end{split}\\
		\begin{split}
			F_{U_i}^{(\gamma,S)}=&-2\beta_3\beta_\tau \gamma_2e^{\frac{s}{2}}\partial_2(SJN_i)\partial_1\partial^{\gamma-e_2}S-\beta_3\beta_\tau(1+\gamma_1)e^{\frac{s}{2}}JN_i\partial_1Z\partial^\gamma S\\
			&-2\beta_3\beta_\tau e^{-\frac{s}{2}}\delta_{i2}\sum_{\substack{|\beta|=|\gamma|-1\\\beta\le\gamma}}\binom{\gamma}{\beta}\partial^{\gamma-\beta}S\partial_2\partial^\beta S-2\beta_3\beta_\tau\delta_{i2}e^{-\frac{s}{2}}\partial^\gamma S\partial_2S-2\beta_3\beta_\tau\gamma_1e^{\frac{s}{2}}\partial_1(JN_i)S\partial^\gamma S\\
			=&F_{U_i,(1)}^{(\gamma,S)}+F_{U_i,(2)}^{(\gamma,S)}+F_{U_i,(3)}^{(\gamma,S)}+F_{U_i,(4)}^{(\gamma,S)}+F_{U_i,(5)}^{(\gamma,S)},
			\label{Ui top order forcing terms involving S}
		\end{split}\\
		\begin{split}
			F_{U_i}^{(\gamma-1,S)}=&-2\beta_3\beta_\tau\sum_{\substack{{1\le|\beta|\le|\gamma|-2}\\{\beta\le\gamma}}}\binom{\gamma}{\beta}\left(e^{\frac{s}{2}}\partial^{\gamma-\beta}(SJN_i)\partial_1\partial^\beta S+e^{-\frac{s}{2}}\delta_{i2}\partial^{\gamma-\beta}S\partial_2\partial^\beta S\right)\\
			&-2\beta_3\beta_\tau e^{\frac{s}{2}}[\partial^\gamma,JN_i]S\partial_1S\\
			=&F_{U_i,(1)}^{(\gamma-1,S)}+F_{U_i,(2)}^{(\gamma-1,S)},
			\label{Ui lower order forcing terms involving S}
		\end{split}\\
		\begin{split}
			F_{S}^{(\gamma,S)}=&-2\beta_3\beta_\tau\left(e^{\frac{s}{2}}\partial^\gamma SJN_j\partial_1 U_j+e^{-\frac{s}{2}}\partial^\gamma S\partial_2 U_2\right)\\
			&-\gamma_2\partial_2g_A\partial_1\partial^{\gamma-e_2} S-\sum_{\substack{|\beta|=|\gamma|-1\\\beta\le\gamma}}\binom{\gamma}{\beta}\partial^{\gamma-\beta}h_A\partial_2\partial^\beta S,
		\end{split}\\
		\begin{split}
			F_S^{(\gamma-1,S)}=&-\sum_{\substack{{1\le|\beta|\le|\gamma|-2}\\{\beta\le\gamma}}}\binom{\gamma}{\beta}\left(\partial^{\gamma-\beta}g_A\partial_1\partial^\beta S+\partial^{\gamma-\beta}h_A\partial_2\partial^\beta S\right)\\
			&-2\beta_3\beta_\tau\sum_{\substack{{1\le|\beta|\le|\gamma|-2}\\{\beta\le\gamma}}}\binom{\gamma}{\beta}\left(e^{\frac{s}{2}}\partial^{\gamma-\beta}(SJN)\cdot\partial_1\partial^\beta U+e^{-\frac{s}{2}}\partial^{\gamma-\beta}S\partial_2\partial^\beta U_2\right)\\
			&-2\beta_3\beta_\tau e^{\frac{s}{2}}\partial_1 U_j[\partial^\gamma,JN_j]S-\beta_\tau e^{\frac{s}{2}}\partial^\gamma\left(2\beta_1V\cdot JN-\frac{\partial_tf}{1+f_{x_1}}\right)\partial_1S-2\beta_1\beta_\tau e^{-\frac{s}{2}}\partial^\gamma V_2\partial_2S,
		\end{split}\\
		\begin{split}
			F_S^{(\gamma,U)}=&-2\beta_3\beta_\tau \gamma_2e^{\frac{s}{2}}\partial_2(SJN)\cdot\partial_1\partial^{\gamma-e_2} U+\beta_\tau(\beta_1+\beta_3\gamma_1)e^{\frac{s}{2}}JN\cdot\partial^\gamma U\partial_1Z\\
			&-2\beta_3\beta_\tau e^{-\frac{s}{2}}\sum_{\substack{|\beta|=|\gamma|-1\\\beta\le\gamma}}\binom{\gamma}{\beta}\partial^{\gamma-\beta}S\partial_2\partial^\beta U_2-2\beta_1\beta_\tau e^{-\frac{s}{2}}\partial^\gamma U_2\partial_2S-2\beta_3\beta_\tau\gamma_1 e^{\frac{s}{2}}S\partial^\gamma U_j\partial_1(JN_j),
		\end{split}\\
		\begin{split}
			F_S^{(\gamma-1,U)}=&-2\beta_1\beta_\tau e^{\frac{s}{2}}\partial_1S[\partial^\gamma,JN_j]U_j.
		\end{split}
	\end{align}
\end{subequations}

\subsection{Estimates for forcing terms}
\begin{lemma}
	Let $k\gg1$ and $\delta\in(0,\frac{1}{32}]$, $\lambda=\frac{\delta^2}{12k^2}$, then for $\varepsilon\ll1$ we have
	\begin{subequations}
		\begin{align}
			\begin{split}
				2\sum_{|\gamma|=k}\lambda^{\gamma_2}\int_{\mathbb{R}^3}\left|F_{U_i}^{(\gamma)}\partial^\gamma U_i\right|&\le(4+8\delta)E_k^2+e^{-s}M^{4k-4},
				\label{Ui forcing estimate}
			\end{split}\\
			\begin{split}
				2\sum_{|\gamma|=k}\lambda^{\gamma_2}\int_{\mathbb{R}^3}\left|F_{S}^{(\gamma)}\partial^\gamma S\right|&\le(4+8\delta)E_k^2+e^{-s}M^{4k-4}.
				\label{S forcing estimate}
			\end{split}
		\end{align}
	\end{subequations}
\end{lemma}
\begin{proof}
	We begin with (\ref{Ui forcing estimate}). 
	
	We first deal with the term $F_{U_i}^{(\gamma,U)}$ involving the top order derivatives of $U$, this term is decomposed as a sum $F_{U_i,(1)}^{(\gamma,U)}+F_{U_i,(2)}^{(\gamma,U)}+F_{U_i,(3)}^{(\gamma,U)}$. From (\ref{bootstrap assumptions of dynamic variables}), $0<\beta_1,\beta_\tau<1$, and (\ref{estimates of f-dependent functions, inequalities}), we have
	\begin{equation}
		\begin{aligned}
			2\sum_{|\gamma|=k}\lambda^{\gamma_2}\int_{\mathbb{R}^3}\left|F_{U_i,(1)}^{(\gamma)}\partial^\gamma U_i\right|
			\overset{(\ref{Ui top order forcing terms involving Ui})}&{\le}4\beta_1\beta_\tau\sum_{|\gamma|=k}\lambda^{\gamma_2}\left[e^{\frac{s}{2}}(1+\varepsilon^{\frac{3}{4}})\|\partial_1U\|_{L^\infty}+e^{-\frac{2}{2}}\|\partial_2U\|_{L^\infty}\right]\|\partial^\gamma U\|_{L^2}^2\\
			&\le(4+\varepsilon^{\frac{1}{2}})E_k^2.
		\end{aligned}
	\end{equation}
	By (\ref{estimates of transport terms for energy estimate}) and Young's inequality, we can see that
	\begin{equation}
		\begin{aligned}
			2\sum_{|\gamma|=k}\lambda^{\gamma_2}\int_{\mathbb{R}^3}\left|F_{U_i,(2)}^{(\gamma)}\partial^\gamma U_i\right|
			\overset{(\ref{Ui top order forcing terms involving Ui})}&{\le}2\sum_{|\gamma|=k}\lambda^{\gamma_2}\gamma_2\|\partial_2g_A\|_{L^\infty(\mathcal{X}(s))}\|\partial_1\partial^{\gamma-e_2}U_i\|_{L^2}\|\partial^\gamma U_i\|_{L^2}\\
			&\le2\sum_{|\gamma|=k}\left(\frac{\gamma_2^2}{\delta}\lambda^{\gamma_2+1}\|\partial^\gamma U\|_{L^2}^2+\mathbbm{1}_{\gamma_2>0}\delta\lambda^{\gamma_2-1}\|\partial_1\partial^{\gamma-e_2}U\|_{L^2}^2\right)\\
			&\le\lambda\frac{2k^2}{\delta}E_k^2+2\delta E_k^2
			\overset{\lambda=\frac{\delta^2}{12k^2}}{\le}3\delta E_k^2,
		\end{aligned}
	\end{equation}
	and 
	\begin{equation}
		\begin{aligned}
			2\sum_{|\gamma|=k}\lambda^{\gamma_2}\int_{\mathbb{R}^3}\left|F_{U_i,(3)}^{(\gamma)}\partial^\gamma U_i\right|\overset{(\ref{Ui top order forcing terms involving Ui})}&{\lesssim}\sum_{|\gamma|=k}\int\sum_{\substack{|\beta|=|\gamma|-1\\\beta\le\gamma}}|\partial^{\gamma-\beta}h_A||\partial_2\partial^\beta U||\partial^\gamma U|\\
			&\lesssim\varepsilon\sum_{|\gamma|=k}\ \sum_{\substack{|\beta|=|\gamma|-1\\\beta\le\gamma}}\left(\|\partial^\gamma U\|_{L^2}^2+\|\partial_2\partial^\beta U\|_{L^2}^2\right)\le\varepsilon^{\frac{1}{2}}E_k^2.
		\end{aligned}
	\end{equation}
	Combining these three estimates, we have
	\begin{equation}
		2\sum_{|\gamma|=k}\lambda^{\gamma_2}\int_{\mathbb{R}^3}\left|F_{U_i}^{(\gamma,U)}\partial^\gamma U_i\right|\le(4+3\delta+\varepsilon^{\frac{1}{2}})E_k^2.
	\end{equation}

	Next we deal with the forcing terms $F_{U_i}^{(\gamma-1,U)}$ involving lower order derivatives of $U$. We decompose its first part as $F_{U_i,(1)}^{(\gamma-1,U)}=I_{i1}+I_{i2}+I_{i3}$ where
	\begin{equation}
		\begin{aligned}
			&I_{i1}=-\sum_{\substack{1\le|\beta|\le|\gamma|-2\\\beta\le\gamma}}\binom{\gamma}{\beta}\partial^{\gamma-\beta}g_A\partial^\beta\partial_1(U\cdot NN_i),\\
			&I_{i2}=-\sum_{\substack{1\le|\beta|\le|\gamma|-2\\\beta\le\gamma}}\binom{\gamma}{\beta}\partial^{\gamma-\beta}g_A\partial^\beta\partial_1(AT_i),\\
			&I_{i3}=-\sum_{\substack{1\le|\beta|\le|\gamma|-2\\\beta\le\gamma}}\binom{\gamma}{\beta}\partial^{\gamma-\beta}h_A\partial^\beta\partial_2U_i.
		\end{aligned}
	\end{equation}
	Since $D(U\cdot N)$ is supported in $\mathcal{X}(,)$, we introduce a positive cut-off function $\tilde{\theta}\in C_c(5\mathcal{X}(0))$ such that $\tilde{\theta}\equiv1$ on $\mathcal{X}(0)$. Let $\tilde{\theta}_s(y)=\tilde{\theta}(y_1e^{-\frac{3}{2}s},y_2e^{-\frac{s}{2}})$, then $\tilde{\theta}_s\in C_c^\infty(5\mathcal{X}(s))$, $\tilde\theta_s\equiv1$ on $\mathcal{X}(s)$, and
	\begin{equation}
		\|\partial^\gamma\tilde\theta_s\|_{L^\infty}\lesssim\varepsilon^{-\frac{\gamma_1}{2}-\frac{\gamma_2}{6}}e^{-\frac{3}{2}\gamma_1s-\frac{\gamma_2}{2}s}\lesssim e^{-\frac{|\gamma|}{3}s}.
	\end{equation}
	By the interpolation inequality (\ref{interpolation of L2 norm of product}), we have
	\begin{equation}
		\begin{aligned}
			\|I_{i1}\|_{L^2(\mathbb{R}^2)}\lesssim&\left\|D^k\left(\tilde{\theta}_sg_A\right)\right\|_{L^2_y(\mathbb{R}^2)}^a\left\|D^2\left(\tilde{\theta}_sg_A\right)\right\|_{L^q(\mathbb{R}^2)}^{1-a}\|D^k(U\cdot NN)\|_{L^2(\mathbb{R}^2)}^b\|D^2(U\cdot NN)\|_{L^q(\mathbb{R}^2)}^{1-b}.
		\end{aligned}
	\end{equation}
	We estimate each factor. We first bound the $D^2g_A$ term:
	\begin{equation}
		\begin{aligned}
			\left\|D^2\left(\tilde{\theta}_sg_A\right)\right\|_{L^q(\mathbb{R}^2)}\overset{(\ref{estimates of transport terms for energy estimate})}&{\lesssim}M^{\frac{1}{4}}e^{\frac{s}{2}}e^{-\frac{2}{3}s}(\varepsilon^{\frac{2}{3}}e^{2s})^{\frac{1}{q}}+e^{-\frac{s}{3}}(\varepsilon^{\frac{2}{3}}e^{2s})^{\frac{1}{q}}+\|M\eta^{-\frac{1}{6}}+M^2e^{-\frac{s}{2}}\|_{L^q(5\mathcal{X}(s))}\\
			&\lesssim M\|\eta^{-1}\|_{L^{\frac{q}{6}}(\mathbb{R}^2)}^{\frac{1}{6}}+M^2e^{-\frac{s}{6}}\varepsilon^{\frac{2}{3q}}e^{\frac{2}{q}s}{\lesssim} M.
		\end{aligned}
	\end{equation}
	In the last inequality we require $q\ge12$ and use the fact that $(1+|y_1|^{\alpha_1}+\cdots+|y_d|^{\alpha_d})^{-1}\in L^1(R^d)$ as long as $\sum\alpha_i^{-1}<1$. From estimates (\ref{estimates of U dot N,S}) of $U\cdot N$ and estimates (\ref{estimates of f-dependent functions, inequalities}) of $N$, we have
	\begin{equation}
		\|D^2(U\cdot NN)\|_{L^q}\lesssim Me^{-\frac{s}{2}}.
	\end{equation}
	Then, as we did in the proof of lemma \ref{W,Z,A controlled by U,S}, we have
	\begin{equation}
		\begin{aligned}
			\|D^k(U\cdot JN)\|_{L^2(5\mathcal{X}(s))}
			{\lesssim}\|D^kU\|_{L^2(\mathbb{R}^2)}+M^2\varepsilon^{\frac{1}{3}}e^{-\frac{k-3}{3}},
		\end{aligned}
	\end{equation}
	\begin{equation}
		\begin{aligned}
			\|D^mg_A\|_{L^2(5\mathcal{X}(s))}&\lesssim e^\frac{s}{2}\left(\|D^m(U\cdot JN)\|_{L^2(5\mathcal{X}(s))}+\|D^m(V\cdot JN)\|_{L^2(5\mathcal{X}(s))}+\left\|D^m(\frac{\partial_tf}{1+f_{x_1}})
			\right\|_{L^2(5\mathcal{X}(s))}\right)\\
			\overset{m>0}&{\lesssim}e^{\frac{s}{2}}\left(\|D^mU\|_{L^2(\mathbb{R}^2)}+M^2\varepsilon^{\frac{1}{3}}e^{-\frac{m-3}{3}}\right),
		\end{aligned}
	\end{equation}
	\begin{equation}
		\begin{aligned}
			\left\|\partial^\gamma\left(\tilde{\theta}_sg_A\right)\right\|_{L^2(\mathbb{R}^2)}&\lesssim_\gamma\varepsilon^{-\frac{\gamma_1}{2}-\frac{\gamma_2}{6}}e^{-\frac{3}{2}\gamma_1s-\frac{\gamma_2}{2}s}\|g_A\|_{L^\infty}|5\mathcal{X}(s)|^{1/2}+\sum_{\beta<\gamma}\varepsilon^{-\frac{\beta_1}{2}-\frac{\beta_2}{6}}e^{-\frac{3}{2}\beta_1s-\frac{\beta_2}{2}s}\|\partial^{\gamma-\beta}g_A\|_{L^2(5\mathcal{X}(s))}\\
			&\lesssim e^{\frac{s}{2}}\left(\|D^{|\gamma|}U\|_{L^2(\mathbb{R}^2)}+M^2\varepsilon^{\frac{1}{3}}e^{-\frac{|\gamma|-3}{3}s}\right).
		\end{aligned}
	\end{equation}
	For $k\ge5$, we have $a+b\ge\frac{1}{2}$, $\frac{2-a-b}{1-a-b}\le2k-4$. Hence, by taking $M$ to be large enough in terms of $\lambda$ and $k$, we have
	\begin{equation}
		\begin{aligned}
			2\sum_{|\gamma|=k}\lambda^{\gamma_2}\int\left|I_{i1}\partial^\gamma U_i\right|&\lesssim\sum_{|\gamma|=k}\lambda^{\gamma_2}\|D^kU\|_{L^2}\left[\|D^kU\|_{L^2}^{a+b}+\left(M^2\varepsilon^\frac{1}{3} e^{-\frac{k-3}{3}s}\right)^{a+b}\right]M^{2-a-b}e^{\frac{a+b-1}{2}s}\\
			&\lesssim\sum_{|\gamma|=k}\lambda^{\gamma_2}\left(\lambda^{-\frac{k}{2}}E_k\right)^{1+a+b}M^{2-a-b}e^{\frac{a+b-1}{2}s}+\sum_{|\gamma|=k}\lambda^{\gamma_2}M^{2+3a+3b}\varepsilon^{\frac{a+b}{3}}e^{-\frac{a+b}{3}ks+\frac{a+b+1}{2}s}\lambda^{-\frac{k}{2}}E_k\\
			\overset{a+b<1}&{\le}2\delta E_k^2+C(a,b,\delta)e^{-s}M^{\frac{2(2-a-b)}{1-a-b}}\lambda^{-\frac{1+a+b}{1-a-b}k}+C(\delta)M^{10}\varepsilon^{\frac{2}{3}(a+b)}\lambda^{-k}e^{-\frac{2}{3}(a+b)ks+(a+b+1)s}\\
			&\le2\delta E_k^2+C(a,b,\delta)e^{-s}M^{4k-8}\lambda^{-\frac{1+a+b}{1-a-b}k}\le2\delta E_k^2+e^{-s}M^{4k-6}.
		\end{aligned}
	\end{equation}
	Next, we estimate the $L^2$ norm of $I_{i2}$:
	\begin{equation}
		\begin{aligned}
			\|I_{i2}\|_{L^2}&\lesssim e^{\frac{s}{2}}\sum_{j=1}^{k-2}\|D^{k-j}(U\cdot JN)D^j\partial_1(AT)\|_{L^2}+e^{\frac{s}{2}}\sum_{\substack{1\le|\beta|\le|\gamma|-2\\\beta\le\gamma}}\left(|\partial^{\gamma-\beta}(V\cdot JN)|+\left|\partial^{\gamma-\beta}\frac{\partial_t f}{1+f_{x_1}}\right|\right)\|\partial^\beta\partial_1(AT)\|_{L^2}\\
			&=I_{i2,1}+I_{i2,2}.
		\end{aligned}
	\end{equation}
	First, for $I_{i2,1}$, we have that 
	\begin{equation}
		\begin{aligned}
			I_{i2,1}\overset{\text{Hölder}}&{\lesssim}e^{\frac{s}{2}}\sum_{j=1}^{k-2}\|D^{k-j-1}D(\tilde{\theta}_sU\cdot JN)\|_{L^{\frac{2(k-1)}{k-1-j}}(\mathbb{R}^2)}\|D^j\partial_1(AT)\|_{L^{\frac{2(k-1)}{j}}}\\
			\overset{(\text{\ref{special case 1 of lemma A.1}})}&{\lesssim}e^{\frac{s}{2}}\sum_{j=1}^{k-2}\|D(\tilde{\theta}_sU\cdot JN)\|_{\dot H^{k-1}}^{\frac{k-j-1}{k-1}}\|D(\tilde{\theta}_sU\cdot JN)\|_{L^\infty}^{\frac{j}{k-1}}\|\partial_1(AT)\|_{\dot H^{k-1}}^{\frac{j}{k-1}}\|\partial_1(AT)\|_{L^\infty}^{\frac{k-j-1}{k-1}}\\
			&\lesssim e^{\frac{s}{2}}\sum_{j=1}^{k-2}\left(\|D^kU\|_{L^2}+M^2\varepsilon^{\frac{1}{3}} e^{-\frac{k-3}{3}s}\right)^{\frac{k-j-1}{k-1}}(Me^{-\frac{s}{2}})^{-\frac{j}{k-1}}\left(\|D^kA\|_{L^2}+M^4\varepsilon e^{-\frac{k-3}{3}s}\right)^{\frac{j}{k-1}}\left(Me^{-\frac{3}{2}s}\right)^{\frac{k-j-1}{k-1}}\\
			&\lesssim M^{\frac{1}{k-1}}e^{-\frac{1}{k-1}s}\left(\lambda^{-\frac{k}{2}}E_k+M^2\varepsilon^{\frac{1}{3}} e^{-\frac{k-3}{3}s}\right).
		\end{aligned}
	\end{equation}
	Then
	\begin{equation}
		\begin{aligned}
			I_{i2,2}&\lesssim e^{\frac{s}{2}}\sum_{\substack{1\le|\beta|\le|\gamma|-2\\\beta\le\gamma}}M^2\varepsilon^{\frac{1}{3}-\frac{\gamma_1-\beta_1}{2}-\frac{\gamma_2-\beta_2}{6}}e^{-\frac{3}{2}(\gamma_1-\beta_1)s-\frac{1}{2}(\gamma_2-\beta_2)s}(\varepsilon^{\frac{1}{6}}e^{\frac{s}{2}})^{|\gamma|-|\beta|-1}\|D^k(AT)\|_{L^2}\\
			&\lesssim \sum_{\substack{1\le|\beta|\le|\gamma|-2\\\beta\le\gamma}}M^2\varepsilon^{\frac{1}{6}-\frac{\gamma_1-\beta_1}{3}}e^{-(\gamma_1-\beta_1)s}\left(\|D^kA\|_{L^2}+M^\frac{9}{4}\varepsilon^\frac{3}{2}e^{-\frac{k-3}{3}s}\right)\\
			&\lesssim M^2\varepsilon^{\frac{1}{6}}\left(\lambda^{-\frac{k}{2}}E_k+M^\frac{9}{4}\varepsilon^\frac{3}{2}e^{-\frac{k-3}{3}s}\right).
		\end{aligned}
	\end{equation}
	Hence we have
	\begin{equation}
		\begin{aligned}
			2\sum_{|\gamma|=k}\lambda^{\gamma_2}\int\left|I_{i2}\partial^\gamma U_i\right|&\overset{k\ge7}{\lesssim}\lambda^{-\frac{k}{2}}E_kM^{\frac{1}{k-1}}e^{-\frac{1}{k-1}s}\left(\lambda^{-\frac{k}{2}}E_k+M^2\varepsilon^\frac{1}{3} e^{-\frac{k-3}{3}s}\right)\\
			&\lesssim(M\varepsilon)^{\frac{1}{k-1}}\lambda^{-k}E_k^2+M^{\frac{1}{k-1}}e^{-\frac{1}{k-1}s}M^6\varepsilon^\frac{2}{3}e^{-\frac{2}{3}(k-3)s}\le\varepsilon^{\frac{1}{k}}E_k^2+e^{-s}.
		\end{aligned}
	\end{equation}
	Similar as the estimate of $I_{i2}$, we can estimate $I_{i3}$:
	\begin{equation}
		2\sum_{|\gamma|=k}\lambda^{\gamma_2}\int\left|I_{i3}\partial^\gamma U_i\right|\le\varepsilon^{\frac{1}{2}}E_k^2+e^{-s}.
	\end{equation}
	Summing up these estimates, we obtain
	\begin{equation}
		2\sum_{|\gamma|=k}\lambda^{\gamma_2}\int\left|F_{U_i,(1)}^{(\gamma-1,U)}\partial^\gamma U_i\right|\le(2\delta+\varepsilon^{\frac{1}{2}}+\varepsilon^{\frac{1}{k}})E_k^2+e^{-s}M^{4k-5}.
	\end{equation}
	Now we turn to the estimate of $F_{U_i,(2)}^{(\gamma-1,U)}$. Using the same method in the proof of lemma \ref{W,Z,A controlled by U,S}, we have
	\begin{equation}
		\left\|[\partial^\gamma,JN]U\right\|_{L^2}\le\varepsilon^{\frac{1}{2}}\|D^kU\|_{L^2}+\varepsilon e^{-\left(\gamma_1+\frac{\gamma_2}{3}-1\right)s}.
	\end{equation}
	Thus, if we choose $\varepsilon$ small enough in terms of $\lambda$ and $k$, we have
	\begin{equation}
		\begin{aligned}
			2\sum_{|\gamma|=k}\lambda^{\gamma_2}\int\left|F_{U_i,(2)}^{(\gamma-1,U)}\partial^\gamma U_i\right|&\lesssim\sum_{|\gamma|=k}e^{\frac{s}{2}}\left\|[\partial^\gamma,JN]U\right\|_{L^2}\|\partial^\gamma U\|_{L^2}\|\partial_1U\|_{L^\infty}\\
			\overset{(\text{\ref{estimates of U,S}})}&{\lesssim} e^{\frac{s}{2}}\left(\varepsilon^{\frac{1}{2}}\|D^kU\|_{L^2}+\varepsilon e^{-\frac{k-3}{3}s}\right)\|D^kU\|_{L^2}e^{-\frac{s}{2}}\\
			&\lesssim \lambda^{-k}\varepsilon^{\frac{1}{2}}E_k^2+\varepsilon\lambda^{-\frac{k}{2}}E_ke^{-\frac{k-3}{3}s}\\
			&\lesssim\lambda^{-k}\varepsilon^{\frac{1}{2}}E_k^2+\varepsilon^{\frac{1}{2}}e^{-\frac{2(k-3)}{3}s}\le\varepsilon^{\frac{1}{4}}E_k^2+e^{-s}.
		\end{aligned}
	\end{equation}
	From the estimates $(\ref{estimate of of V})$(\ref{estimates of f-dependent functions, inequalities}) of $V$ and $J$,$N$, we can see that
	\begin{equation}
		\left|\partial^\gamma(V\cdot JN)\right|+\left|\partial^\gamma\frac{\partial_tf}{1+f_{x_1}}\right|\lesssim M^2\varepsilon^{\frac{1}{3}}e^{-\left(\gamma_1+\frac{\gamma_2}{3}\right)s}.
	\end{equation}
	Therefore, we have
	\begin{equation}
		\begin{aligned}
			2\sum_{|\gamma|=k}\lambda^{\gamma_2}\int\left|F_{U_i,(3)}^{(\gamma-1,U)}\partial^\gamma U_i\right|&\lesssim e^{\frac{s}{2}}\sum_{|\gamma|=k}M^2\varepsilon^{\frac{1}{3}}e^{-\left(\gamma_1+\frac{\gamma_2}{3}\right)s}\|\partial_1U\|_{L^\infty}\|\partial^\gamma U\|_{L^2}|\mathcal{X}(s)|^{\frac{1}{2}}\\
			&\lesssim M^2\varepsilon^{\frac{2}{3}}e^{-\frac{k-3}{3}s}\|D^kU\|_{L^2}\lesssim\varepsilon^{\frac{2}{3}}\|D^kU\|_{L^2}^2+M^4\varepsilon^{\frac{2}{3}}e^{-\frac{2(k-3)}{3}s}\\
			&\le\varepsilon^{\frac{1}{2}}E_k^2+e^{-s}.
		\end{aligned}
	\end{equation}
	The estimate of $F_{U_i,(4)}^{(\gamma-1,U)}$ is much the same, we have
	\begin{equation}
		2\sum_{|\gamma|=k}\lambda^{\gamma_2}\int\left|F_{U_i,(4)}^{(\gamma-1,U)}\partial^\gamma U_i\right|\le\varepsilon^{\frac{1}{2}}E_k^2+e^{-s}.
	\end{equation}
	Combining the above estimates, we arrive at
	\begin{equation}
		2\sum_{|\gamma|=k}\lambda^{\gamma_2}\int\left|F_{U_i}^{(\gamma-1,U)}\partial^\gamma U_i\right|\le2(\delta+\varepsilon^{\frac{1}{4}})E_k^2+e^{-s}M^{4k-4}.
	\end{equation}
	Now we estimate the terms involving $k$ order derivatives of $S$. 
	\begin{equation}
		\begin{aligned}
			2\sum_{|\gamma|=k}\lambda^{\gamma_2}\int\left|\left(F_{U_i,(2)}^{(\gamma,S)}+F_{U_i,(4)}^{(\gamma,S)}\right)\partial^\gamma U_i\right|&\lesssim\left(e^{\frac{s}{2}}\|\partial_1Z\|_{L^\infty}+e^{-\frac{s}{2}}\|\partial_2S\|_{L^\infty}\right)\lambda^{-k}E_k^2\le\varepsilon^{\frac{1}{2}}E_k^2.
		\end{aligned}
	\end{equation}
	\begin{equation}
		\begin{aligned}
			2\sum_{|\gamma|=k}\lambda^{\gamma_2}\int\left|F_{U_i,(3)}^{(\gamma,S)}\partial^\gamma U_i\right|&\lesssim\sum_{|\gamma|=k}\lambda^{\gamma_2}e^{-\frac{s}{2}}\sum_{\substack{|\beta|=|\gamma|-1\\\beta\le\gamma}}\|\nabla S\|_{L^\infty}\|\partial_2\partial^\beta S\|_{L^2}\|\partial^\gamma U\|_{L^2}\\
			&\lesssim e^{-s}\lambda^{-k}E_k^2\le\varepsilon^{\frac{1}{2}}E_k^2,
		\end{aligned}
	\end{equation}
	\begin{equation}
		\begin{aligned}
			2\sum_{|\gamma|=k}\lambda^{\gamma_2}\int\left|F_{U_i,(1)}^{(\gamma,S)}\partial^\gamma U_i\right|&\lesssim\sum_{|\gamma|=k}\lambda^{\frac{\gamma_2+1}{2}}\gamma_2\|\partial_2(SJN)\|_{L^\infty}\|\partial^\gamma U\|_{L^2}\|\partial_1\partial^{\gamma-e_2}S\|_{L^2}\lambda^{\frac{\gamma_2-1}{2}}\\
			&\lesssim\sum_{|\gamma|=k}e^{-\frac{s}{2}}\left(\lambda^{\gamma_2+1}\|\partial^\gamma U\|_{L^2}^2+\lambda^{\gamma_2-1}\gamma_2^2\|\partial_1\partial^{\gamma-e_2}S\|_{L^2}^2\right)\le\varepsilon^{\frac{1}{4}}E_k^2,
		\end{aligned}
	\end{equation}
	\begin{equation}
		\begin{aligned}
			2\sum_{|\gamma|=k}\lambda^{\gamma_2}\int\left|F_{U_i,(5)}^{(\gamma,S)}\partial^\gamma U_i\right|&\lesssim e^{\frac{s}{2}}\|\partial_1(JN)\|_{L^\infty}\|S\|_{L^\infty}\lambda^{-k}E_k^2\\
			&\lesssim e^{\frac{s}{2}}M^2\varepsilon^{\frac{5}{6}-\frac{1}{2}}e^{-\frac{3}{2}s}M^{\frac{1}{4}}\lambda^{-k}E_k^2\le\varepsilon E_k^2.
		\end{aligned}
	\end{equation}
	Summing up the above inequalities, we get
	\begin{equation}
		2\sum_{|\gamma|=k}\lambda^{\gamma_2}\int\left|F_{U_i}^{(\gamma,S)}\partial^\gamma U_i\right|\le2\varepsilon^{\frac{1}{4}}E_k^2.
	\end{equation}
	Now we look at the terms involving lower order derivatives of $S$. We decompose $F_{U_i,(1)}^{(\gamma-1,S)}=I_{i1}+I_{i2}+I_{i3}$ where
	\begin{equation}
		\begin{aligned}
			&I_{i1}=-2\beta_3\beta_\tau\sum_{\substack{{1\le|\beta|\le|\gamma|-2}\\{\beta\le\gamma}}}\binom{\gamma}{\beta}e^{\frac{s}{2}}\partial^{\gamma-\beta}((S-S_\infty)JN_i)\partial_1\partial^\beta S,\\
			&I_{i2}=-2\beta_3\beta_\tau\sum_{\substack{{1\le|\beta|\le|\gamma|-2}\\{\beta\le\gamma}}}\binom{\gamma}{\beta}e^{\frac{s}{2}}S_\infty\partial^{\gamma-\beta}(JN_i)\partial_1\partial^\beta S,\\
			&I_{i3}=-2\beta_3\beta_\tau\sum_{\substack{{1\le|\beta|\le|\gamma|-2}\\{\beta\le\gamma}}}\binom{\gamma}{\beta}e^{-\frac{s}{2}}\delta_{i2}\partial^{\gamma-\beta}S\partial_2\partial^\beta S.
		\end{aligned}
	\end{equation}
	For the first part $I_{i_1}$ we have that
	\begin{equation}
		\begin{aligned}
			2\sum_{|\gamma|=k}\lambda^{\gamma_2}\int\left|I_{i1}\partial^\gamma U_i\right|&\lesssim e^{\frac{s}{2}}\|D^kU\|_{L^2}\sum_{j=1}^{k-2}\left\|D^{k-1-(j-1)}((S-S_\infty)JN)D^{j-1}D^2S\right\|_{L^2}\\
			&\lesssim e^{\frac{s}{2}}\|D^kU\|_{L^2}\sum_{j=1}^{k-2}\|D^k((S-S_\infty)JN)\|_{L^2}^a\|D^2((S-S_\infty)JN)\|_{L^q}^{1-a}\|D^kS\|_{L^2}^b\|D^2S\|_{L^q}^{1-b}\\
		\end{aligned}
	\end{equation}
	As before we use Leibniz rule, estimates (\ref{estimates of f-dependent functions, inequalities}) of $J$,$N$ and the Poincaré inequality in $y_2$ direction to deduce that
	\begin{equation}
		\|D^k((S-S_\infty)JN)\|_{L^2(\mathbb{R}^2)}\lesssim\|D^kS\|_{L^2},\ |D^2(JN)|\lesssim\varepsilon^{\frac{1}{4}}e^{-s},\ \|D^2((S-S_\infty)JN)\|_{L^q(\mathbb{R}^2)}\lesssim Me^{-\frac{s}{2}}.
	\end{equation}
	In the last inequality we used the fact that $q>4\Rightarrow\|\eta^{-1}\|_{L^{\frac{q}{6}}(\mathbb{R}^2)}<\infty$. Thus we have
	\begin{equation}
		\begin{aligned}
			2\sum_{|\gamma|=k}\lambda^{\gamma_2}\int\left|I_{i1}\partial^\gamma U_i\right|&\lesssim e^{\frac{s}{2}}\|D^kU\|_{L^2}\sum_{j=1}^{k-2}\|D^kS\|_{L^2}^{a+b}\left(Me^{-\frac{s}{2}}\right)^{2-a-b}\\
			&\lesssim\sum_{j=1}^{k-2}\lambda^{-\frac{k}{2}(1+a+b)}M^{2-a-b}e^{-\frac{1-a-b}{2}s}E_k^{1+a+b}\\
			&\le\sum_{j=1}^{k-2}\left(\delta E_k^2+C(\delta)\lambda^{-\frac{2k(1+a+b)}{2(1-a-b)}}M^{\frac{2(2-a-b)}{1-a-b}}e^{-s}\right)\le\delta E_k^2+e^{-s}M^{4k-6}
		\end{aligned}
	\end{equation}
	$I_{i2}$ is estimated as
	\begin{equation}
		\begin{aligned}
			\|I_{i2}\|_{L^2}&\lesssim\sum_{\substack{{1\le|\beta|\le|\gamma|-2}\\{\beta\le\gamma}}}e^{\frac{s}{2}}M^3\varepsilon^{\frac{5}{6}-\frac{\gamma_1-\beta_1}{2}-\frac{\gamma_2-\beta_2}{6}}e^{-\frac{3}{2}(\gamma_1-\beta_1)s-\frac{1}{2}(\gamma_2-\beta_2)s}\cdot\left(\varepsilon^{\frac{1}{6}}e^{\frac{s}{2}}\right)^{k-1-|\beta|}\|D^kS\|_{L^2}\\
			\overset{|\gamma|=k}&{\lesssim}M^3\varepsilon^{\frac{2}{3}}\|D^kS\|_{L^2}
		\end{aligned}
	\end{equation}
	And $I_{i3}$ is estimated as
	\begin{equation}
		\begin{aligned}
			2\sum_{|\gamma|=k}\lambda^{\gamma_2}\int\left|I_{i3}\partial^\gamma U_i\right|&\lesssim\sum_{|\gamma|=k}e^{-\frac{s}{2}}\|D^kU\|_{L^2}\sum_{j=1}^{k-2}\|S\|_{\dot H^k}^{\frac{k-1-j}{k-1}}\|DS\|_{L^\infty}^{\frac{j}{k-1}}\|S\|_{\dot H^k}^{\frac{j}{k-1}}\|\partial_2S\|_{L^\infty}^{\frac{k-1-j}{k-1}}\\
			&\lesssim e^{-\frac{s}{2}}\|U\|_{\dot H^k}\|S\|_{\dot H^k}e^{-\frac{s}{2}}\le\varepsilon^{\frac{1}{2}}E_k^2
		\end{aligned}
	\end{equation}
	Hence, we have
	\begin{equation}
		2\sum_{|\gamma|=k}\lambda^{\gamma_2}\int\left|F_{U_i,(1)}^{(\gamma-1,S)}\partial^\gamma U_i\right|\le(\delta+2\varepsilon^{\frac{1}{2}})E_k^2+e^{-s}M^{4k-6}
	\end{equation}
	Next, we turn to $F_{U_i,(2)}^{(\gamma-1,S)}$. From Leibniz rule we have
	\begin{equation}
		\|[\partial^\gamma,JN_i]S\|_{L^2(\mathcal{X}(s))}\lesssim\varepsilon^{\frac{1}{2}}\|D^kS\|_{L^2(\mathbb{R}^2)}+\varepsilon e^{-\left(\gamma_1+\frac{\gamma_2}{3}-1\right)s},
	\end{equation}
	and
	\begin{equation}
		\begin{aligned}
			2\sum_{|\gamma|=k}\lambda^{\gamma_2}\int\left|F_{U_i,(2)}^{(\gamma-1,S)}\partial^\gamma U_i\right|&\lesssim\sum_{|\gamma|=k}e^{\frac{s}{2}}\left\|[\partial^\gamma,JN_i]S\right\|_{L^2(\mathcal{X}(s))}\|\partial_1S\|_{L^\infty}\|D^kU_i\|_{L^2}\\
			&\lesssim e^{\frac{s}{2}}\left(\varepsilon^{\frac{1}{2}}\|D^kS\|_{L^2}+\varepsilon e^{-\left(\gamma_1+\frac{\gamma_2}{3}-1\right)s}\right)e^{-\frac{s}{2}}\|D^kU\|_{L^2}\le\varepsilon^{\frac{1}{4}}E_k^2+e^{-s}.
		\end{aligned}
	\end{equation}
	Thus we have
	\begin{equation}
		2\sum_{|\gamma|=k}\lambda^{\gamma_2}\int_{\mathbb{R}^3}\left|F_{U_i}^{(\gamma-1,S)}\partial^\gamma U_i\right|\le(\delta+2\varepsilon^{\frac{1}{2}}+\varepsilon^{\frac{1}{4}})E_k^2+e^{-s}M^{4k-5}.
	\end{equation}
	Summing all the estimates together leads us to
	\begin{equation}
		2\sum_{|\gamma|=k}\lambda^{\gamma_2}\int_{\mathbb{R}^3}\left|F_{U_i}^{(\gamma)}\partial^\gamma U_i\right|\le\left(4+C\varepsilon^{\frac{1}{4}}+6\delta\right)E_k^2+e^{-s}M^{4k-4}.
	\end{equation}
	The proof of (\ref{S forcing estimate}) is similar. 
\end{proof}

\begin{proof}[Proof of $\dot H^k$ estimates of $U$, $S$]
	We multiply the equations of $\partial^\gamma U_i$, $\partial^\gamma S$ by $\partial^\gamma U_i$, $\partial^\gamma S$ respectively and sum over, then we arrive at
	\begin{subequations}
		\begin{align}
			\begin{split}
				\frac{1}{2}\frac{d}{ds}\|\partial^\gamma U\|_{L^2}^2\le&\frac{1}{2}\int|\partial^\gamma U|^2(\operatorname{div}\mathcal{V}_A-2D_\gamma)+\frac{1}{2}(1+\gamma_1)\beta_3\beta_\tau(1+\varepsilon^{\frac{1}{13}})\left(\|\partial^\gamma S\|_{L^2}^2+\|\partial^\gamma U\|_{L^2}^2\right)\\
				&-2\beta_3\beta_\tau\int S\left(e^{\frac{s}{2}}JN_i\partial_1\partial^\gamma S+e^{-\frac{s}{2}}\delta_{i2}\partial_2\partial^\gamma S\right)\partial^\gamma U_i+\int\left|F_{U_i}^{(\gamma)}\partial^\gamma U_i\right|,
			\end{split}\\
			\begin{split}
				\frac{1}{2}\frac{d}{d s}\left\|\partial^{\gamma} S\right\|_{2}^{2}&\le\frac{1}{2}\int\left|\partial^{\gamma}S\right|^{2}(\operatorname{div}\mathcal{V}_{A}-2D_\gamma)+\frac{1}{2}\beta_{\tau}\left(\beta_{1}+\beta_{3}\gamma_{1}\right)\left(1+\varepsilon^{\frac{1}{13}}\right)\left(\left\|\partial^{r}S\right\|_{2}^{2}+\left\|\partial^{\gamma}U\right\|_{2}^{2}\right)\\
				&-2 \beta_{3} \beta_{\tau} \int S\left(e^{\frac{s}{2}} \partial_{1} \partial^{\gamma} U_{j}JN_{j}+e^{-\frac{s}{2}} \partial_{2} \partial^{\gamma} U_{2}\right)\partial^{\gamma}S+\int\left|F_{S}^{(\gamma)} \partial^{\gamma} S\right|.
			\end{split}
		\end{align}
	\end{subequations}
	Here we used the fact that $|JN\partial_1W|\le|JN||\partial_1W|\le (1+\varepsilon^{\frac{2}{3}})(1+\varepsilon^{\frac{1}{12}})\le1+\varepsilon^{\frac{1}{13}}$. By summing up the above two inequalities and integrating by part, we get
	\begin{equation}
		\begin{aligned}
			&\frac{d}{ds}\left(\|\partial^\gamma U\|_{L^2}^2+\|\partial^\gamma S\|_{L^2}^2\right)+\int\left(2D_\gamma-\operatorname{div}\mathcal{V}_A-\beta_\tau(1+2\gamma_1\beta_3)(1+\varepsilon^{\frac{1}{13}})\right)\left(|\partial^\gamma U|^2+|\partial^\gamma S|^2\right)\\
			&\le2\int\left|F_{U_i}^{(\gamma)}\partial^\gamma U_i\right|+2\int\left|F_{S}^{(\gamma)}\partial^\gamma S\right|+4\beta_3\beta_\tau\int\left[e^{\frac{s}{2}}\partial^\gamma S\partial^\gamma U\cdot\partial_1(SJN)+e^{-\frac{s}{2}}\partial^\gamma S\partial_2U_2\partial_2S\right]\\
			&\le2\int\left|F_{U_i}^{(\gamma)}\partial^\gamma U_i\right|+2\int\left|F_{S}^{(\gamma)}\partial^\gamma S\right|+2\beta_3\beta_\tau(1+2\varepsilon^{\frac{1}{2}})\left(\|\partial^\gamma U\|_{L^2}^2+\|\partial^\gamma S\|_{L^2}^2\right).
		\end{aligned}
		\label{energy estimate partial result}
	\end{equation}
	In the last inequality we used the facts that $|\partial_1(SJN)|\le(1+\varepsilon^{\frac{1}{2}})e^{-\frac{s}{2}}$ and estimate (\ref{estimates of U,S}) of $S$, the first fact can be obtained from (\ref{estimates of U,S}) and estimates (\ref{estimates of f-dependent functions, inequalities}) of $J$, $N$. 
	
	Now we estimate the damping term: 
	\begin{equation}
		\begin{aligned}
			&2D_\gamma-\operatorname{div}\mathcal{V}_A-\beta_\tau(1+2\gamma_1\beta_3)(1+\varepsilon^{\frac{1}{13}})-2\beta_3\beta_\tau(1+2\varepsilon^{\frac{1}{2}})\\
			&\ge|\gamma|+2\gamma_1\left(1+\beta_1\beta_\tau\partial_1(JW)+\partial_1G_A\right)-2-\beta_1\beta_\tau\partial_1(JW)-\partial_1G_A-\partial_2h_A\\
			&\ \ \ \ \ \ \ \ \ \ \ \ -2\beta_3\beta_\tau(1+\varepsilon^{\frac{1}{13}})\gamma_1-\underbrace{\left[\beta_\tau(1+\varepsilon^{\frac{1}{13}})+2\beta_3\beta_\tau(1+2\varepsilon^{\frac{1}{2}})\right]}_{\le3}\\
			\overset{(\ref{estimate of D1W that will appear in damping terms})}&{\ge}|\gamma|+2\gamma_1\left(1-\beta_1\beta_\tau-\beta_3\beta_\tau\right)-6-C\varepsilon^{\frac{1}{13}}\ge k-7.
		\end{aligned}
	\end{equation}
	Multiply (\ref{energy estimate partial result}) by $\lambda^{\gamma_2}$ and take the sum, we have
	\begin{equation}
		\frac{d}{ds}E_k^2+(k-7)E_k^2\le (8+16\delta)E_k^2+e^{-s}M^{4k-3}.
	\end{equation}
	Taking $k\ge 18$ we have
	\begin{equation}
		\frac{d}{ds}E_k^2+2E_k^2\le e^{-s}M^{4k-3},
	\end{equation}
	which results in
	\begin{equation}
		E_k^2(s)\le e^{-2(s-s_0)}E_k^2(s_0)+(1-e^{-(s-s_0)})e^{-s}M^{4k-3}.
	\end{equation}
	By Leibniz rule, we have
	\begin{equation}
		\left\{
		\begin{aligned}
			&\|WN\|_{\dot H^k}\le(1+C\varepsilon^{\frac{1}{2}})\|W\|_{\dot H^k}+CM^2\varepsilon^{\frac{5}{3}}e^{-\left(\frac{k}{3}-\frac{3}{2}\right)s}\\
			&\|AT\|_{\dot H^k},\|ZN\|_{\dot H^k}\le(1+C\varepsilon^{\frac{1}{2}})\|A\text{ or }Z\|_{\dot H^k}+CM^3\varepsilon^{\frac{5}{3}}e^{-\frac{k-3}{3}s},
		\end{aligned}
		\right.
	\end{equation}
	and
	\begin{equation}
		\left\{
		\begin{aligned}
			&\|U\|_{\dot H^k}\le(1+C\varepsilon^{\frac{1}{2}})\left[\frac{1}{2}\left(e^{-\frac{s}{2}}\|W\|_{\dot H^k}+\|Z\|_{\dot H^k}\right)+\|A\|_{\dot H^k}\right]+CM^3\varepsilon^{\frac{5}{3}}e^{-\frac{k-3}{3}s}\\
			&\|S\|_{\dot H^k}\le\frac{1}{2}(1+C\varepsilon^{\frac{1}{2}})\left(e^{-\frac{s}{2}}\|W\|_{\dot H^k}+\|Z\|_{\dot H^k}\right)+CM^3\varepsilon^{\frac{5}{3}}e^{-\frac{k-3}{3}s}.
		\end{aligned}\right.
	\end{equation}
	According to the assumption (\ref{initial sobolev norm}) of $\dot H^k$ norm of $W$,$Z$,$A$, we have
	\begin{equation}
		E_k^2(s_0)\le (2+C\varepsilon^{\frac{1}{2}})\varepsilon.
	\end{equation}
	Thus we finally obtain that
	\begin{equation}
		\lambda^k\left(\|U\|_{\dot H^k}+\|S\|_{\dot H^k}\right)\le E_k^2\le(2+C\varepsilon^{\frac{1}{2}})\varepsilon^{-1}e^{-2s}+M^{4k-3}e^{-s}(1-\varepsilon^{-1}e^{-s}).
	\end{equation}
	This finishes the proof of energy estimate. 
\end{proof}

\subsection{Higher order estimates for $W,Z,A$}
Using the energy estimate, we can further obtain higher order estimates for $W,Z,A$. 
\begin{lemma}
	For $k\gg1$, we have that
	\begin{subequations}
		\begin{align}
			\begin{split}
				&|\partial^\gamma W|\lesssim\left\{
				\begin{aligned}
					&\eta^{-\frac{1}{6}}e^{\frac{s}{2(k-3)}},\ \ \ \ \ \  \gamma_1=0,\ |\gamma|=3\\
					&\eta^{-\frac{1}{3}}e^{\frac{s}{k-3}},\ \ \ \ \ \  \gamma_1>0,\ |\gamma|=3,
				\end{aligned}\right.
			\label{new estimate of W}
			\end{split}\\
			\begin{split}
				&|\partial^\gamma Z|\lesssim\left\{
				\begin{aligned}
					&e^{-\left(\frac{3}{2}-\frac{1}{2(k-3)}\right)s},\ \ \ \gamma_1\ge1,\ |\gamma|=3\\
					&e^{-\left(1-\frac{|\gamma|-1}{2(k-3)}\right)s},\ \ \ \ \  |\gamma|=3,4,5,
				\end{aligned}\right.
			\label{new estimate of Z}
			\end{split}\\
			\begin{split}
				&|\partial^\gamma A|\lesssim\left\{
				\begin{aligned}
					&e^{-\left(\frac{3}{2}-\frac{|\gamma|}{k-2}\right)s},\ \ \ \gamma_1\ge1,\ |\gamma|=2,3\\
					&e^{-\left(1-\frac{|\gamma|-1}{2(k-3)}\right)s},\ \ \ \  |\gamma|=3,4,5.
				\end{aligned}\right.
			\label{new estimate of A}
			\end{split}
		\end{align}
	\end{subequations}
\end{lemma}
\begin{proof}
	The proof is similar to the interpolation in \cite{buckmaster2022formation}, still for the reader's convinience we recap the proof here. 
	
	First we deal with $A$. For $\gamma_1\ge1$, $|\gamma|=2,3$, we have
	\begin{equation}
		\begin{aligned}
			|\partial^\gamma A|&\lesssim\|\partial_1A\|_{\dot H^{k-1}}^{\frac{|\gamma|-1}{k-2}}\|\partial_1A\|_{L^\infty}^{1-\frac{|\gamma|-1}{k-2}}\lesssim(M^{2k}e^{-\frac{s}{2}})^{\frac{|\gamma|-1}{k-2}}(Me^{-\frac{3}{2}s})^{1-\frac{|\gamma|-1}{k-2}}\lesssim M^{2k}e^{-\left(\frac{3}{2}-\frac{|\gamma|-1}{k-2}\right)s}\lesssim e^{-\left(\frac{3}{2}-\frac{|\gamma|}{k-2}\right)s}.
		\end{aligned}
	\end{equation}
For $|\gamma|=3,4,5$, we have
	\begin{equation}
		|\partial^\gamma A|\lesssim\|D^kA\|_{L^2}^{\frac{|\gamma|-2}{k-3}}\|D^2A\|_{L^\infty}^{1-\frac{|\gamma|-2}{k-3}}\lesssim(M^{2k}e^{-\frac{s}{2}})^{\frac{|\gamma|-2}{k-3}}(Me^{-s})^{1-\frac{|\gamma|-2}{k-3}}\lesssim M^{2k}e^{-\left(1-\frac{|\gamma|-2}{2(k-3)}\right)s}\lesssim e^{-\left(1-\frac{|\gamma|-1}{2(k-3)}\right)s}.
	\end{equation}

	Next we estimate $Z$. For $\gamma_1\ge1$, $|\gamma|=3$, we have
	\begin{equation}
		|\partial^\gamma Z|\lesssim\|\partial_1\nabla Z\|_{\dot H^{k-2}}^{\frac{1}{k-3}}\|\partial_1\nabla Z\|_{L^\infty}^{1-\frac{1}{k-3}}\lesssim(M^{2k}e^{-\frac{s}{2}})^{\frac{1}{k-3}}(Me^{-\frac{3}{2}s})^{1-\frac{1}{k-3}}\lesssim M^{2k}e^{-\left(\frac{3}{2}-\frac{1}{k-3}\right)s}\lesssim e^{-\left(\frac{3}{2}-\frac{1}{2(k-3)}\right)}.
	\end{equation}
	For $|\gamma|=3,4,5$, the estimates for $Z$ are the same as $A$. 
	
	Now we turn to $W$. Since $|\gamma|=3$, we can split $\gamma$ as $\gamma=\gamma'+\gamma''$, where $|\gamma'|=1$ and $\gamma''_1=\min(\gamma_1,2)$, then $\eta^\mu\partial^\gamma W=\partial^{\gamma'}(\eta^\mu\partial^{\gamma''}W)-\partial^{\gamma'}(\eta^{\mu})\partial^{\gamma''}W=I_1+I_2$. Let
	\begin{equation}
		\mu=\left\{
		\begin{aligned}
			&\frac{1}{6},\ \ \ \gamma_1=0\\
			&\frac{1}{3},\ \ \ \text{otherwise.}
		\end{aligned}\right.
	\end{equation}
	Note that $|\partial_1(\eta^\mu)|\lesssim\eta^{\mu-\frac{1}{2}}$, $|\partial_2(\eta^\mu)|\lesssim\eta^{\mu-\frac{1}{6}}$, thus when $\gamma_1=0$ we have $|I_2|\lesssim\eta^{\mu-\frac{1}{6}}|\partial_{22}W|\lesssim M$, when $\gamma_1>0$ we have $|I_2|\lesssim M\eta^{-\frac{1}{6}}\lesssim M$. By interpolation and bootstrap assumptions for $W$, we have
	\begin{equation}
		|I_1|\lesssim\|D(\eta^\mu\partial^{\gamma''}W)\|_{L^\infty}\lesssim\|\eta^\mu\partial^{\gamma''}W\|_{\dot H^{k-2}}^{\frac{1}{k-3}}\|\eta^\mu\partial^{\gamma''}W\|_{L^\infty}^{1-\frac{1}{k-3}}\lesssim\ M\|\eta^\mu\partial^{\gamma''}W\|_{\dot H^{k-2}}^{\frac{1}{k-3}}.
	\end{equation}
	We estimate the $\dot H^{k-2}$ norm as follows:
	\begin{equation}
		\begin{aligned}
			\|\eta^\mu\partial^{\gamma''}W\|_{\dot H^{k-2}}&\lesssim\sum_{m=0}^{k-2}\|D^m\partial^{\gamma''}WD^{k-2-m}\eta^\mu\|_{L^2}\lesssim\sum_{m=0}^{k-2}\|D^m\partial^{\gamma''}W\|_{L^{\frac{2(k-1)}{m+1}}}\|D^{k-2-m}\eta^\mu\|_{L^{\frac{2(k-1)}{k-2-m}}(\mathcal{X}(s))}\\
			&\lesssim\sum_{m=0}^{k-2}\|W\|_{\dot H^k}^{\frac{m+1}{k-1}}\|\nabla W\|_{L^\infty}^{1-\frac{m+1}{k-1}}\|D^{k-2-m}\eta^\mu\|_{L^{\frac{2(k-1)}{k-2-m}}(\mathcal{X}(s))}\\
			&\lesssim\sum_{m=0}^{k-2}(M^{2k})^{\frac{m+1}{k-1}}\|D^{k-2-m}\eta^\mu\|_{L^{\frac{2(k-1)}{k-2-m}}(\mathcal{X}(s))}.
		\end{aligned}
	\end{equation} 
	Simple calculation yields $|D^k(\eta^\mu)|\lesssim\eta^{\mu-\frac{k}{6}}$, thus we have that
	\begin{equation}
		\begin{aligned}
			\|D^{k-2-m}\eta^\mu\|_{L^{\frac{2(k-1)}{k-2-m}}(\mathcal{X}(s))}&\lesssim\|\eta^{\mu-\frac{k-2-m}{6}}\|_{L^{\frac{2(k-1)}{k-2-m}}(\mathcal{X}(s))}\lesssim\|\eta^\mu\|_{L^\infty(\mathcal{X}(s))}\|\eta^{-\frac{k-2-m}{6}}\|_{L^{\frac{2(k-1)}{k-2-m}}(\mathcal{X}(s))}\\
			&\lesssim\varepsilon e^{3\mu s}\times\left\{
			\begin{aligned}
				&1,\ \ \ \ \ \ \ \ \ &m=k-2\\
				&\|\eta^{-1}\|_{L^{\frac{k-1}{3}}(\mathcal{X}(s))},\ \ &m<k-2
			\end{aligned}\right.\\
		\overset{k>3}&{\lesssim}\varepsilon^{\mu}e^{3\mu s}.
		\end{aligned}
	\end{equation}
	Consequently, we obtain $|I_1|\lesssim M(M^{2k}\varepsilon^\mu e^{3\mu s})^{\frac{1}{k-3}}\lesssim e^{\frac{3\mu s}{k-3}}$, and $|\eta^\mu\partial^\gamma W|\lesssim e^{\frac{3\mu s}{k-3}}+M\lesssim e^{\frac{3\mu s}{k-3}}$. 
\end{proof}

\section{Constraints and evolution of modulation variables}
In this section we close the bootstrap argument for the modulation variables $\xi,n,\phi,\tau,\kappa$. The equation of these variables are deduced from the constraints that we impose on $W$. 
\subsection{Constraints}
We impose constraints on $W$ and its derivatives up to second order at the origin, i.e.
\begin{equation}
	W(0,s)=\ovl{W}(0)=0,\ \ \ 
	\nabla W(0,s)=\nabla\ovl{W}(0)=(-1,0)^T,\ \ \ 
	\nabla^2 W(0,s)=\nabla^2\ovl{W}(0)=
	\begin{pmatrix}
		0&0\\
		0&0
	\end{pmatrix}.
\end{equation}

It is possible to impose these constraints. In fact, as long as the initial data $W(y,-\log\varepsilon)$ satisfies these constraints, we can choose 6 modulation variables $\xi$, $n_2$, $\phi$, $\tau$, $\kappa$ in a continuous manner with respect to time in terms of $w(x,t)$, ensuring that $W(y,s)$ still satisfies these constraints. 

\subsection{The functions $G_W$, $h_W$, $F_W$ and their derivatives, evaluated at $y=0$}
In a neighborhood of the origin, $\tilde f$ reduces to $\tilde f(\tilde x,t)=\frac{1}{2}\phi\tilde x_2^2$, and as a consequence in a neighborhood of 0, $f(x,t)=\frac{1}{2}\phi x_2^2$. Note that any derivatives with respect to $x_1$ or $\tilde{x}_1 $ of those function vanish at the origin, we can conveniently evaluated the $f$-related functions at the origin:  
\begin{subequations}
	\label{evaluation of f-related functions at 0}
	\begin{align}
		\begin{split}
			&\tilde f^0=0,\ \ \partial_{\tilde x_2}\tilde f^0=0,\ \ \partial_{\tilde x_2}^2\tilde f^0=0;
		\end{split}\\
		\begin{split}
			&(\partial_t)_{\tilde x}\tilde f^0=0,\ \ \partial_{\tilde x_2}(\partial_t)_{\tilde x}\tilde f^0=0,\ \ \partial_{\tilde x_2}^2(\partial_t)_{\tilde x}\tilde f^0=\dot\phi;
		\end{split}\\
		\begin{split}
			&f^0=0,\ \ \partial_{x_2}f^0=0,\ \ \partial_{x_2}^2f^0=0;
		\end{split}\\
		\begin{split}
			&J^0=0,\ \ \partial_{x_2}J^0=0,\ \ \partial_{x_2}^2J^0=\phi^2,\ \ \partial_{x_2}^3J^0=0;
		\end{split}\\
		\begin{split}
			N^0=(1,0)^T,\ \ \partial_{x_2}N^0=(0,-\phi)^T,\ \ \partial_{x_2}^2N^0=(-\phi^2,0)^T,\ \ \partial_{x_2}^3N^0=(0,2\phi^3)^T;
		\end{split}\\
		\begin{split}
			T^0=(0,1)^T,\ \ \partial_{x_2}T^0=(\phi,0)^T,\ \ \partial_{x_2}^2T^0=(0,\phi^2)^T,\ \ \partial_{x_2}^3T^0=(-2\phi^3,0)^T;
		\end{split}\\
		\begin{split}
			&(\partial_t)_xf^0=0,\ \ \partial_{x_2}(\partial_t)_x f^0=0,\ \ \partial_{x_2}^0(\partial_t)_x f^0=\dot\phi;
		\end{split}\\
		\begin{split}
			&\partial_tJ^0=0,\ \ \partial_{x_2}\partial_tJ^0=0,\ \ \partial_{x_2}^2\partial_tJ^0=2\phi\dot\phi;
		\end{split}\\
		\begin{split}
			\partial_tN^0=(0,0)^T,\ \ \partial_{x_2}\partial_tN^0=(0,-\dot\phi)^T,\ \ \partial_{x_2}^2\partial_tN^0=(-2\phi\dot\phi,0)^T;
		\end{split}\\
		\begin{split}
			\partial_tT^0=(0,0)^T,\ \ \partial_{x_2}\partial_tT^0=(\dot\phi,0)^T,\ \ \partial_{x_2}^2\partial_tT^0=(0,-2\phi\dot\phi)^T.
		\end{split}
	\end{align}
\end{subequations}
By the definition of $V$, we have
\begin{subequations}
	\label{evaluaton of V at 0}
	\begin{align}
		\begin{split}
			V_i^0=-\frac{1+\alpha}{2}R_{ji}\dot\xi_j,
		\end{split}\\
		\begin{split}
			\partial_1V^0=\frac{1+\alpha}{2}e^{-\frac{3}{2}s}(0,Q_{21})^T,\ \ \partial_2V^0=\frac{1+\alpha}{2}e^{-\frac{s}{2}}(Q_{12},0)^T,
		\end{split}\\
		\begin{split}
			\partial_{11}V^0=\frac{1+\alpha}{2}\phi e^{-3s}(0,Q_{21})^T,\ \ \partial_{12}V^0=0,\ \ \partial_{22}V^0=0.
		\end{split}
	\end{align}
\end{subequations}

From the definition of $G_W$ and (\ref{evaluation of f-related functions at 0})(\ref{evaluaton of V at 0}), we have 
\begin{subequations}
	\begin{align}
		\begin{split}
			\frac {1} {\beta_{\tau}}G_W^0 &= e^{\frac s 2 }[\kappa+\beta_2Z^0-(1+\alpha)\beta_1R_{j1}\dot\xi_j],
			\label{evaluation of G_W at 0}
		\end{split}\\
		\begin{split}
			\frac {1} {\beta_{\tau}}\partial_1 G_W^0 &=\beta_2e^{\frac {s} {2} }\partial_1 Z^0,
			\label{evaluation of D1GW at 0}
		\end{split}\\
		\begin{split}
			\frac {1} {\beta_{\tau}}\partial_2 G_W^0 &=\beta_2e^{\frac {s} {2} }\partial_2 Z^0+(1+\alpha)\beta_1Q_{12}+\beta_1(1+\alpha)\phi R_{j2}\dot\xi_j,
			\label{evaluation of D2GW at 0}
		\end{split}\\
		\begin{split}
			\frac {1} {\beta_{\tau}}\partial_{11} G_W^0 &=\beta_2e^{\frac {s} {2} }\partial_{11} Z^0,
			\label{evaluation of D11GW at 0}
		\end{split}\\
		\begin{split}
			\frac {1} {\beta_{\tau}}\partial_{12} G_W^0 &=\beta_2e^{\frac {s} {2} }\partial_{12} Z^0-(1+\alpha)\beta_1 e^{-\frac 3 2 s}\phi Q_{21},
			\label{evaluation of D12GW at 0}
		\end{split}\\
		\begin{split}
			\frac {1} {\beta_{\tau}}\partial_{22} G_W^0 &= -\phi e^{-\frac {s} {2}}+\phi^2e^{-s}\frac {G_W^0}{\beta_{\tau}}+e^{-\frac {s} {2} }\beta_2\partial_{22}Z^0-(1+\alpha)\beta_1 \phi^2e^{-\frac {s} {2}} R_{j1}\dot \xi_j.
			\label{evaluation of D22GW at 0}
		\end{split}
	\end{align}
\end{subequations}
Similarly for $h_W$, we have
\begin{equation}
	\frac 1 {\beta_{\tau}}h_W^0 =\beta_1 e ^{-\frac s 2} \left(2A^0-(1+\alpha)R_{j2}\dot\xi_j\right)
	\label{evaluation of h_W at 0}.
\end{equation}
And for the forcing terms, we also insert the above evaluation to the definition (\ref{forcing terms of W,Z,A}), then we have
\begin{subequations}
	\begin{align}
		\begin{split}
			F_W^0 =&-\beta_3\beta_\tau(\kappa-Z^0)\partial_2A^0+\beta_\tau e^{-\frac{s}{2}}Q_{12}A^0\\
			& -2\phi \beta_1 \beta_{\tau} e ^{-\frac s 2}\left(-\frac {1+\alpha} 2 R_{j2}\dot\xi_j+ A^0\right)A^0+\frac 1 2 \phi \beta_3 \beta_{\tau}e^{-\frac s 2} (\kappa+Z^0)(Z^0- \kappa),
			\label{evaluation of FW at 0}
		\end{split}\\
		\begin{split}
			\partial_1 F_W^0 =& \beta_3\beta_{\tau}\partial_2 A^0(\partial_1Z^0+ e^{-\frac s 2})-\beta_3 \beta_{\tau} \partial_{12} A^0(\kappa - Z^0)+\beta_{\tau} e^{\frac s 2 } Q_{12}\partial_1 A^0 \\
			&-\phi \partial_1 A h_W^0 - \phi \beta_2 \beta_{\tau} e^{-\frac s 2} A^0\left((1+\alpha) Q_{21} e^{-\frac 3 2 s}+ 2\partial_1 A^0\right) \\
			&-\frac 1 2 \phi \beta_3 \beta_{\tau} e ^{-\frac s 2}(e^{-\frac s 2 }+\partial_1 Z^0)(\kappa +Z^0 )+ \frac 1 2 \phi \beta_3 \beta_{\tau}e ^{-\frac s 2}(\kappa -Z^0 )(\partial_1 Z^0- e^{-\frac s 2}),
			\label{evaluation of D1FW at 0}
		\end{split}\\
		\begin{split}
			\partial_2 F_W^0 =&-\beta_3 \beta_{\tau} (\kappa - Z^0) \partial_{22} A^0+ \beta_3 \beta_{\tau}\partial_2 Z^0 \partial_2  A^0 -\dot \phi \beta_{\tau} e^{-s} A^0+ \beta_{\tau} e^{-\frac s 2} Q_{12} \partial_2 A^0\\
			&-\phi \beta_3 \beta_{\tau} e^{-\frac s 2} \partial_2 Z^0 Z^0+ \phi^2  \beta_3 \beta_{\tau} e^{-s} A^0 (\kappa - Z^0)\\ &-\phi \beta_1 \beta_{\tau} e^{-\frac s 2} A^0\left(2\partial_2 A^0 -\phi e^{-\frac s 2}(\kappa+ Z^0)\right)-\phi \partial_2 A^0 h_W^0,
			\label{evaluation of D2FW at 0}
		\end{split}\\
		\begin{split}
			\partial_{11} F_W^0 =&2\beta_3 \beta_{\tau} (e^{-\frac s 2}+\partial_{12} Z^0) \partial_{12} A^0- \beta_3 \beta_{\tau}(\kappa - Z^0) \partial_{112} A^0 + \beta_{\tau} e^{-\frac s 2}Q_{12} \partial_{11} A^0\\
			&-2\phi \beta_1 \beta_{\tau} e^{-\frac s 2} \partial_{11} A^0\left(2A^0-\frac {1+\alpha} {2} R_{j2}\dot \xi_j\right)-4\phi \beta_1 \beta_{\tau} e^{-\frac s 2} \partial_{1} A^0\left(\frac {1+\alpha} {2} Q_{21} e^{-\frac 3 2 s} +\partial_1 A^0\right)\\
			&- \phi \beta_3 \beta_{\tau} e^{-\frac s 2}\left({(\partial_1 Z^0)}^2 -e^{-s} +Z^0\partial_{11} Z^0\right)+ \beta_3 \beta_{\tau} \partial_{11} Z^0 \partial_2 A^0,
			\label{evaluation of D11FW at 0}
		\end{split}\\
		\begin{split}
			\partial_{12} F_W^0 =&-2\beta_3 \beta_{\tau}(\kappa - Z^0) \partial_{122} A^0+ \beta_3 \beta_{\tau}\partial_{12} Z^0 \partial_{2} A^0 + \beta_3 \beta_{\tau}\partial_{2} Z^0 \partial_{12} A^0\\
			&+\beta_3\beta_{\tau} (e^{-\frac s 2}+ \partial_{1}Z^0)\partial_{22}A^0-\beta_{\tau} \dot\phi e^{-s} \partial_{1} A^0+ \beta_{\tau} e^{-\frac s 2} Q_{12} \partial_{12}A^0 \\
			&- \phi \beta_3 \beta_{\tau} e^{-\frac s 2}\left(\partial_{12} Z^0Z^0 +\partial_{1}Z^0\partial_{2} Z^0\right)+ \phi^2 \beta_3 \beta_{\tau} e^{-s}\left((\kappa - Z^0) \partial_{1} A^0-(e^{-\frac s 2}+ \partial_{1}Z^0)A^0\right)\\
			&- 2\phi \beta_1 \beta_{\tau} e^{-\frac s 2}\left[\partial_{1}A^0\partial_{2} A^0+\left(\frac {1+\alpha}{2} Q_{21} e^{-\frac 3 2 s}+\partial_1 A^0\right)\partial_2 A^0 +A^0\partial_{12} A^0\right]\\
			&-\phi \partial_{12} A^0 h_W^0+ \phi^2 \beta_1 \beta_{\tau} e^{-s}\left[\partial_1 A^0 (\kappa + Z^0)+ A^0(\partial_1 Z^0- e^{-\frac s 2})\right],
			\label{evaluation of D12FW at 0}
		\end{split}\\
		\begin{split}
			\partial_{22} F_W^0 =&\beta_3 \beta_{\tau}\left[\partial_{22} Z^0 \partial_2 A^0- (\kappa - Z^0) \partial_{222} A^0+\partial_2 Z^0 \partial_{22} A^0\right]+ \phi^2 \beta_3 \beta_{\tau}e^{-s}(\kappa - Z^0) \partial_{2} A^0\\
			&-2\dot\phi\beta_{\tau} e^{-s} \partial_{2} A^0-\phi \beta_3 \beta_{\tau} e^{-\frac s 2} \partial_{2}Z^0 \partial_{2}Z^0+ \beta_{\tau} e^{-\frac s 2}\partial_{22} A^0 Q_{12}\\
			&+2 {\phi}^2 \beta_3 \beta_{\tau} e^{-s}\left[(\kappa - Z^0) \partial_{2} A^0 -A^0\partial_{2}Z^0\right]- \phi^3 \beta_3 \beta_{\tau} e^{-\frac 3 2 s}(\kappa - Z^0) (\kappa + Z^0)\\
			&- 2\phi \beta_1 \beta_{\tau} e^{-\frac s 2}[2\partial_{2}A^0\partial_{2} A^0+A^0\partial_{22} A^0]+ 2\phi^3\beta_1 \beta_{\tau} e^{-\frac 3 2 s} {(A^0)}^2\\
			&-2\phi^2 \beta_1 \beta_{\tau} e^{-s} \partial_2{[A(U\cdot N)]}^0 -\phi h_W^0 \partial_{22} A^0 +\phi^3 e^{-s} h_W^0 A^0\\
			&-(1+\alpha) \phi^2 \beta_1 \beta_{\tau} e^{-\frac 3 2 s} Q_{21} A^0 - \phi \beta_3 \beta_{\tau} e^{-\frac s 2 }Z^0 \partial_{22} Z^0.
			\label{evaluation of D22FW at 0}
		\end{split}
	\end{align}
\end{subequations}

Also note that if $|\gamma|=1,2$, we have
\begin{equation}
	F_W^{(\gamma),0}=\partial^\gamma F_W^0+\partial^\gamma G_W^0.
\end{equation}

\subsection{Evolution of modulation variables}
Setting $y=0$ in the equation of $W$, we can see that
\begin{equation}
	\dot{\kappa}=\frac{1}{\beta_\tau}e^{\frac{s}{2}}(F_W^0+G_W^0).
	\label{evolution of kappa}
\end{equation}

Setting $y=0$ in the equation of $\partial_1W$, we have
\begin{equation}
	\dot\tau=\frac{1}{\beta_\tau}(\partial_1F_W^0+\partial_1G_W^0).
	\label{evolution of tau}
\end{equation}

Setting $y=0$ in the equation of $\partial_2W$, we have
\begin{equation}
	0=\partial_2F_W^0+\partial_2G_W^0.
\end{equation}
Combining this with $(\ref{evaluation of D2GW at 0})$, we obtain
\begin{equation}
	Q_{12}=-\frac{1}{\beta_1\beta_\tau(1+\alpha)}\left(\partial_2F_W^0+\beta_2\beta_\tau e^{\frac{s}{2}}\partial_2Z^0+\beta_1\beta_\tau(1+\alpha)e^{\frac{s}{2}}\phi R_{j2}\dot\xi_{j}\right).
	\label{evolution of Q12}
\end{equation}

Setting $y=0$ in the equation of $\partial_{11}W$ and $\partial_{12}W$, we have
\begin{equation}
	\begin{pmatrix}
		\partial_{111}W^0&\partial_{112}W^0\\
		\partial_{112}W^0&\partial_{122}W^0
	\end{pmatrix}
	\begin{pmatrix}
		G_W^0\\
		h_W^0
	\end{pmatrix}=
	\begin{pmatrix}
		\partial_{11}F_W^0+\partial_{11}G_W^0\\
		\partial_{12}F_W^0+\partial_{12}G_W^0
	\end{pmatrix}.
\end{equation}
Denote the matrix $\partial_1\nabla^2W^0$ by $H^0(s)$, then we have
\begin{equation}
	|G_W^0|+|h_W^0|\lesssim\left|(H^0)^{-1}\right|\left(|\partial_1\nabla F_W^0|+|\partial_1\nabla G_W^0|\right),
	\label{GW0 and hW0 controlled by D1DFW0}
\end{equation}
which shall be used to establish an upper bound for $|G_W^0|$ and $|h_W^0|$. Since $R\in SO(2)$, we have
\begin{equation}
	\dot\xi_j=R_{ji}R_{ki}\dot\xi_k=R_{j1}R_{k1}\dot\xi_k+R_{j2}R_{k2}\dot\xi_k.
\end{equation}
Combining this with (\ref{evaluation of G_W at 0})(\ref{evaluation of h_W at 0}), we have
\begin{equation}
	\dot\xi_j=\frac{R_{j1}}{(1+\alpha)\beta_1}\left(\kappa+\beta_2Z^0-\frac{1}{\beta_\tau}e^{-\frac{s}{2}}G_W^0\right)+\frac{R_{j2}}{1+\alpha}\left(2A^0-\frac{e^{\frac{s}{2}}}{\beta_1\beta_\tau}h_W^0\right).
	\label{evolution of xi}
\end{equation}

Setting $y=0$ in the equation of $\partial_{11}W$ and $\partial_{12}W$, we have
\begin{equation}
	G_W^0\partial_{122}W^0+h_W^0\partial_{222}W^0=\partial_{22}F_W^0+\partial_{22}G_W^0.
\end{equation}
Then from (\ref{evaluation of D22GW at 0}), we have
\begin{equation}
	\begin{aligned}
		\dot\phi=&\frac{e^{\frac{s}{2}}}{\beta_\tau}\left(\partial_{122}W^0G_W^0+\partial_{222}W^0h_W^0-\partial_{22}F_W^0\right)+\beta_2e^s\partial_{22}Z^0+\phi^2\left(\kappa+\beta_2Z^0-\frac{e^{-\frac{s}{2}}}{\beta_\tau}G_W^0\right)+\frac{\phi^2}{\beta_\tau}e^{-\frac{s}{2}}G_W^0.
	\end{aligned}
	\label{evolution of phi}
\end{equation}

\section{Closure of bootstrap argument for the modulation variables}
From (\ref{evaluation of ovl W at 0})(\ref{bootstrap assumptions of tilde W when y=0, |gamma|=3}), we can see that
\begin{equation}
	H^0:=\partial_1\nabla^2W^0=\partial_1\nabla^2\ovl{W}^0+\partial_1\nabla^2\widetilde W^0=\mathrm{diag}(6,2)+O(\varepsilon^{\frac{1}{4}}).
\end{equation}
As a consequence, we have
\begin{equation}
	\left|(H^0)^{-1}\right|\le1.
\end{equation}
Next we estimate $|\partial_1\nabla F_W^0|$. From (\ref{evaluation of D11FW at 0})(\ref{evaluation of D12FW at 0}), bootstrap assumptions and (\ref{new estimate of A}), we have $|\partial_{11}F_W^0|\lesssim e^{-s}$ and $|\partial_{12}F_W^0|\lesssim e^{-s}+\varepsilon^2|h_W^0|$. Then by invoking (\ref{GW0 and hW0 controlled by D1DFW0}), one can see that
\begin{equation}
	|G_W^0|+|h_W^0|\lesssim e^{-s}.
	\label{estimate of GW0 and hW0}
\end{equation}

Now we give a new estimate for $V_2=\frac{1+\alpha}{2}\left[Q_{21}\left(y_1e^{-\frac{3}{2}s}+f\right)+\frac{e^{\frac{s}{2}}}{(1+\alpha)\beta_1\beta_\tau}h_W^0+\frac{2}{1+\alpha}A^0\right]$. Recall that in (\ref{estimate of of V}) we already have a bound $|V_2|\lesssim M^{\frac{1}{4}}$, but now with the help of (\ref{estimate of GW0 and hW0}) one can see that for all $y\in\mathcal{X}(s)$, there holds that
\begin{equation}
	|V_2|\lesssim M\varepsilon^{\frac{1}{2}}.
\end{equation}

\subsection{The $\xi$ estimate}
From (\ref{evolution of xi}) we have
\begin{equation}
	\begin{aligned}
		|\dot\xi_j|=&\kappa_0+M\varepsilon+e^{-\frac{s}{2}}Me^{-s}\le\frac{1}{10}M^{\frac{1}{4}}.
	\end{aligned}
\end{equation}
From (\ref{bootstrap time}) and $\xi(-\varepsilon)=0$, we have
\begin{equation}
	|\xi_j(t)|\le\int_{-\varepsilon}^t|\dot\xi_j|dt\le\frac{1}{10}M^{\frac{1}{4}}\varepsilon.
\end{equation}

\subsection{The $\kappa$ estimate}
From (\ref{evaluation of FW at 0}) and bootstrap assumptions, we have $|F_W^0|\lesssim \varepsilon^{\frac{1}{4}}e^{-\frac{s}{2}}$, thus according to (\ref{evolution of kappa})(\ref{estimate of GW0 and hW0}), we have that
\begin{equation}
	|\dot \kappa|\lesssim e^{\frac{s}{2}}\left(Me^{-s}+\varepsilon^{\frac{1}{4}}e^{-\frac{s}{2}}\right)\le\frac{1}{2}M,
\end{equation}
and
\begin{equation}
	|\kappa-\kappa_0|\le\frac{1}{2}M|t+\varepsilon|\lesssim M\varepsilon\le\frac{1}{4}\kappa_0.
	\label{closure of kappa bootstrap}
\end{equation}

\subsection{The $\phi$ estimate}
From (\ref{evaluation of D22FW at 0}), bootstrap assumptions and (\ref{new estimate of A}), we have $|\partial_{22}F_W^0|\lesssim e^{-\frac{s}{2}}$, thus via (\ref{evolution of phi}), we obtain
\begin{equation}
	\begin{aligned}
		|\dot\phi|&\lesssim e^{\frac{s}{2}}\left(\varepsilon^{\frac{1}{4}}Me^{-s}+\varepsilon^{\frac{1}{4}}Me^{-s}+e^{-\frac{s}{2}}\right)+e^sMe^{-s} +M^4\varepsilon^2\left(\kappa_0+M\varepsilon+e^{-\frac{s}{2}}Me^{-s}\right)+M^4\varepsilon^2e^{-\frac{s}{2}}M^{-s}\\
		&\lesssim M\le\frac{1}{10}M^2.
	\end{aligned}
\end{equation}
Since $|\phi(-\varepsilon)|=|\phi_0|\le\varepsilon$, we can further obtain that
\begin{equation}
	|\phi|\le\varepsilon+|\dot\phi||t+\varepsilon|\le\frac{1}{2}M^2\varepsilon.
\end{equation}

\subsection{The $\tau$ estimate}
Also from (\ref{evaluation of D1FW at 0})(\ref{evaluation of D1GW at 0}) and bootstrap assumptions, we have $|\partial_1F_W^0|\lesssim e^{-s}$ and $|\partial_1G_W^0|\lesssim M^\frac{1}{2}e^{-s}$, thus by (\ref{evolution of tau}), we have
\begin{equation}
	|\dot\tau|\lesssim e^{-s}+M^\frac{1}{2}e^{-s}\le\frac{1}{4}Me^{-s}.
\end{equation}
Since $\tau(-\varepsilon)=0$, we get
\begin{equation}
	|\tau(t)|\le\int_{-\varepsilon}^t\frac{1}{4}M\varepsilon dt\le\frac{1}{4}M\varepsilon^2.
\end{equation}

\subsection{The $n_2$ estimate}
We first estimate $Q_{12}$. From (\ref{evaluation of D2FW at 0})(\ref{estimate of GW0 and hW0}) and bootstrap assumptions, $|\partial_2F_W^0|\lesssim M\kappa_0e^{-s}$, thus via (\ref{evolution of Q12}), we can bound $Q_{12}$ by
\begin{equation}
	|Q_{12}|\lesssim M\kappa_0e^{-s}+e^{\frac{s}{2}}M\varepsilon^{\frac{1}{2}}e^{-\frac{s}{2}}\le2M\varepsilon^{\frac{1}{2}}.
\end{equation}
From the definition of $Q$, we have
\begin{equation}
	Q_{12}=-\dot n_2\sqrt{1-n_2^2}-\frac{n_2^2\dot n_2}{\sqrt{1-n_2^2}},
\end{equation}
thus by bootstrap assumption of $n_2$, we finally can see that
\begin{equation}
	|\dot n_2|=|Q_{12}|\left(\sqrt{1-n_2^2}+\frac{n_2^2}{\sqrt{1-n_2^2}}\right)^{-1}\le\left(1+\varepsilon^{\frac{1}{2}}\right)|Q_{12}|\le\frac{1}{2}M^2\varepsilon^{\frac{1}{2}}.
\end{equation}
By $n_2(-\varepsilon)=0$, we improve the assumption of $n_2$ by factor $\frac{1}{2}$. 
\section{Estimates for transport and forcing terms}
To close the bootstrap argument of the Riemann variables $W,Z,A$, we will estimate each term in the transport-type equations they satisfy. 
\subsection{Transport estimates}
\begin{lemma}For the transport terms in the equations of $W,Z,A$, we have the following inequalities:
	\begin{equation}
		|\partial^\gamma G_W|\lesssim\left\{
		\begin{aligned}
			&Me^{-s}+M^{\frac{1}{2}}e^{-s}|y_1|+M^2\varepsilon^{\frac{1}{2}}|y_2|\lesssim\varepsilon^{\frac{1}{3}}e^{\frac{s}{2}}\ \ &\ \gamma=(0,0)\\
			&M^2e^{-\frac{5}{6}s}\ \ &\ \gamma=(1,0)\\
			&M^2\varepsilon^{\frac{1}{6}}\ \ &\ \gamma=(0,1)\\
			&M^2e^{-\frac{s}{2}}\ \ &|\gamma|=2,
		\end{aligned}\right.
		\label{estimate for GW}
	\end{equation}
	\begin{equation}
		\left|\partial^\gamma\left(G_A+(1-\beta_1)\kappa_0e^{\frac{s}{2}}\right)\right|+\left|\partial^\gamma\left(G_Z+(1-\beta_2)\kappa_0e^{\frac{s}{2}}\right)\right|\lesssim\left\{
		\begin{aligned}
			&\varepsilon^{\frac{1}{3}}e^{\frac{s}{2}}\ \ &\gamma=(0,0)\\
			&M^2e^{-\frac{5}{6}s}\ \ &\gamma=(1,0)\\
			&M^2\varepsilon^{\frac{1}{6}}\ \ &\gamma=(0,1)\\
			&M^2e^{-\frac{s}{2}}\ \ &|\gamma|=2,
		\end{aligned}\right.
		\label{estimates for GZ and GA}
	\end{equation}
	\begin{equation}
		|\partial^\gamma h_W|+|\partial^\gamma h_Z|+|\partial^\gamma h_A|\lesssim\left\{
		\begin{aligned}
			&M\varepsilon^{\frac{1}{2}}e^{-\frac{s}{2}}\ \ &\gamma=(0,0)\\
			&M\varepsilon^{\frac{1}{3}}e^{-s}\eta^{-\frac{1}{3}}\ \ &\gamma=(1,0)\\
			&\varepsilon^{\frac{1}{3}}e^{-s}\ \ &\gamma=(0,1)\\
			&\varepsilon^{\frac{1}{6}}e^{-s}\eta^{-\frac{1}{6}}\ \ &\gamma=(2,0)\\
			&\varepsilon^{\frac{1}{6}}e^{-s}\eta^{-\frac{1}{6}}\ \ &\gamma=(1,1)\\
			&e^{-s}\eta^{-\frac{1}{6}}\ \ &\gamma=(0,2).
		\end{aligned}
		\right.
		\label{estimates for h}
	\end{equation}
	Furthermore, for $|\gamma|=3,4$ we have
	\begin{equation}
		\left\{
		\begin{aligned}
			&|\partial^\gamma G_W|\lesssim e^{-\left(\frac{1}{2}-\frac{|\gamma|-1}{2(k-3)}\right)s}\\
			&|\partial^\gamma h_W|\lesssim e^{-s}.
		\end{aligned}
		\right.
		\label{higher order estimates for transport terms}
	\end{equation}
\end{lemma}
\begin{proof}
	For $\gamma>0$, from the definition (\ref{transport terms of W,Z,A}) of $G_W$, we have
	\begin{equation}
		|\partial^\gamma G_W|\lesssim e^{\frac{s}{2}}\left|\partial^\gamma\frac{\partial_t f}{1+f_{x_1}}\right|+e^{\frac{s}{2}}\sum_{\beta\le\gamma}\left|\partial^\beta J\right|\left(\kappa_0\mathbbm{1}_{\beta=\gamma}+|\partial^{\gamma-\beta}Z|+|\partial^{\gamma-\beta}(V\cdot N)|\right).
	\end{equation}
	Then appealing to bootstrap assumptions and (\ref{estimates of f-dependent functions, inequalities})(\ref{estimate of of V})(\ref{new estimate of A})(\ref{new estimate of Z}), we obtain the desired estimates for $G_W$. 
	For the case $\gamma=0$,  we have that
	\begin{equation}
		\begin{aligned}
			|G_W|&\le\left|\left(G_W+\beta_\tau e^{\frac{s}{2}}\frac{\partial_tf}{1+f_{x_1}}\right)^0\right|+\left\|\partial_1\left(G_W+\beta_\tau e^{\frac{s}{2}}\frac{\partial_tf}{1+f_{x_1}}\right)\right\|_{L^\infty}|y_1|+\left\|\partial_2\left(G_W+\beta_\tau e^{\frac{s}{2}}\frac{\partial_tf}{1+f_{x_1}}\right)\right\|_{L^\infty}|y_2|\\
			&\ \ \left|\beta_\tau e^{\frac{s}{2}}\frac{\partial_tf}{1+f_{x_1}}\right|\\
			&\lesssim|G_W^0|+M^{\frac{1}{2}}\varepsilon^{\frac{1}{2}}e^{-s}|y_1|+M^2\varepsilon^{\frac{2}{3}}e^{\frac{s}{2}}\\
			&\lesssim Me^{-s}+M^{\frac{1}{2}}\varepsilon^{\frac{s}{2}}+M^2\varepsilon^{\frac{2}{3}}e^{\frac{s}{2}}\lesssim \varepsilon^{\frac{1}{3}}e^{\frac{s}{2}}.
		\end{aligned}
	\end{equation}
	Once we have the bounds for $G_W$ and its derivatives, the estimates of $G_Z$ and $G_A$ follow from the identities
	\begin{equation}
		\begin{aligned}
			G_Z+(1-\beta_2)\kappa_0e^{\frac{s}{2}}&=G_W+(1-\beta_2)e^{\frac{s}{2}}\left[(\kappa_0-\kappa)+(1-\beta_\tau J)\kappa+\beta_\tau J\right],\\
			G_A+(1-\beta_1)e^{\frac{s}{2}}\kappa_0&=G_W+(1-\beta_1)e^{\frac{s}{2}}\left[(\kappa_0-\kappa)+(1-\beta_\tau J)\kappa\right]+(\beta_2-\beta_1)\beta_\tau e^{\frac{s}{2}} JZ.
		\end{aligned}
	\end{equation}
	
	The estimates of $h_W$, $h_Z$, $h_A$ can be obtain by the definition of these transport terms, bootstrap assumptions and (\ref{estimates of f-dependent functions, inequalities})(\ref{estimate of of V})(\ref{new estimate of A})(\ref{new estimate of Z})(\ref{new estimate of W}). 
\end{proof}

\subsection{Forcing estimates}
Now we deal with the forcing terms that appear in the equations of $W,Z,A$. 
\begin{lemma}
	For derivatives of the forcing terms, we have the following bounds:
	\begin{equation}
		|\partial^\gamma F_W|+e^{\frac{s}{2}}|\partial^\gamma F_Z|\lesssim\left\{
		\begin{aligned}
			&e^{-\frac{s}{2}},\ &\gamma=(0,0)\\
			&e^{-s}\eta^{-\frac{1}{6}+\frac{2}{3(k-2)}},\ &\gamma=(1,0)\\
			&M^2e^{-s},\ &\gamma=(0,1)\\
			&e^{-s}\eta^{-\frac{1}{6}+\frac{1}{k-2}},\ &\gamma=(2,0)\\
			&e^{-s}\eta^{-\frac{1}{6}+\frac{1}{k-2}},\ &\gamma=(1,1)\\
			&M^{\frac{1}{4}} e^{-\left(1-\frac{1}{k-3}\right)s},\ &\gamma=(0,2),\\
		\end{aligned}
		\right.
		\label{estimates for derivatives of FW and FZ}
	\end{equation}
	\begin{equation}
		|\partial^\gamma F_W|\lesssim\left\{
		\begin{aligned}
			&e^{-\frac{s}{2}},\ &|\gamma|=3\\
			&\varepsilon^{\frac{1}{6}},\ \ \ \ \  &|\gamma|=4,|y|\le l,
		\end{aligned}
		\right.
		\label{estimates for higher order derivatives of FW}
	\end{equation}
	\begin{equation}
		|\partial^\gamma F_A|\lesssim\left\{
		\begin{aligned}
			&M^{\frac{1}{2}}e^{-s},\ &\gamma=(0,0)\\
			&M^{\frac{1}{4}}e^{-s},\ &\gamma=(0,1)\\
			&M^{\frac{1}{4}} e^{-\left(1-\frac{1}{k-3}\right)s}\eta^{-\frac{1}{6}},\ &\gamma=(0,2),\\
		\end{aligned}
		\right.
		\label{estimates for derivatives of FA}
	\end{equation}
	\begin{equation}
	|\partial^\gamma \widetilde F_W|\lesssim\left\{
		\begin{aligned}
			&M\varepsilon^{\frac{1}{6}}\eta^{-\frac{1}{6}},\ &\gamma=(0,0),\ |y|\le L\\
			&\varepsilon^{\frac{1}{6}}\eta^{-\frac{1}{2}+\frac{2}{3(k-2)}},\ &\gamma=(1,0),\ |y|\le L\\
			&M^2\varepsilon^{\frac{1}{6}}\eta^{-\frac{1}{3}},\ &\gamma=(0,1),\ |y|\le L\\
			&\varepsilon^{\frac{1}{7}},\ &|\gamma|\le4,\ |y|\le l,\\
		\end{aligned}
	\right.
	\label{estimates for derivatives of tilde FW}
	\end{equation}
and
	\begin{equation}
		\left|(\partial^\gamma \widetilde{F}_W)^0\right|\overset{|\gamma|=3}{\lesssim} e^{-\left(\frac{1}{2}-\frac{1}{k-3}\right)s}.
	\label{estimates for derivatives of tilde FW at origin}
	\end{equation}
\end{lemma}
\begin{proof}
	The proof of (\ref{estimates for derivatives of FW and FZ})(\ref{estimates for higher order derivatives of FW})(\ref{estimates for derivatives of FA})(\ref{estimates for derivatives of tilde FW}) is just taking derivatives of the forcing terms, then using the bootstrap assumptions and the estimates (\ref{estimates of f-dependent functions, inequalities})(\ref{estimate of Q})(\ref{estimate of of V})(\ref{estimates of derivatives of V})(\ref{new estimate of W})(\ref{new estimate of Z})(\ref{new estimate of A})(\ref{estimate for GW})(\ref{estimates for GZ and GA})(\ref{estimates for h})(\ref{higher order estimates for transport terms}) to estimate each term therein. Finally we prove (\ref{estimates for derivatives of tilde FW at origin}). Since $\partial^\gamma\ovl{W}^0=0$ when $|\gamma|$ is even, and $\partial_2G_W^0+\partial_2F_W^0=0$, we have
	\begin{equation}
		\begin{aligned}
			|(\partial^\gamma\widetilde{F}_W)^0|&\lesssim e^{-\frac{s}{2}}+|(1-\beta_\tau J)^0|+\sum_{m=1}^3|\nabla^m J^0|+|\nabla G_W^0|+|\nabla^3 G_W^0|+|\nabla h_W^0|+|\nabla^3 h_W^0|\\
			&\lesssim Me^{-\frac{s}{2}}+M^2e^{-\frac{5}{6}s}+e^{-\left(\frac{1}{2}-\frac{1}{k-3}\right)s}\lesssim e^{-\left(\frac{1}{2}-\frac{1}{k-3}\right)s}.
		\end{aligned}
	\end{equation}
\end{proof}

\begin{lemma}For the forcing terms of $\partial^\gamma W,\partial^\gamma Z,\partial^\gamma A$, we have that
	\begin{equation}
		|F_W^{(\gamma)}|\lesssim\left\{
		\begin{aligned}
			&e^{-\frac{s}{2}},\ &\gamma=(0,0)\\
			&\varepsilon^{\frac{1}{4}}\eta^{-\frac{1}{2}+\frac{2}{3(k-2)}},\ &\gamma=(1,0)\\
			&M^2\varepsilon^{\frac{1}{6}}\eta^{-\frac{1}{3}},\ &\gamma=(0,1)\\
			&\eta^{-\frac{1}{2}+\frac{1}{k-2}},\ &\gamma=(2,0)\\
			&M^{\frac{1}{3}}\eta^{-\frac{1}{3}},\ &\gamma=(1,1)\\
			&M^{\frac{2}{3}} \eta^{-\frac{1}{3}+\frac{1}{3(k-3)}},\ &\gamma=(0,2),\\
		\end{aligned}
		\right.
		\label{estimate of forcing terms of W}
	\end{equation}
	\begin{equation}
		|F_Z^{(\gamma)}|\lesssim\left\{
		\begin{aligned}
			&e^{-s},\ &\gamma=(0,0)\\
			&e^{-\frac{3}{2}s}\eta^{-\frac{1}{6}+\frac{2}{3(k-2)}},\ &\gamma=(1,0)\\
			&M^2e^{-\frac{3}{2}s},\ &\gamma=(0,1)\\
			&e^{-\frac{3}{2}s}\left(1+M\eta^{-\frac{1}{3}}\right),\ &\gamma=(2,0)\\
			&e^{-\frac{3}{2}s}\left(M^{\frac{1}{2}}+M^2\eta^{-\frac{1}{3}}\right),\ &\gamma=(1,1)\\
			&M^{\frac{1}{4}}e^{-\left(\frac{3}{2}-\frac{1}{k-3}\right)s},\ &\gamma=(0,2),\\
		\end{aligned}
		\right.
		\label{estimate of forcing terms of Z}
	\end{equation}
	\begin{equation}
		|F_A^{(\gamma)}|\lesssim\left\{
		\begin{aligned}
			&M^{\frac{1}{4}}e^{-s},\ &\gamma=(0,0)\\
			&M^{\frac{1}{4}}e^{-s},\ &\gamma=(0,1)\\
			&e^{-\left(1-\frac{2}{k-3}\right)s}\eta^{-\frac{1}{6}},\ &\gamma=(0,2),\\
		\end{aligned}
		\right.
		\label{estimate of forcing terms of A}
	\end{equation}
	\begin{equation}
		|\widetilde F_W^{(\gamma)}|\lesssim\left\{
		\begin{aligned}
			&\varepsilon^{\frac{1}{11}}\eta^{-\frac{1}{2}},\ &\gamma=(1,0),\ |y|\le L\\
			&\varepsilon^{\frac{1}{12}}\eta^{-\frac{1}{3}},\ &\gamma=(0,1),\ |y|\le L\\
			&\varepsilon^{\frac{1}{7}}+\varepsilon^{\frac{1}{10}}\left(\log M\right)^{\gamma_2-1},\ &|\gamma|\le4,\ |y|\le l.\\
		\end{aligned}
		\right.
		\label{estimate of forcing terms of tilde W}
	\end{equation}
And for $y=0$ and $|\gamma|=3$, we have
	\begin{equation}
		\left|\widetilde{F}_W^{(\gamma),0}\right|\lesssim e^{-\left(\frac{1}{2}-\frac{1}{k-3}\right)s},\ \ \ \ |\gamma|=3.
		\label{estimate of forcing terms of tilde W when |gamma|=3 and y=0}
	\end{equation}
\end{lemma}
\begin{proof}
	First, we have
	\begin{equation}
		\left|F_W^{(0,0)}\right|=|F_W|\lesssim e^{-\frac{s}{2}}.
	\end{equation}
	For the case $1\le|\gamma|\le2$, we decompose the estimate for forcing term as
	\begin{equation}
		\begin{aligned}
			\left|F_W^{(\gamma)}\right|\overset{1\le|\gamma|\le2}{\lesssim}&|\partial^\gamma F_W|+\sum_{0\le\beta<\gamma}\left(|\partial^{\gamma-\beta}G_W||\partial_1\partial^\beta W|+|\partial^{\gamma-\beta}h_W||\partial_2\partial^\beta W|\right)\\
			&+\mathbbm{1}_{|\gamma|=2}\gamma_2|\partial_2(JW)||\partial_1^{\gamma_1+1}\partial_2^{\gamma_2-1}W|+\left|[\partial^\gamma,J]W\partial_1W\right|\\
			=&|\partial^\gamma F_W|+I_1^{(\gamma)}+I_2^{(\gamma)}+I_3^{(\gamma)}.
		\end{aligned}
	\end{equation}
	Then one can check that each term do not exceed the proposed bound. $F_Z^{(\gamma)}$, $F_A^{(\gamma)}$ and $\widetilde{F}_W^{(\gamma)}$ can be estimated in a similar fashion. 
\end{proof}

\section{Bounds on Lagrangian trajectories}
Given a point $y_0$ and an initial time $s_0\ge-\log\varepsilon$, we define the Lagrangian trajectory $\Phi_W^{y_0}$ by
\begin{equation}
	\left\{
		\begin{aligned}
			&\frac{d}{ds}\Phi_W^{y_0}(s)=\mathcal{V}_W\circ\Phi_W^{y_0}(s)\\
			&\Phi_W^{y_0}(s_0)=y_0.
		\end{aligned}
	\right.
\end{equation}
Similarly we define $\Phi_Z^{y_0}$ and $\Phi_A^{y_0}$ using the transport terms in the equations of $Z$ and $A$ respectively. 

We now discuss the upper bound and the lower bound of these Lagrangian trajectories, and we will close the bootstrap argument for the spatial support of $W,Z,A$. 

\subsection{Upper bound of the trajectories}\label{upper bound of the trajectories}
\begin{lemma}Let $\Phi$ denote either $\Phi_W^{y_0}$, $\Phi_Z^{y_0}$, $\Phi_A^{y_0}$, for any $y_0\in\mathcal{X}_0$, we have that
	\begin{equation}
		\begin{aligned}
			|\Phi_1(s)|&\le\frac{3}{2}\varepsilon^{\frac{1}{2}}e^{\frac{3}{2}s},\\
			|\Phi_2(s)|&\le\frac{3}{2}\varepsilon^{\frac{1}{6}}e^{\frac{s}{2}}.
		\end{aligned}
	\end{equation}
\end{lemma}
\begin{proof}
	We first deal with the case $\Phi=\Phi_W^{y_0}$. Note that
	\begin{equation}
		\begin{aligned}
			\frac{d}{ds}\left(e^{-\frac{3}{2}s}\Phi_1(s)\right)&=e^{-\frac{3}{2}s}g_W\circ\Phi,\\
			\frac{d}{ds}\left(e^{-\frac{s}{2}}\Phi_2(s)\right)&=e^{-\frac{s}{2}}h_W\circ\Phi,\\
			\Phi(-\log\varepsilon)&=y_0.
		\end{aligned}
	\end{equation}
	Then the estimates are direct consequences of $|g_W|\le e^{\frac{s}{2}}$ and $|h_W|\le e^{\frac{s}{2}}$. We omit the detail, which is the same as that in \cite{buckmaster2022formation} The estimates for $\Phi_Z$ and $\Phi_A$ are similar. 
\end{proof}

Now we close the bootstrap bound for spatial support. We attempt to show that
\begin{equation}
	\operatorname{supp}(DW,DZ,DA)\subset\frac{7}{8}\mathcal{X}(s)=\left\{|y_1|\le\frac{7}{4}\varepsilon^{\frac{1}{2}}e^{\frac{3}{2}s},|y_2|\le\frac{7}{4}\varepsilon^{\frac{1}{6}}e^{\frac{s}{2}}\right\}.
\end{equation}
Since $\mathrm{supp}_x(D_xN,D_xT)\subset\frac{3}{4}\mathcal{X}(s)=\{|x_1|\le\frac{3}{2}\varepsilon^{\frac{1}{2}},|x_2|\le\frac{3}{2}\varepsilon^{\frac{1}{6}}\}$, in $\left(\frac{3}{4}\mathcal{X}(s)\right)^c$, there hold
\begin{equation}
	\left\{\begin{aligned}
		&g_W=\beta_\tau JW+\beta_\tau e^{\frac{s}{2}}\left[-\frac{\partial_t f}{1+f_{x_1}}+J\left(\kappa+\beta_2Z+2\beta_1V_1\right)\right]\\
		&g_Z=\beta_2\beta_\tau JW+\beta_\tau e^{\frac{s}{2}}\left[-\frac{\partial_t f}{1+f_{x_1}}+J\left(\beta_2\kappa+Z+2\beta_1V_1\right)\right]\\
		&g_A=\beta_1\beta_\tau JW+\beta_\tau e^{\frac{s}{2}}\left[-\frac{\partial_t f}{1+f_{x_1}}+J\left(\beta_1\kappa+\beta_1Z+2\beta_1V_1\right)\right],
	\end{aligned}\right.
	\label{transport terms of W,Z,A outside the support}
\end{equation}

\begin{equation}
	h_W=h_Z=h_A=2\beta_1\beta_\tau e^{-\frac{s}{2}}\left(V_2+A\right),
\end{equation}

\begin{equation}
	\left\{
	\begin{aligned}
		F_W&=-2\beta_{3}\beta_\tau S\partial_2A+\beta_\tau e^{-\frac{s}{2}}Q_{12}A\\
		F_Z&=2\beta_{3}\beta_\tau S\partial_2A+\beta_\tau e^{-\frac{s}{2}}Q_{12}A\\
		F_A&=-2\beta_{3}\beta_\tau S\partial_2S-\beta_\tau e^{-\frac{s}{2}}Q_{12}U\cdot N.
	\end{aligned}
	\right.
\end{equation}

We also define
\begin{equation}
	\left\{
	\begin{aligned}
		W_\infty(t)&=\left[\frac{\kappa_0}{2}(n_1+1)-\kappa\right]e^{\frac{s}{2}}\\
		Z_\infty(t)&=\frac{\kappa_0}{2}(n_1-1)\\
		A_\infty(t)&=-\frac{\kappa_0}{2}n_2\\
		S_\infty(t)&=\frac{e^{-\frac{s}{2}}W_\infty+\kappa-Z_\infty}{2}=\frac{\kappa_0}{2}.
	\end{aligned}
	\right.
\end{equation}

Then $W-W_\infty$, $Z-Z_\infty$, $A-A_\infty$ satisfy transport-type equations:
\begin{equation}
	\begin{aligned}
		\left(\partial_s-\frac{1}{2}\right)(W-W_\infty)+\mathcal{V}_W\cdot\nabla(W-W_\infty)&=F_{W-W_\infty},\\
		\partial_s(Z-Z_\infty)+\mathcal{V}_Z\cdot\nabla(Z-Z_\infty)&=F_{Z-Z_\infty},\\
		\partial_s(A-A_\infty)+\mathcal{V}_A\cdot\nabla(A-A_\infty)&=F_{A-A_\infty}.
	\end{aligned}
\end{equation}
where
\begin{equation}
	\begin{aligned}
			F_{W-W_\infty}=&-\beta_3\beta_\tau e^{-\frac{s}{2}}(W-W_\infty)\partial_2A+\beta_3\beta_\tau(Z-Z_\infty)\partial_2A+\beta_\tau e^{-\frac{s}{2}}Q_{12}(A-A_\infty)-2\beta_3\beta_\tau S_\infty\partial_2A,\\
			F_{Z-Z_\infty}=&\beta_3\beta_\tau e^{-s}(W-W_\infty)\partial_2A-\beta_3\beta_\tau e^{-\frac{s}{2}}(Z-Z_\infty)\partial_2A+2\beta_3\beta_\tau e^{-\frac{s}{2}}S_\infty\partial_2A+\beta_\tau e^{-s}Q_{12}(A-A_\infty),\\
			F_{A-A_\infty}=&-\beta_3\beta_\tau e^{-s}(W-W_\infty)\partial_2S+\beta_3\beta_\tau e^{-\frac{s}{2}}(Z-Z_\infty)\partial_2S-2\beta_3\beta_\tau e^{-\frac{s}{2}}S_\infty\partial_2S\\
			&-\beta_\tau e^{-\frac{3}{2}s}Q_{12}(W-W_\infty)-\beta_\tau e^{-s}Q_{12}(Z-Z_\infty).
		\end{aligned}
\end{equation}

For $y_0\notin\frac{7}{8}\mathcal{X}(s)$, let $M'>|y_0|$ be a large enough constant. Define
\begin{equation}
	Q_{big}=\left\{|y_1|\le M',|y_2|\le M'\right\},\ Q_{small}(s)=\left\{|y_1|\le e^{\frac{3}{2}s}\mu_1(s),|y_2|\le e^{\frac{s}{2}}\mu_2(s)\right\},
\end{equation}
where
\begin{equation}
	\left\{\begin{aligned}
		\mu_1(s)&=\frac{3+\varepsilon}{2}\varepsilon^{\frac{1}{2}}-2CM^{\frac{1}{4}}e^{-s}\\
		\mu_2(s)&=\frac{3+\varepsilon}{2}\varepsilon^{\frac{1}{6}}-2CM^{\frac{1}{4}}e^{-s}.
	\end{aligned}\right.
\end{equation}
One can verify that $\frac{3}{4}\mathcal{X}(s)\subset Q_{small}\subset\frac{7}{8}\mathcal{X}(s)\subset Q_{big}$ if we take $\varepsilon$ small enough and $M'$ large enough. Define
\begin{equation}
	E(y,s)=\frac{1}{2}\left(e^{-s}(W-W_\infty)^2+(Z-Z_\infty)+2(A-A_\infty)^2\right),
\end{equation}
then we have
\begin{equation}
	\frac{d}{ds}\int_{Q_{big}\setminus Q_{small}}E\le C\int_{Q_{big}\setminus Q_{small}}E.
\end{equation}
From the initial condition, we can see that when $s=-\log\varepsilon$, $\int_{Q_{big}\setminus Q_{small}}E=0$, thus $\int_{Q_{big}\setminus Q_{small}}E\equiv0$ at each time according to Gronwall's inequality. This tells us as long as $y_0\notin\frac{7}{8}\mathcal{X}(s)$, $W(y_0,s)=W_\infty$, $Z(y_0,s)=Z_\infty$, $A(y_0,s)=A_\infty$, thus we proved (\ref{refined spatial support}). 

\subsection{Lower bounds for lagrangian trajectories}
\begin{lemma}
	\label{lower bound for trajectory of W}
	Suppose $|y_0|\ge l$, $s_0\ge-\log\varepsilon$, then we have
	\begin{equation}
		|\Phi_W^{y_0}(s)|\ge|y_0|e^{\frac{s-s0}{5}}\ \ \ \text{ for all }s\ge s_0.
	\end{equation}
\end{lemma}
\begin{proof}
	It suffices to prove that $y\cdot\mathcal{V}_W\ge\frac{1}{5}|y|^2$. Note that by definition of $\mathcal{V}_W$, we can see that
	\begin{equation}
		y\cdot\mathcal{V}_W(y)\ge\frac{1}{2}|y|^2+y_1^2-\beta_\tau|y_1JW|-|y_1G_W|-|y_2h_W|.
	\end{equation}
	We split the estimate of $W$ into two cases: $|y|\le L$ and $|y|> L$. 
	If $|y|\le L$, by (\ref{bootstrap assumptions of tilde W when |y|<L})(\ref{estimates of D ovl W}) we have
	\begin{equation}
		\begin{aligned}
			|W(y)|&\le|W(y_1,y_2)-W(0,y_2)|+|W(0,y_2)-\ovl{W}(0,y_2)|+\underbrace{|\ovl{W}(0,y_2)|}_{=0}\\
			&\le(1+\varepsilon^{\frac{1}{12}})|y_1|+\varepsilon^{\frac{1}{13}}|y_2|.
		\end{aligned}
	\end{equation}
	If $|y|>L$, from bootstrap assumption we have
	\begin{equation}
		|W(y)|\le(1+\varepsilon^{\frac{1}{20}})\eta^{\frac{1}{6}}(y)\le(1+\varepsilon^{\frac{1}{20}})^2|y|.
	\end{equation}
	Then appealing to (\ref{estimate for GW})(\ref{estimates for h}) we have the desired result. 
\end{proof}

\begin{lemma}
	\label{lower bound of trajectories of Z,A}
	Let $\Phi$ denote either $\Phi_Z^{y_0}$ or $\Phi_A^{y_0}$. If
	\begin{equation}
		\kappa_0\ge\frac{3}{1-\max(\beta_1,\beta_2)}.
	\end{equation}
	then for any $0\le\sigma_1<\frac{1}{2}$ and $2\sigma_1<\sigma_2$, we have the bound
	\begin{equation}
		\int_{-\log\varepsilon}^\infty e^{\sigma_1s'}\left(1+|\Phi_1(s')|\right)^{-\sigma_2}ds'\le C(\sigma_1,\sigma_2).
		\label{integral of eta along trajectories of W and A are bounded}
	\end{equation}
\end{lemma}
\begin{proof}
	The proof is the same as that in \cite{buckmaster2022formation}. 
\end{proof}

\begin{lemma}
	Let $\Phi^{y_0}$ denote either $\Phi_Z^{y_0}$ or $\Phi_A^{y_0}$, then
	\begin{equation}
		\sup_{y_0\in\mathcal{X}_0}\int_{-\log\varepsilon}^\infty|\partial_1W|\circ \Phi^{y_0}(s')ds'\lesssim1.
		\label{integral of D1W along Phi is bounded}
	\end{equation}
\end{lemma}
\begin{proof}
	Using lemma \ref{lower bound of trajectories of Z,A} and bootstrap assumption of $\partial_1W$, we can deduce the above inequality. 
\end{proof}

\section{Closure of bootstrap argument for $\partial_1A$}
Since the vorticity is purely transported by $u$, the bootstrap of $\partial_1A$ is easy to close from the bound of the vorticity and bootstrap assumptions, in no need of the evolution equation of $\partial_1A$. 
\begin{lemma}[Relating $A$ and $\Omega$]
	We have the following identity:
	\begin{equation}
		\begin{aligned}
			J e^{\frac 3 2 s} \partial_1 A =&-{(\alpha S)}^{\frac 1 \alpha} \Omega- T_2 e^{\frac 3 2 s} \partial_2 \left(\frac {e^{\frac 3 2 s}W+\kappa+Z }{2}\right)- N_2 e^{-\frac 3 2 s}\partial_1 A+ U\cdot (N_2\partial_{x_2}T -T_2\partial_{x_2} N+J\partial_{x_1} T).
		\end{aligned}
	\end{equation}
\end{lemma}
\begin{proof}
	Note that $\mathrm{curl }\ \mathring u =\partial_T \mathring u \cdot N- \partial_T \mathring u \cdot T$. We compute each term as follows: 
	\begin{equation}
		\begin{aligned}
			\partial_T\mathring u &= T_1\partial_{\tilde{x}_1} \mathring u + T_2\partial_{\tilde{x}_2} \mathring u =T_1 \frac {1} {1+f_{x_1}}\partial_{x_1} \mathring u +T_2\left(-\frac{f_{x_2}}{1+f_{x_1}}\partial_{x_1} \mathring u + \partial_{x_2} \mathring u \right)\\
			&= \frac {f_{x_2}}{\sqrt{1+f_{x_2}^2}}\frac 1 {1+f_{x_1}}\partial_{x_1} \mathring u - \frac {f_{x_2}}{\sqrt{1+f_{x_2}^2}}\frac 1 {1+f_{x_1}}\partial_{x_1} \mathring u + \frac{\partial_{x_2}\mathring u}{\sqrt{1+f_{x_2}^2}}= T_2 \partial_{x_2}\mathring u,
		\end{aligned}
	\end{equation}
	\begin{equation}
		\begin{aligned}
			\partial_N\mathring u &= N_1\partial_{\tilde{x}_1} \mathring u + N_2\partial_{\tilde{x}_2} \mathring u =\frac {1}{\sqrt{1+f_{x_2}^2}}\frac {1} {1+f_{x_1}}\partial_{x_1} \mathring u -\frac{f_{x_2}}{\sqrt{1+f_{x_1}^2}} (-\frac{f_{x_2}}{1+f_{x_1}}\partial_{x_1} \mathring u + \partial_{x_2} \mathring u )\\
			&= \frac {\sqrt{1+f_{x_2}^2}}{1+f_{x_1}}\frac 1 {1+f_{x_1}}\partial_{x_1} \mathring u - \frac {f_{x_2}}{\sqrt{1+f_{x_2}^2}}\partial_{x_1} \mathring u= J \partial_{x_1}\mathring u+ N_2\partial_{x_2} \mathring u.
		\end{aligned}
	\end{equation}
	Thus, we have
	\begin{equation}
		\begin{aligned}
			\mathrm{curl}\ \mathring u &= T_2 \partial_{x_2} \mathring u \cdot N-(J \partial_{x_1}\mathring u+ N_2\partial_{x_2} \mathring u)\cdot T\\
			&= T_2 \partial_{x_2} (\mathring u \cdot N)- T_2 \mathring u \partial_{x_2} N-J\partial_{x_1}(\mathring u \cdot T)+ J\mathring u \partial_{x_1} T -N_2 \partial_{x_2}\left(\mathring u \cdot T\right)+N_2 \mathring u \cdot  \partial_{x_2}T\\
			&= T_2 \partial_{x_2}\left(\frac {w+z}{2}\right)- T_2 \mathring u \partial_{x_2}N-J \partial_{x_1}a +J\mathring u \partial_{x_1} T-N_2\partial_{x_2}a+ N_2 \mathring u \partial_{x_2} T\\
			&= T_2 \partial_{x_2}\left(\frac {w+z}{2}\right)-J\partial_{x_1}a- N_2 \partial_{x_2}a+\mathring u \cdot (N_2 \partial_{x_2}T-T_2\partial_{x_2} N+ J\partial_{x_1} T).
		\end{aligned}
	\end{equation}
	On the other hand, $\mathrm{curl}\ \mathring{u}=\tilde \rho\tilde\zeta=\left(\alpha S\right)^{\frac{1}{\alpha}}\Omega$, thus we get the desired result. 
\end{proof}

With the help of this identity, we have
\begin{equation}
	\begin{aligned}
		e^{\frac 3 2 s} |\partial_1 A| &\lesssim \kappa_0^{\frac 1 \alpha} + e^{\frac s 2} (e^{-\frac s 2}+ M\varepsilon^{\frac 1 2} e^{-\frac s 2})+ \varepsilon^{\frac 1 2} e^{\frac s 2}M \varepsilon^{-\frac 1 2} e^{\frac s 2}+ M^{\frac 1 4}(\varepsilon^{\frac 1 2} M^2 \varepsilon- M^2\varepsilon +M^2 \varepsilon^{\frac 2 3})\le \frac 1 2 M. 
	\end{aligned}
\end{equation}
This improves the bootstrap bound for $\partial_1A$. 

\section{Closure of bootstrap argument for $Z$ and $A$}
In this section we improve the bootstrap bound of $Z$ and $A$. 
\begin{lemma}[Close $Z$ bootstrap]
	For the Riemann variable $Z$, we have the improved bootstrap bound:
	\begin{equation}
		\begin{aligned}
			\left|Z\circ\Phi_Z^{y_0}(s)\right|&\le\frac{1}{2}M\varepsilon,\\
			e^{\frac{3}{2}s}\left|\partial_1Z\circ\Phi_Z^{y_0}(s)\right|&\le\frac{1}{2}M^{\frac{1}{2}},\\
			e^{\frac{s}{2}}\left|\partial_2Z\circ\Phi_Z^{y_0}(s)\right|&\le\frac{1}{2}M\varepsilon^{\frac{1}{2}},\\
			e^{\frac{3}{2}s}\left|\partial_{11}Z\circ\Phi_Z^{y_0}(s)\right|&\le\frac{1}{2}M^{\frac{1}{2}},\\
			e^{\frac{3}{2}s}\left|\partial_{12}Z\circ\Phi_Z^{y_0}(s)\right|&\le\frac{1}{2}M,\\
			e^{s}\left|\partial_{22}Z\circ\Phi_Z^{y_0}(s)\right|&\le\frac{1}{2}M.
		\end{aligned}
	\end{equation}
\end{lemma}
\begin{proof}
	Since $e^{\mu s}\partial^\gamma Z$ obeys
	\begin{equation}
		\partial_s(e^{\mu S}\partial^\gamma Z)+D_Z^{(\gamma,\mu)}(e^{\mu s}\partial^\gamma Z)+(\mathcal{V}_Z\cdot\nabla)(e^{\mu s}\partial^\gamma Z)=e^{\mu s}F_Z^{(\gamma)},
	\end{equation}
	by Gronwall's inequality we can see that
	\begin{equation}
		\begin{aligned}
			e^{\mu s}\left|\partial^\gamma Z\circ\Phi_Z^{y_0}(s)\right|\lesssim&\ \varepsilon^{-\mu}\left|\partial^\gamma Z(y_0,-\log\varepsilon)\right|\exp\left(-\int_{-\log\varepsilon}^sD_Z^{(\gamma,\mu)}\circ\Phi_Z^{y_0}(s')ds'\right)\\
			&+\int_{-\log\varepsilon}^se^{\mu s'}\left|F_Z^{(\gamma)}\circ\Phi_Z^{y_0}(s')\right|\exp\left(-\int_{s'}^sD_Z^{(\gamma,\mu)}\circ\Phi_Z^{y_0}(s'')ds''\right)ds',
		\end{aligned}
	\end{equation}
	where
	\begin{equation}
		D_Z^{(\gamma,\mu)}=D_Z^{(\gamma)}-\mu=\frac{3}{2}\gamma_1+\frac{1}{2}\gamma_2+\beta_2\beta_\tau\gamma_1J\partial_1W-\mu.
	\end{equation}
	If we require that $\frac{3}{2}\gamma_1+\frac{1}{2}\gamma_2\ge\mu$, then we have
	\begin{equation}
		D_Z^{(\gamma,\mu)}\le\beta_2\beta_\tau\gamma_1|J\partial_1W|\overset{|\gamma|\le2}{\le} 2|\partial_1W|.
	\end{equation}
	Thus the damping term is bound by
	\begin{equation}
		\begin{aligned}
			\exp\left(-\int_{s'}^sD_Z^{(\gamma,\mu)}\circ\Phi_Z^{y_0}(s'')ds''\right)&\lesssim e^{-\left(\frac{3\gamma_1+\gamma_2}{2}-\mu\right)(s-s')}\exp\left(\int_{s'}^s2|\partial_1W|\circ\Phi_Z^{y_0}(s'')ds''\right)\\
			\overset{(\ref{integral of D1W along Phi is bounded})}{\lesssim}e^{-\left(\frac{3\gamma_1+\gamma_2}{2}-\mu\right)(s-s')}.
		\end{aligned}
	\end{equation}
	And finally we have
	\begin{equation}
		e^{\mu s}\left|\partial^\gamma Z\circ\Phi_Z^{y_0}(s)\right|\lesssim\ \varepsilon^{-\mu}\left|\partial^\gamma Z(y_0,-\log\varepsilon)\right|+\int_{-\log\varepsilon}^se^{\mu s'}\left|F_Z^{(\gamma)}\circ\Phi_Z^{y_0}(s')\right|e^{-\left(\frac{3\gamma_1+\gamma_2}{2}-\mu\right)(s-s')}ds'.
	\end{equation}
	Next, for different multi-index $\gamma$, we choose different $\mu$ in the above inequality. 
	
	\emph{Case 1}. $\gamma=(0,0)$. We set $\mu=0$. From (\ref{initial condition of Z})(\ref{estimate of forcing terms of Z}), we have
	\begin{equation}
		\begin{aligned}
			\left|Z\circ\Phi_Z^{y_0}(s)\right|&\lesssim\varepsilon+\int_{-\log\varepsilon}^se^{-s'}ds'\lesssim\varepsilon\le\frac{1}{2}M\varepsilon.
			\end{aligned}
	\end{equation}
	\emph{Case 2}. $\gamma=(1,0)$. We set $\mu=\frac{3}{2}$. Also from (\ref{initial condition of Z})(\ref{estimate of forcing terms of Z}), we have
	\begin{equation}
		\begin{aligned}
			e^{\frac{3}{2}s}\left|Z\circ\Phi_Z^{y_0}(s)\right|&\lesssim\varepsilon^{-\frac{3}{2}}\varepsilon^{\frac{3}{2}}+\int_{-\log\varepsilon}^se^{-\frac{3}{2}}e^{\frac{3}{2}}\eta^{-\frac{1}{6}+\frac{2}{3(k-2)}}\circ\Phi_Z^{y_0}(s')ds'\\
			&\lesssim1+\int_{-\log\varepsilon}^s\left(1+|\Phi_1(s')|^2\right)^{-\frac{1}{6}+\frac{2}{3(k-2)}}ds'
			\overset{(\ref{integral of eta along trajectories of W and A are bounded})}{\lesssim}1\le\frac{1}{2}M^{\frac{1}{2}}.
		\end{aligned}
	\end{equation}
	\emph{Case 3}. $\gamma=(2,0)$. We set $\mu=\frac{3}{2}$ and we can duduce that
	\begin{equation}
		\begin{aligned}
			e^{\frac{3}{2}s}\left|\partial_{11}Z\circ\Phi_Z^{y_0}(s)\right|&\lesssim\varepsilon^{-\frac{3}{2}}\varepsilon^{\frac{3}{2}}+\int_{-\log\varepsilon}^se^{\frac{3}{2}s'}e^{-\frac{3}{2}s'}\left(1+M\eta^{-\frac{1}{3}}\circ\Phi(s')\right)e^{-\frac{3}{2}(s-s')}ds'\\
			&\lesssim1+M\int_{-\log\varepsilon}^se^{-\frac{1}{8}(s-s')}\left(1+|\Phi_1(s')|\right)^{-\frac{2}{3}}ds'
			\overset{(\ref{integral of eta along trajectories of W and A are bounded})}{\lesssim}1+Me^{-\frac{s}{8}}\le\frac{1}{2}M^\frac{1}{2}.
		\end{aligned}
	\end{equation}
	\emph{Case 4}. $\gamma=(1,1)$. We set $\mu=\frac{3}{2}$ and we can duduce that
	\begin{equation}
		\begin{aligned}
			e^{\frac{3}{2}s}\left|\partial_{12}Z\circ\Phi_Z^{y_0}(s)\right|&\lesssim\varepsilon^{-\frac{3}{2}}\varepsilon^{\frac{3}{2}}+\int_{-\log\varepsilon}^se^{\frac{3}{2}s'}e^{-\frac{3}{2}s'}\left(M^{\frac{1}{2}}+M^2\eta^{-\frac{1}{3}}\circ\Phi(s')\right)e^{-\frac{1}{2}(s-s')}ds'\\
			&\lesssim1+M^\frac{1}{2}+M^2\int_{-\log\varepsilon}^se^{-\frac{1}{8}(s-s')}\left(1+|\Phi_1(s')|\right)^{-\frac{2}{3}}ds'\\
			\overset{(\ref{integral of eta along trajectories of W and A are bounded})}&{\lesssim}M^\frac{1}{2}+M^2e^{-\frac{s}{8}}\le\frac{1}{2}M.
		\end{aligned}
	\end{equation}
	\emph{Case 5}. $\gamma=(0,2)$. We set $\mu=1$ and we can duduce that
	\begin{equation}
		\begin{aligned}
			e^{\frac{3}{2}s}\left|\partial_{22}Z\circ\Phi_Z^{y_0}(s)\right|&\lesssim\varepsilon^{-\frac{1}{2}}\varepsilon+\int_{-\log\varepsilon}^se^{s'}M^{\frac{1}{4}}e^{-\left(\frac{3}{2}-\frac{1}{k-3}\right)s'}ds'\lesssim\varepsilon^{\frac{1}{2}}+M^{\frac{1}{4}}\varepsilon^{\frac{1}{2}-\frac{1}{k-3}}\le\frac{1}{2}M.
		\end{aligned}
	\end{equation}
\end{proof}

Next we close the bootstrap argument oF $A$ by proving (\ref{refined bootstrap inequality of A}). 

\begin{lemma}[Close $A$ bootstrap]
	For the Riemann variable $A$, we have the improved bootstrap bound:
	\begin{equation}
		\begin{aligned}
			\left|A\circ\Phi_A^{y_0}(s)\right|&\le\frac{1}{2}M\varepsilon,\\
			e^{\frac{s}{2}}\left|\partial_2A\circ\Phi_A^{y_0}(s)\right|&\le\frac{1}{2}M\varepsilon^{\frac{1}{2}},\\
			e^{s}\left|\partial_{22}A\circ\Phi_A^{y_0}(s)\right|&\le\frac{1}{2}M.
		\end{aligned}
	\end{equation}
\end{lemma}
\begin{proof}
	As in the closure of $Z$ bootstrap, if $\mu=\frac{3\gamma_1+\gamma_2}{2}$, we have
	\begin{equation}
		e^{\mu s}\left|\partial^\gamma A\circ\Phi_A^{y_0}(s)\right|\lesssim\ \varepsilon^{-\mu}\left|\partial^\gamma A(y_0,-\log\varepsilon)\right|+\int_{-\log\varepsilon}^se^{\mu s'}\left|F_A^{(\gamma)}\circ\Phi_A^{y_0}(s')\right|ds'.
	\end{equation}
	For different multi-index $\gamma$, we set different $\mu$ in the above inequality. 
	
	\emph{Case 1}. $\gamma=(0,0)$. We set $\mu=0$. From (\ref{initial condition of A})(\ref{estimate of forcing terms of A}), we have
	\begin{equation}
		\begin{aligned}
			\left|A\circ\Phi_A^{y_0}(s)\right|&\lesssim\varepsilon+\int_{-\log\varepsilon}^sM^{\frac{1}{4}}e^{-s'}ds'\lesssim M^{\frac{1}{4}}\varepsilon\le\frac{1}{2}M\varepsilon.
		\end{aligned}
	\end{equation}
	\emph{Case 2}. $\gamma=(0,1)$. We set $\mu=\frac{1}{2}$ and we can deduce that
	\begin{equation}
		\begin{aligned}
			e^{\frac{s}{2}}\left|\partial_2A\circ\Phi_A^{y_0}(s)\right|&\lesssim\varepsilon^{-\frac{1}{2}}\varepsilon+\int_{-\log\varepsilon}^se^{\frac{s'}{2}}M^{\frac{1}{4}}e^{-s'}ds'\lesssim M^{\frac{1}{4}}\varepsilon^\frac{1}{2}\le\frac{1}{2}M\varepsilon^{\frac{1}{2}}.
		\end{aligned}
	\end{equation}
	\emph{Case 3}. $\gamma=(0,2)$. We set $\mu=1$ and we can deduce that
	\begin{equation}
		\begin{aligned}
			e^s\left|\partial_{22}A\circ\Phi_A^{y_0}(s)\right|&\lesssim\varepsilon^{-1}\varepsilon+\int_{-\log\varepsilon}^se^{s'}e^{-\left(1-\frac{3}{k-2}\right)s'}\eta^{-\frac{1}{6}}\circ\Phi_A^{y_0}(s')ds'\\
			&\lesssim1+M^{\frac{1}{4}}\int_{-\log\varepsilon}^se^{\frac{2}{k-2}s'}\left(1+|\Phi_1(s)|\right)^{-\frac{1}{3}}ds'
			\overset{(\ref{integral of eta along trajectories of W and A are bounded})}{\lesssim}1\le\frac{1}{2}M.
		\end{aligned}
	\end{equation}
\end{proof}

\section{Closure of bootstrap argument for $W$ and $\widetilde W$}
In this section we prove the improved bootstrap bounds (\ref{refined bootstrap inequality of W})(\ref{refined bootstrap inequality of tilde W when y=0, |gamma|=3})(\ref{refined bootstrap inequality of tilde W when |y|<L})(\ref{refined bootstrap inequality of tilde W when |y|<l, |gamma|<4})(\ref{refined bootstrap inequality of tilde W when |y|<l, |gamma|=4}) for $W$ and $\widetilde{W}$. 
\subsection{Closure of bootstrap argument for high order derivatives of $\widetilde W$}
As stated in (\ref{evolution of derivatives of tilde W}), $\partial^\gamma\widetilde W$ satisfies the equation
\begin{equation}
	\partial_s\partial^\gamma\widetilde W+D_{\widetilde W}^{(\gamma)}\partial^\gamma\widetilde W+(\mathcal{V}_W\cdot\nabla)\partial^\gamma\widetilde W=\widetilde{F}_{W}^{(\gamma)},
\end{equation}
where the damping term has a lower bound according to (\ref{estimates of f-dependent functions, inequalities})(\ref{estimate of ovl W})(\ref{estimate of D1W that will appear in damping terms}): 
\begin{equation}
	\begin{aligned}
		D_{\widetilde W}^{(\gamma)}&=\frac{3\gamma_1+\gamma_2-1}{2}+\beta_\tau J\left(\partial_1\ovl{W}+\gamma_1\partial_1W\right)\\
		&\ge\frac{3}{2}+\gamma_1-(1+\varepsilon^{\frac{1}{2}})\left(1+\gamma_1(1+\varepsilon^{\frac{1}{12}})\right)\ge\frac{3}{2}-1+\gamma_1-\gamma_1-C\varepsilon^{\frac{1}{12}}\ge\frac{1}{3}.
		\label{lower bound of damping terms of tilde W}
	\end{aligned}
\end{equation}
From the equation of $\partial^\gamma\widetilde{W}$, we have
\begin{equation}
	\frac{d}{ds}\left|\partial^\gamma\widetilde{W}\circ\Phi_W^{y_0}\right|+\left(D_{\widetilde{W}}^{(\gamma)}\circ\Phi_W^{y_0}\right)\left|\partial^\gamma\widetilde{W}\circ\Phi_W^{y_0}\right|\le\left|\widetilde{F}_W^{(\gamma)}\circ\Phi_W^{y_0}\right|.
\end{equation}
If $|\gamma|=4$ and $|y|\le l$, from (\ref{estimate of forcing terms of tilde W})(\ref{lower bound of damping terms of tilde W}), we have
\begin{equation}
	\begin{aligned}
		e^{\frac{s}{3}}\left|\partial^\gamma\widetilde{W}\circ\Phi_W^{y_0}(s)\right|\le\varepsilon^{-\frac{1}{3}}\varepsilon^{\frac{1}{8}}+Ce^{\frac{s}{3}}\left(\varepsilon^{\frac{1}{7}}+\varepsilon^{\frac{1}{10}}(\log M)^{\gamma_2-1}\right).
	\end{aligned}
\end{equation}
Thus for $|\gamma|=4$ and $|y|\le l$, we have
\begin{equation}
	\left|\partial^\gamma\widetilde{W}\circ\Phi_W^{y_0}(s)\right|\le\frac{1}{4}\varepsilon^{\frac{1}{10}}(\log M)^{\gamma_2}.
	\label{bootstrap estimate for D^gamma tilde W when |gamma|=4 and |y|le l}
\end{equation}

Now we consider the case $|\gamma|=3$, $y=0$. Let $y=0$ in (\ref{evolution of derivatives of tilde W}), we have
\begin{equation}
	\begin{aligned}
		\left|\partial_s\partial^\gamma\widetilde{W}^0\right|&=\left|\widetilde{F}_W^{(\gamma),0}-G_W^0\partial_1\partial^\gamma\widetilde{W}^0-h_W^0\partial_2\partial^\gamma\widetilde{W}^0-\left(1-\beta_\tau\right)(1+\gamma_1)\partial^\gamma\widetilde{W}^0\right|\\
		&\lesssim e^{-\left(\frac{1}{2}-\frac{1}{k-3}\right)s}+Me^{-s}\varepsilon^{\frac{1}{10}}(\log M)^4+Me^{-s}\varepsilon^{\frac{1}{4}}\lesssim e^{-\left(\frac{1}{2}-\frac{1}{k-3}\right)s}.
	\end{aligned}
\end{equation}
Thus from (\ref{initial condition of tilde W})
\begin{equation}
	|\partial^\gamma\widetilde{W}^0(s)|\le|\partial^\gamma\widetilde{W}(-\log\varepsilon)|+Ce^{-\left(\frac{1}{2}-\frac{1}{k-3}\right)s}\le\frac{1}{10}\varepsilon^{\frac{1}{4}}.
	\label{bootstrap estimate for tilde W for |gamma|=3 and y=0}
\end{equation}

Next, we consider the case $|\gamma|\le3$, $|y|\le l$. For $|\gamma|=3$, by (\ref{bootstrap estimate for D^gamma tilde W when |gamma|=4 and |y|le l})(\ref{bootstrap estimate for tilde W for |gamma|=3 and y=0}), we have
\begin{equation}
	|\partial^\gamma\widetilde{W}|\le\varepsilon^{\frac{1}{4}}+\frac{1}{2}\varepsilon^{\frac{1}{10}}(\log M)^{\gamma_2+1}|y|\le\frac{1}{2}(\log M)^4\varepsilon^{\frac{1}{10}}|y|+\frac{1}{2}M\varepsilon^{\frac{1}{4}}.
\end{equation}
Now by induction and $\partial^\gamma\widetilde{W}^0=0$ for $|\gamma|\le2$, we can close the bootstrap argument of $\partial^\gamma\widetilde{W}$ as in the case $|\gamma|=3$. 

\subsection{A general discusion of weighted estimates}
In order to close the bootstrap argument for $W$ and the rest part of $\widetilde{W}$ case, we consider the evolution of $q=\eta^\mu R$, where $R$ is the derivatives of $W$ and $\widetilde{W}$, or them itself, and $|\mu|\le\frac{1}{2}$. Suppose $R$ satisfies
\begin{equation}
	\partial_sR+D_RR+\mathcal{V}_W\cdot R=F_R,
\end{equation}
then $q$ satisfies
\begin{equation}
	\partial_sq+D_qq-\mathcal{V}_W\cdot\nabla q=\eta^\mu F_R,
\end{equation}
where
\begin{equation}
	\begin{aligned}
		D_q&=D_R-\mu\eta^{-1}\mathcal{V}_W\cdot\nabla q=D_R-3\mu+3\mu\eta^{-1}-2\mu\eta^{-1}\left(y_1g_W+3y_2^5h_W\right).
	\end{aligned}
\end{equation}
By (\ref{estimates of f-dependent functions, inequalities})(\ref{estimate for GW})(\ref{estimates for h}) and the bootstrap assumption for $W$, one can see that $|D_\eta|\le3\eta^{-\frac{1}{3}}$. Thus $D_q\ge D_R-3\mu+3\mu\eta^{-1}-6|\mu|\eta^{-\frac{1}{3}}$. 

By composing $q$ with the trajectory of $\mathcal{V}_W$, we have
\begin{equation}
	\begin{aligned}
		\left|q\circ\Phi_W^{y_0}(s)\right|\le&\ \left|q(y_0,s_0)\right|\exp\left(-\int_{s_0}^sD_q\circ\Phi_W^{y_0}(s')ds'\right)+\int_{s_0}^s\left|F_q^{(\gamma)}\circ\Phi_W^{y_0}(s')\right|\exp\left(-\int_{s'}^sD_q\circ\Phi_W^{y_0}(s'')ds''\right)ds',
	\end{aligned}
\end{equation}
where $(y_0,s_0)$ is the starting position and starting time of the trajectory. Note that $s_0$ need not to be $-\log\varepsilon$. 

If $|y_0|\ge l$, we have that
\begin{equation}
	\begin{aligned}
		2\mu\int_{s'}^sD_\eta\circ\Phi_W^{y_0}(s'')ds''\overset{|\mu|\le\frac{1}{2}}&{\le}\int_{s_0}^s 3\eta^{-\frac{1}{3}}\circ\Phi_W^{y_0}(s')ds'\\
		&\le3\cdot2^{\frac{1}{3}}\int_{s_0}^\infty\left(1+l^2e^{\frac{2}{5}(s'-s_0)}\right)^{-\frac{1}{3}}ds'\le-30\log l,
	\end{aligned}
\end{equation}
consequently, we can bound $q$ by
\begin{equation}
	\begin{aligned}
		\left|q\circ\Phi_W^{y_0}\right|\le&\ l^{-30}|q(y_0,s_0)|\exp\left[-\int_{s_0}^s\left(D_R-3\mu+3\mu\eta^{-1}\right)\circ\Phi_W^{y_0}(s')ds'\right]\\
		&+l^{-30}\int_{s_0}^s\left|F_q^{(\gamma)}\circ\Phi_W^{y_0}(s')\right|\exp\left[-\int_{s'}^s\left(D_R-3\mu+3\mu\eta^{-1}\right)\circ\Phi_W^{y_0}(s'')ds''\right]ds'.
		\label{estimate of q when y_0 ge l}
	\end{aligned}
\end{equation}
We remark that as long as $|y_0|\ge l$ and $p>0$, one can verify that
\begin{equation}
	\int_{s_0}^\infty\eta^{-p}\circ\Phi_W^{y_0}(s)ds\lesssim_p-\log l
\end{equation} 

If $|y_0|\ge L$, we have another inequality:
\begin{equation}
	\begin{aligned}
		2\mu\int_{s'}^sD_\eta\circ\Phi_W^{y_0}(s'')ds''\overset{|\mu|\le\frac{1}{2}}&{\le}\int_{s_0}^s 3\eta^{-\frac{1}{3}}\circ\Phi_W^{y_0}(s')ds'\\
		&\le3\cdot2^{\frac{1}{3}}\int_{s_0}^\infty\left(1+L^2e^{\frac{2}{5}(s'-s_0)}\right)^{-\frac{1}{3}}ds'\le CL^{-\frac{20}{3}}.
	\end{aligned}
\end{equation}
And $q$ is bound by
\begin{equation}
	\begin{aligned}
		\left|q\circ\Phi_W^{y_0}\right|\le&\ e^\varepsilon|q(y_0,s_0)|\exp\left[-\int_{s_0}^s\left(D_R-3\mu+3\mu\eta^{-1}\right)\circ\Phi_W^{y_0}(s')ds'\right]\\
		&+e^\varepsilon\int_{s_0}^s\left|F_q^{(\gamma)}\circ\Phi_W^{y_0}(s')\right|\exp\left[-\int_{s'}^s\left(D_R-3\mu+3\mu\eta^{-1}\right)\circ\Phi_W^{y_0}(s'')ds''\right]ds'.
	\end{aligned}
	\label{estimate of q when y_0 ge L}
\end{equation}

\subsection{Closure of bootstrap argument for $\widetilde W$}
For different multi-index $\gamma$, we choose different $\mu$, and we will use (\ref{estimate of q when y_0 ge l}) or (\ref{estimate of q when y_0 ge L}), depending on the location of $y$. We establish the estimates case by case. 

\emph{Case 1}. $|\gamma|=0$, $l\le|y|\le L$. In this case we set $\mu=-\frac{1}{6}$, thus we have $q=\eta^{-\frac{1}{6}}\widetilde{W}$ and $D_R-3\mu+3\mu\eta^{-1}=-\frac{1}{2}\eta^{-1}+\beta_\tau J\partial_1\ovl{W}$. We estimate the damping term and the forcing term. 
\begin{equation}
	\begin{aligned}
		-\int_{s'}^s\left(\beta_\tau J\partial_1\ovl{W}-\frac{1}{2}\eta^{-1}\right)\circ\Phi_W^{y_0}(s'')ds''&\le(1+\varepsilon^{\frac{1}{2}})\int_{s_0}^s\left|\partial_1\ovl{W}\circ\Phi_W^{y_0}(s'')\right|ds''+\frac{1}{2}\int_{s_0}^s\eta^{-1}\circ\Phi_W^{y_0}(s'')ds''\\
		&\le2\int_{s_0}^s\eta^{-\frac{1}{3}}\circ\Phi_W^{y_0}(s'')ds''\le-20\log l,
	\end{aligned}
\end{equation}
\begin{equation}
	\begin{aligned}
		\int_{s_0}^s\left|\left(\eta^{-\frac{1}{6}}\widetilde{F}_W\right)\circ\Phi_W^{y_0}(s')\right|ds'&\lesssim\int_{s_0}^s M\varepsilon^{\frac{1}{6}}\eta^{-\frac{1}{3}}\circ\Phi_W^{y_0}(s')ds'\le-\varepsilon^{\frac{1}{8}}\log l.
	\end{aligned}
\end{equation}
According to lemma \ref{lower bound for trajectory of W}, it is possible to require that either $|y_0|=l$ or $s_0=-\log\varepsilon$, thus we can use the initial condition or bootstrap assumptions to bound $|q(y_0,s_0)|$. From (\ref{estimate of q when y_0 ge l})(\ref{initial condition of tilde W})(\ref{bootstrap assumptions of tilde W when |y|<L}), we have
\begin{equation}
	\begin{aligned}
		\left|\eta^{-\frac{1}{6}}\widetilde{W}\circ\Phi_W^{y_0}(s)\right|&\le l^{-30}\left|\widetilde{W}(y_0,s_0)\right|\eta^{-\frac{1}{6}}(y_0)l^{-20}+l^{-30}l^{-20}(-\varepsilon^\frac{1}{8})\log l\\
		&\le l^{-50}\eta^{-\frac{1}{6}}(y_0)\max\left(\varepsilon^{\frac{1}{10}}\eta^{-\frac{1}{6}}(y_0),2(\log M)^4\varepsilon^{\frac{1}{10}}l^4\right)-l^{-50}\varepsilon^{\frac{1}{8}}\log l\le\frac{1}{2}\varepsilon^{\frac{1}{11}}.
	\end{aligned}
\end{equation}

\emph{Case 2}. $\gamma=(1,0)$, $l\le|y|\le L$. Let $\mu=\frac{1}{3}$, then we have $D_R-3\mu+3\mu\eta^{-1}\ge\beta_\tau J(\partial_1\ovl{W}+\partial_1W)$, and
\begin{equation}
	\begin{aligned}
		-\int_{s'}^s\left(D_R-3\mu+3\mu\eta^{-1}\right)\circ\Phi_W^{y_0}(s'')ds''&\le4\int_{s_0}^\infty\eta^{-\frac{1}{3}}\circ\Phi_W^{y_0}(s'')ds''\le -40\log l,
	\end{aligned}
\end{equation}
\begin{equation}
	\begin{aligned}
		\int_{s_0}^s\left|F_q\circ\Phi_W^{y_0}(s')\right|ds'&\lesssim\varepsilon^{\frac{1}{11}}\int_{s_0}^s\left(\eta^{\frac{1}{3}}\eta^{-\frac{1}{2}}\right)\circ\Phi_W^{y_0}(s')ds'\lesssim-\varepsilon^{\frac{1}{11}}\log l.
	\end{aligned}
\end{equation}
Now we can bound $q$ by
\begin{equation}
	\begin{aligned}
		\left|\eta^{\frac{1}{3}}\widetilde{W}\circ\Phi_W^{y_0}(s)\right|&\le l^{-30}\left|\widetilde{W}(y_0,s_0)\right|\eta^{\frac{1}{3}}(y_0)l^{-40}+l^{-30}l^{-40}(-\varepsilon^\frac{1}{11})\log l\\
		&\le l^{-70}\eta^{\frac{1}{3}}(y_0)\max\left(\varepsilon^{\frac{1}{11}}\eta^{-\frac{1}{3}}(y_0),2(\log M)^4\varepsilon^{\frac{1}{10}}l^3\right)-l^{-70}\varepsilon^{\frac{1}{11}}\log l\\
		&\le\frac{1}{2}\varepsilon^{\frac{1}{12}}.
	\end{aligned}
\end{equation}

\emph{Case 3}. $\gamma=(0,1)$, $l\le|y|\le L$. Let $\mu=0$, then we have $D_R-3\mu+3\mu\eta^{-1}=\beta_\tau J\partial_1\ovl{W}$, and $|F_q|\lesssim\varepsilon^{\frac{1}{12}}\eta^{-\frac{1}{3}}$. The rest is almost the same as Case 2. 

\subsection{Closure of bootstrap for $W$}
Similarly, for different $\gamma$ we choose different $\mu$, and we will use (\ref{estimate of q when y_0 ge l}) or (\ref{estimate of q when y_0 ge L}), depending on the location of $y$. 

\emph{Case 1}. $|\gamma|=2$, $|y|\ge l$. Now $R=\partial^\gamma W$, and we let
\begin{equation}
	\mu=\left\{\begin{aligned}
		&\frac{1}{3},\ \ \ \ \gamma=(2,0),(1,1)\\
		&\frac{1}{6},\ \ \ \ \gamma=(0,2).
	\end{aligned}\right.
\end{equation}
The damping term becomes
\begin{equation}
	3\mu-D_R=\left\{
	\begin{aligned}
		&-\gamma_1+\frac{1}{2}-\beta_\tau\left(1+\gamma_1\mathbbm{1}_{\gamma_1\ge2}\right)J\partial_1W,\ \ \ &\gamma_1\ge1\\
		&-\beta_\tau J\partial_1W,\ \ \ \ \ \ \ \ \ \ \ \ \ \ \ \ \ \ \ \ \ \ \ \ \ \ \ \ \ \ \ \ \ \ &\gamma_1=0.
	\end{aligned}\right.
\end{equation}

When $\gamma_1=0$, we have
\begin{equation}
	\begin{aligned}
		\int_{s'}^s\left(3\mu-D_R\right)\circ\Phi_W^{y_0}(s'')ds''&\le2\int_{s_0}^\infty|\partial_1W|\circ\Phi_W^{y_0}(s'')ds''\le-20\log l,
	\end{aligned}
\end{equation}
and the forcing term is bound by
\begin{equation}
	\begin{aligned}
		\int_{s_0}^s\left|\eta^{\frac{1}{6}}F_W^{(0,2)}\right|\circ\Phi_W^{y_0}(s')ds'&\lesssim M^{\frac{2}{3}}\int_{s_0}^s\left(\eta^{\frac{1}{6}}\eta^{-\frac{1}{3}+\frac{1}{3(k-3)}}\right)\circ\Phi_W^{y_0}(s')ds'\le -M^{\frac{5}{6}}\log l.
	\end{aligned}
\end{equation}
Thus, we have that
\begin{equation}
	\begin{aligned}
		\left|\eta^{\frac{1}{6}}\partial_{22}{W}\circ\Phi_W^{y_0}(s)\right|&\le l^{-30}\eta^{\frac{1}{6}}(y_0)\left|\partial_{22}{W}(y_0,s_0)\right|l^{-20}-l^{-30}M^{\frac{5}{6}}\log l\\
		&\le l^{-50}\eta^{\frac{1}{6}}(y_0)\max\left(\eta^{-\frac{1}{6}}(y_0),\frac{6}{7}\eta^{-\frac{1}{6}}(y_0)+2(\log M)^4\varepsilon^{\frac{1}{10}}l^2\eta^{-\frac{1}{6}}(y_0)\|\eta^{\frac{1}{6}}\|_{L^\infty(|y|\le l)}\right)-l^{-50}M^{\frac{5}{6}}\log l\\
		\overset{\varepsilon\text{ small}}&{\le}-2l^{-50}M^{\frac{5}{6}}\log l
		\overset{M\text{ large}}{\le}\frac{1}{2}M.
	\end{aligned}
\end{equation}

When $\gamma_1>0$, we have that
\begin{equation}
	\begin{aligned}
		\exp\left(\int_{s'}^s\left(3\mu-D_R\right)\circ\Phi_W^{y_0}(s'')ds''\right)&\le\exp\left\{3\int_{s'}^s|\partial_1W|\circ\Phi_W^{y_0}(s'')ds''+\int_{s'}^s\left(\frac{1}{2}-1\right)ds''\right\}\\
		&\le\exp\left\{4\int_{s'}^s\eta^{-\frac{1}{3}}\circ\Phi_W^{y_0}(s'')ds''-\frac{1}{2}(s-s')\right\}\le l^{-80}e^{-\frac{1}{2}(s-s')},
	\end{aligned}
\end{equation}
and $|F_q|=\left|\eta^{\frac{1}{3}}F_W^{(\gamma)}\right|\lesssim\eta^{\frac{1}{3}}M^{\frac{1}{3}\gamma_2}\eta^{-\frac{1}{3}}\le M^{\frac{1}{3}\gamma_2+\frac{1}{6}}$. Thus, we have the bound for $\partial^\gamma W$:
\begin{equation}
	\begin{aligned}
		\left|\eta^{\frac{1}{3}}\partial^\gamma W\right|\circ\Phi_W^{y_0}(s)&\le l^{-20}\eta^{\frac{1}{3}}(y_0)|\partial^\gamma W(y_0,s_0)|l^{-80}e^{-\frac{1}{2}(s-s_0)}+L^{-20}\int_{s_0}^sM^{\frac{1}{3}\gamma_2+\frac{1}{6}}l^{-80}e^{-\frac{1}{2}(s-s')}ds'\\
		&\le l^{-100}\eta^{\frac{1}{3}}(y_0)\max\left(\eta^{-\frac{1}{3}}(y_0),C\eta^{-\frac{1}{2}}(y_0)+2(\log M)^4\varepsilon^{\frac{1}{10}}l^2\eta^{-\frac{1}{3}}(y_0)\|\eta^{\frac{1}{3}}\|_{L^\infty(|y|\le l)}\right)e^{-\frac{1}{2}(s-s_0)}\\
		&\ \ \ +l^{-101}M^{\frac{1}{3}\gamma_2+\frac{1}{6}}\\
		&\le l^{-100}\max\left(1,C+3(\log M)^4\varepsilon^{\frac{1}{10}}l^2\right)+l^{-101}M^{\frac{1}{3}\gamma_2+\frac{1}{6}}\\
		&\le M^{\frac{1+\gamma_2}{3}}\underbrace{\left(CM^{-\frac{1}{3}}+l^{-101}M^{-\frac{1}{6}}\right)}_{<\frac{1}{2}\text{ when }M\text{ large}}\le\frac{1}{2}M^{\frac{1+\gamma_2}{3}}.
	\end{aligned}
\end{equation}

\emph{Case 2}. $|\gamma|=0$ and $|y|\ge L$. Let $\mu=-\frac{1}{6}$. Now we have $3\mu-D_R-3\mu\eta^{-1}=\frac{1}{2}\eta^{-1}$ and $F_q=\eta^{-\frac{1}{6}}\left(F_W-e^{-\frac{s}{2}}\beta_\tau\dot\kappa\right)$. And we bound the damping term and the forcing term by
\begin{equation}
	\begin{aligned}
		\int_{s'}^s\frac{1}{2}\eta^{-1}\circ\Phi_W^{y_0}(s'')ds''&\le\int_{s_0}^\infty\left(1+L^2e^{\frac{2}{5}(s''-s_0)}\right)^{-1}ds''\le L^{-2}\int_{s_0}^\infty e^{-\frac{2}{5}(s''-s_0)}ds''\le L^{-1}=\varepsilon^{\frac{1}{10}},
	\end{aligned}
\end{equation}
\begin{equation}
	\begin{aligned}
		\int_{s_0}^s\left|F_q\circ\Phi_W^{y_0}(s')\right|ds'&\lesssim\int_{s_0}^s\left(e^{-\frac{s'}{2}}+Me^{-\frac{s'}{2}}\right)\eta^{-\frac{1}{6}}\circ\Phi_W^{y_0}(s')ds'\lesssim M\int_{s_0}^s e^{-\frac{s'}{2}}ds'\le\varepsilon^\frac{1}{3}.
	\end{aligned}
\end{equation}
Thus, we have that
\begin{equation}
	\begin{aligned}
		\left|\eta^{-\frac{1}{6}}W\right|\circ\Phi_W^{y_0}(s)&\le e^\varepsilon\eta^{-\frac{1}{6}}(y_0)|W(y_0,s_0)|e^{\varepsilon^\frac{1}{10}}+e^\varepsilon\varepsilon^{\frac{1}{3}}e^{\varepsilon^{\frac{1}{10}}}\\
		&\le e^\varepsilon e^{\varepsilon^\frac{1}{10}}\eta^{-\frac{1}{6}}(y_0)\max\left(\eta^{\frac{1}{6}}(y_0)(1+\varepsilon^{\frac{1}{11}}),\eta^{\frac{1}{6}}(y_0)+\varepsilon^{\frac{1}{12}}\eta^{-\frac{1}{3}}(y_0)\right)+e^\varepsilon\varepsilon^{\frac{1}{3}}e^{\varepsilon^{\frac{1}{10}}}\\
		&\le 1+\varepsilon^{\frac{1}{19}}.
	\end{aligned}
\end{equation}

\emph{Case 3}. $\gamma=(1,0)$ and $|y|\ge L$. In this case, we can see that  $q=\eta^{\frac{1}{3}}\partial_1W$, $3\mu-D_R-3\mu\eta^{-1}=-\beta_\tau J\partial_1W-\eta^{-1}\le-\beta_\tau J\partial_1W$, and
\begin{equation}
	\begin{aligned}
		\int_{s'}^s\left(3\mu-D_R-3\mu\eta^{-1}\right)\circ\Phi_W^{y_0}(s'')ds''&\le2\int_{s'}^s|\partial_1W|\circ\Phi_W^{y_0}(s'')ds''\\
		&\lesssim\int_{s_0}^\infty\left(1+L^2e^{\frac{2}{5}(s''-s_0)}\right)^{-\frac{1}{3}}ds''\lesssim L^{-\frac{2}{3}}\le\varepsilon.
	\end{aligned}
\end{equation}
The forcing term is bound by
\begin{equation}
	\begin{aligned}
		\int_{s_0}^s\left|F_q\circ\Phi_W^{y_0}(s')\right|ds'&\lesssim\int_{s_0}^s\varepsilon^{\frac{1}{4}}\left|\eta^{\frac{1}{3}}\eta^{-\frac{1}{2}+\frac{2}{3(k-2)}}\right|\circ\Phi_W^{y_0}(s')ds'\lesssim \varepsilon^{\frac{1}{4}}\int_{s_0}^s \eta^{-\frac{1}{12}}\circ\Phi_W^{y_0}(s')ds'\lesssim\varepsilon^{\frac{1}{12}}.
	\end{aligned}
\end{equation}
Thus we have the bound
\begin{equation}
	\begin{aligned}
		\left|\eta^{\frac{1}{3}}\partial_1W\right|\circ\Phi_W^{y_0}(s)&\le e^\varepsilon\eta^{\frac{1}{3}}(y_0)|\partial_1W(y_0,s_0)|e^{\varepsilon}+Ce^\varepsilon\varepsilon^{\frac{1}{4}}e^{\varepsilon}\\
		&\le e^\varepsilon e^{2\varepsilon}\eta^{\frac{1}{3}}(y_0)\max\left(\eta^{-\frac{1}{3}}(y_0)(1+\varepsilon^{\frac{1}{112}}),\eta^{-\frac{1}{3}}(y_0)+\varepsilon^{\frac{1}{12}}\eta^{-\frac{1}{3}}(y_0)\right)+Ce^{2\varepsilon}\varepsilon^{\frac{1}{4}}\\
		&\le 1+\varepsilon^{\frac{1}{13}}.
	\end{aligned}
\end{equation}

\emph{Case 4}. $\gamma=(0,1)$ and $|y|\ge L$. Let $\mu=0$, we have $q=R=\partial_2W$, and $3\mu-D_R-3\mu\eta^{-1}=-\beta_\tau J\partial_1W$. Thus we have the bound for damping term:
\begin{equation}
	\int_{s'}^s\left(3\mu-D_R-3\mu\eta^{-1}\right)\circ\Phi_W^{y_0}(s'')ds''\le\varepsilon.
\end{equation}
The forcing term is bound by
\begin{equation}
	\begin{aligned}
		\int_{s_0}^s\left|F_q\circ\Phi_W^{y_0}(s')\right|ds'&\lesssim\int_{s_0}^sM^2\varepsilon^{\frac{1}{6}}\eta^{-\frac{1}{3}}\circ\Phi_W^{y_0}(s')ds'\le\varepsilon^{\frac{1}{8}}.
	\end{aligned}
\end{equation}
Finally we have that
\begin{equation}
	\begin{aligned}
		\left|\partial_2W\right|\circ\Phi_W^{y_0}(s)&\le e^\varepsilon|\partial_2W(y_0,s_0)|e^{\varepsilon}+e^\varepsilon\varepsilon^{\frac{1}{8}}e^{\varepsilon}\le e^{2\varepsilon}\max\left(\frac{3}{4},\frac{2}{3}+\varepsilon^{\frac{1}{13}}\right)+e^{2\varepsilon}\varepsilon^{\frac{1}{8}}\le \frac{5}{6}.
	\end{aligned}
\end{equation}

\section{Proof of the main theorem}\label{proof of the main theorem}
In this section we prove the main theorem, discuss the H\"older regularity of $w$ and deduce a lower bound of vorticity. 
\begin{proof}[Proof of the main theorem]
	The local well-posedness of $(u,\sigma)$ in physical variables implies the local well-posedness of $(W,Z,A,\kappa,\tau,\xi,n,\phi)$ in self-similar variables, and the global existence of $(W,Z,A,\kappa,\tau,\xi,n,\phi)$ in self-similar variables is obtained via the bootstrap bound. 
	
	Now we prove the solution has the desired blow-up behavior. From the bootstrap assumptions and  $\tau(t)-t=\int_t^{T_*}(1-\dot\tau(t'))dt'$ we can see that $c(T_*-t)\le\tau-t=e^{-s}\le C(T_*-t)$. Since $R(t)\in SO(2)$, using (\ref{estimates of f-dependent functions, inequalities}) and (\ref{estimates of U,S}), we have that 
	\begin{equation}
		\begin{aligned}
			\left|[(R(t)N)\cdot\nabla_{\mathrm{x}}]u\right|&=|N\cdot\nabla_{\tilde{x}}\tilde{u}|=\left|\left(\frac{\sqrt{1+f_{x_2}^2}}{1+f_{x_1}}\partial_{x_1}-\frac{f_{x_2}}{\sqrt{1+f_{x_2}^2}}\partial_{x_2}\right)\mathring{u}\right|\le (1+\varepsilon^{\frac{2}{3}})(1+\varepsilon^{\frac{3}{4}})e^s+\varepsilon\le\frac{C}{T_*-t}.
		\end{aligned}
	\end{equation}
	Similarly, we can see that the derivative of $u$ that is aligned to the shock is bounded:
	\begin{equation}
		\begin{aligned}
			\left|[(R(t)T)\cdot\nabla_{\mathrm{x}}]u\right|&=|T\cdot\nabla_{\tilde{x}}\tilde{u}|=\left|\frac{1}{\sqrt{1+f_{x_2}^2}}\partial_{x_2}\mathring{u}\right|\le1+\varepsilon^{\frac{1}{2}}.
		\end{aligned}
	\end{equation}
	In a same way, we can prove that $\left|[(R(t)N)\cdot\nabla_{\mathrm{x}}]\sigma\right|\le\frac{C}{T_*-t}$ and $\left|[(R(t)T)\cdot\nabla_{\mathrm{x}}]u\right|\le C$. Consequently, we have that
	\begin{equation}
		|\nabla_\mathrm{x}u(t)|\le\left|[(R(t)N)\cdot\nabla_{\mathrm{x}}]u\right|+\left|[(R(t)T)\cdot\nabla_{\mathrm{x}}]u\right|\le\frac{C}{T_*-t},
	\end{equation}
	\begin{equation}
		|\nabla_\mathrm{x}\sigma(t)|\le\left|[(R(t)N)\cdot\nabla_{\mathrm{x}}]\sigma\right|+\left|[(R(t)T)\cdot\nabla_{\mathrm{x}}]\sigma\right|\le\frac{C}{T_*-t}.
	\end{equation}
	From the bootstrap assumptions $|\dot\xi|\le M^\frac{1}{4}$ and $|\dot n_2|\le M^2\varepsilon^{\frac{1}{2}}$, we know that both $\xi$ and $n$ have limits as $t\rightarrow T_*$. 
	
	Next, by the definition of $n$ and $N$, and the coordinate transformations, we have that $n(t)=R(t)N(0,t)$. Also we can see that
	\begin{equation}
		\begin{aligned}
			\left|[(R(t)N)\cdot\nabla_{\mathrm{x}}]u(\xi(t),t)\right|&=\left|\left(\frac{\sqrt{1+f_{x_2}^2}}{1+f_{x_1}}\partial_{x_1}-\frac{f_{x_2}}{\sqrt{1+f_{x_2}^2}}\partial_{x_2}\right)\mathring{u}(0,t)\right|\\
			&=\left|\frac{-e^s+\partial_{x_1}z(0,t)}{2}\tilde{e}_1+\partial_{{x_1}}a(0,t)\tilde{e}_2\right|\ge(\frac{1}{2}-\varepsilon^{\frac{1}{2}})e^s.
		\end{aligned}
	\end{equation}
	Similarly, we have that
	\begin{equation}
		\begin{aligned}
			\left|[(R(t)N)\cdot\nabla_{\mathrm{x}}]\sigma(\xi(t),t)\right|=\left|\partial_{{x_1}}\mathring{\sigma}(0,t)\right|=\left|\frac{-e^s-\partial_{{x_1}}z(0,t)}{2}\right|\ge(\frac{1}{2}-\varepsilon^{\frac{1}{2}})e^s.
		\end{aligned}
	\end{equation}
	Thus, we can conclude that  $\|\nabla_{\mathrm{x}}u\|_{L^\infty}\ge\left|[(R(t)N)\cdot\nabla_{\mathrm{x}}]u(\xi(t),t)\right|\ge\frac{c}{T_*-t}$, and $\|\nabla_{\mathrm{x}}\sigma\|_{L^\infty}\ge\left|[(R(t)N)\cdot\nabla_{\mathrm{x}}]\sigma(\xi(t),t)\right|\ge\frac{c}{T_*-t}$. 
	
	Next, we prove (\ref{boundedness of u and sigma away from origin}). This follows from that $\|\partial_{{x_1}}w\|_{L^\infty(B_x(0,\delta))}\le C(\delta)$. From (\ref{refined bootstrap inequality of W}), we have that
	\begin{equation}
		\begin{aligned}
			\|\partial_{{x_1}}w\|_{L^\infty(B_x(0,\delta))}&\le(1+\varepsilon^{\frac{1}{13}})e^s\left\|\frac{1}{(1+y_1^2+y_2^6)^{1/3}}\right\|_{L^\infty_y(\{e^{-3s}y_1^2+e^{-s}y_2^2\le\delta^2\}^c)}\\
			&\le2\delta^{-2}(1+\varepsilon^{\frac{1}{13}})\frac{e^s}{(1+e^{3s})^{\frac{1}{3}}}\le3\delta^{-2}.
		\end{aligned}
	\end{equation}
	
	Now we have complete the proof of the main shock formation result and (\ref{blow up speed for u})-(\ref{limit of n}). The H\"older bound is left to the next subsection. 
\end{proof}

\subsection{H\"older regularity for $w$}
We now prove that Riemann invariant $w$ posseses a uniform 1/3-H\"older bound up to the blow-up time. 
\begin{proposition}
	For the Riemann variable $w$, we have that $w\in L^\infty([-\varepsilon,T_*);C^{1/3})$. 
\end{proposition}
\begin{proof}
	The proof of this proposition is the same as that in \cite{buckmaster2022formation}, and for the reader's convenience we outline the proof here. 
	
	Using the bootstrap assuptions we directly compute the $C^{1/3}$ norm:
	\begin{equation}
		\begin{aligned}
			\frac{|w(x_1,x_2,t)-w(x_1',x_2',t)|}{|x-x'|^{1/3}}&=\frac{e^{-\frac{s}{2}}|W(y,s)-W(y',s)|}{[e^{-3s}(y_1-y_1')^2+e^{-s}(y_2-y_2')^2]^{1/6}}\\
			&\le\frac{|W(y_1,y_2,s)-W(y_1',y_2,s)|}{|y_1-y_1'|^{1/3}}+e^{-\frac{s}{3}}\frac{|W(y_1',y_2,s)-W(y_1',y_2',s)|}{|y_2-y_2'|^{1/3}}\\
			&\lesssim\frac{\int_{y_1'}^{y_1}(1+z^2)^{-1/3}dz}{|y_1-y_1'|^{1/3}}+e^{-\frac{s}{3}}|y_2-y_2'|^{2/3}
			\overset{y\in\mathcal{X}(s)}{\lesssim}1.
		\end{aligned}
	\end{equation}
	Now we have proved $w$ is uniformly H\"older-1/3 continuous with respect to $x$, and one can check that the transformation $\tilde{x}\mapsto x$, $\mathrm{x}\mapsto\tilde{x}$ do not affect the H\"older-1/3 continuity of $w$. 
\end{proof}

\subsection{Discussion of the vorticity}
From (\ref{evolution of specific vorticity in tilde x coordinate}), we know that in  $\tilde{x}$-coordinate, the specific vorticity $\tilde{\zeta}$ is purely transported by $\tilde{u}+\tilde{v}$. From (\ref{estimate of of V})(\ref{estimates of U,S}) and the estimate (\ref{estimates of f-dependent functions, inequalities}) of $|f|$, we can deduce that $|\tilde{u}+\tilde{v}|\lesssim M^\frac{1}{4}$ on $\{|\tilde{x}_1|\le 10\varepsilon^{\frac{1}{2}},|\tilde{x}_2|\le10\varepsilon^\frac{1}{6}\}\supset B_{\tilde x}(0,\varepsilon^{\frac{3}{4}})$. Note that $|T_*-t_0|=|T_*+\varepsilon|\lesssim\varepsilon$, we have that if $\tilde{\zeta}(\tilde{x},t_0)\ge c_0$ for some $c_0>0$ on $B_{\tilde x}(0,\varepsilon^{\frac{3}{4}})$, then $\tilde{\zeta}(\tilde{x},t)\ge c_0/2$ on $B_{\tilde x}(0,\varepsilon^{\frac{3}{4}}/2)$. 

From the bootstrap assumptions and (\ref{closure of kappa bootstrap}) we have that 
$$\left|S-\frac{\kappa_0}{2}\right|\lesssim|\kappa-\kappa|+e^{-\frac{s}{2}}|W|+|Z|\lesssim M\varepsilon+\varepsilon^{\frac{1}{6}}\lesssim\varepsilon^{\frac{1}{6}}.$$
Thus the sound speed $\tilde{\sigma}\ge\frac{\kappa_0}{4}$, and $|\tilde\omega|=|\tilde\zeta||\tilde\rho|=|\zeta|(\alpha|\sigma|)^{1/\alpha}\ge\frac{c_0}{2}\cdot(\frac{\alpha\kappa_0}{4})^{1/\alpha}$ on $B_{\tilde x}(0,\varepsilon^{\frac{3}{4}}/2)$. 

The initial conditions stated in subsection \ref{Initial data in physical variables} can not rule out the possibility that $\tilde{\zeta}(\tilde{x},t_0)$ have a positive lower bound on $B_{\tilde x}(0,\varepsilon^{\frac{3}{4}})$, thus there do exist solutions satisfying the listed initial condition and present non-zero vorticity at the blow-up point. 

\section*{Data availability statement}
Data sharing is not applicable to this article as no new data were created or analyzed in this study.

\section*{Acknowledgements}
The author is supported by the China Scholarship Council (File No.202106100096), and thank for the warm host of the department of mathematics of the National University of Singapore. The author is grateful to Prof. Xinliang An and Dr. Haoyang Chen for valuable instruction, discussions and suggestions; the author would also like to thank Prof. Lifeng Zhao and Yiya Qiu for helpful correspondence. 

\appendix
\section{Toy model of 1D Burgers profile}\label{universality}
Consider the following cauchy problem for the 1D Burgers equation:  
\begin{equation}
	\left\{
	\begin{aligned}
		&u_t+uu_x=0\\
		&u(x,0)=u_0(x):=-xe^{-x^2}.
	\end{aligned}
	\right.
	\label{original burgers equation}
\end{equation}
It is well-known for the Burgers equation that the blow-up time is $T=-\frac{1}{\inf_{x\in\mathbb{R}} \partial_xu_0}=1$, and the blow-up point is $(x,t)=(0,1)$, and 
\begin{equation}
	\|\partial_xu(\cdot,t)\|_{L^\infty}\le\frac{1}{1-t}.
	\label{upper bound for Dxu, appendix}
\end{equation}

Now we claim that $\frac{1}{\sqrt{1-t}}u\left((1-t)^{\frac{3}{2}}y,t\right)$
converges uniformly to a profile(a fixed stationary function) $\ovl U(y)$ on any compact set as $t\rightarrow1$. This fact characterizes the blow-up behavior of $u$. We can formally write this fact as
\begin{equation}
	u(x,t)\sim(1-t)^{\frac{1}{2}}\ovl U\left((1-t)^{-\frac{3}{2}}x\right),\ \ \ \text{as }t\rightarrow1.
\end{equation}
To closely investigate this fact, we use the "self-similar transformation" $y=(1-t)^{-\frac{3}{2}}x$. $y$ is a "zoom-in" version of $x$ in the sense that any compact set of $y$ correspond to a set of $x$ that converging to 0, thus in $y$-coordinate we can observe the behavior of $u$ near the blow-up point in detail as $t\rightarrow1$.

For the sake of convenience we introduce the self-similar time $s=-\log(1-t)$, thus $1-t=e^{-s}$, this ``self-similar time" has the advantage that $t\rightarrow1$ is equivalent to $s\rightarrow\infty$. Now the self-similar transformation becomes $y=e^{\frac{3}{2}s}x$ and we can rewrite $\frac{1}{\sqrt{1-t}}u\left((1-t)^{\frac{3}{2}}y,t\right)$ in the self-similar coordinate as
\begin{equation}
	U(y,s):=\frac{1}{\sqrt{1-t}}u\left((1-t)^{\frac{3}{2}}y,t\right)=e^{\frac{s}{2}}u(e^{-\frac{3}{2}s}y,1-e^{-s}).
\end{equation}
In this coordinate, the proposition we claimed becomes
\begin{equation}
	\boxed{U(y,s)\overset{s\rightarrow\infty}{\rightrightarrows}\ovl U(y)\ \ \ \ \ y\in K,\ \text{for all compact }K.}
	\label{proposition of convergence, appendix}
\end{equation}
The following figures show the graphs of $U$ and how $U$ converges to $\ovl U$: 

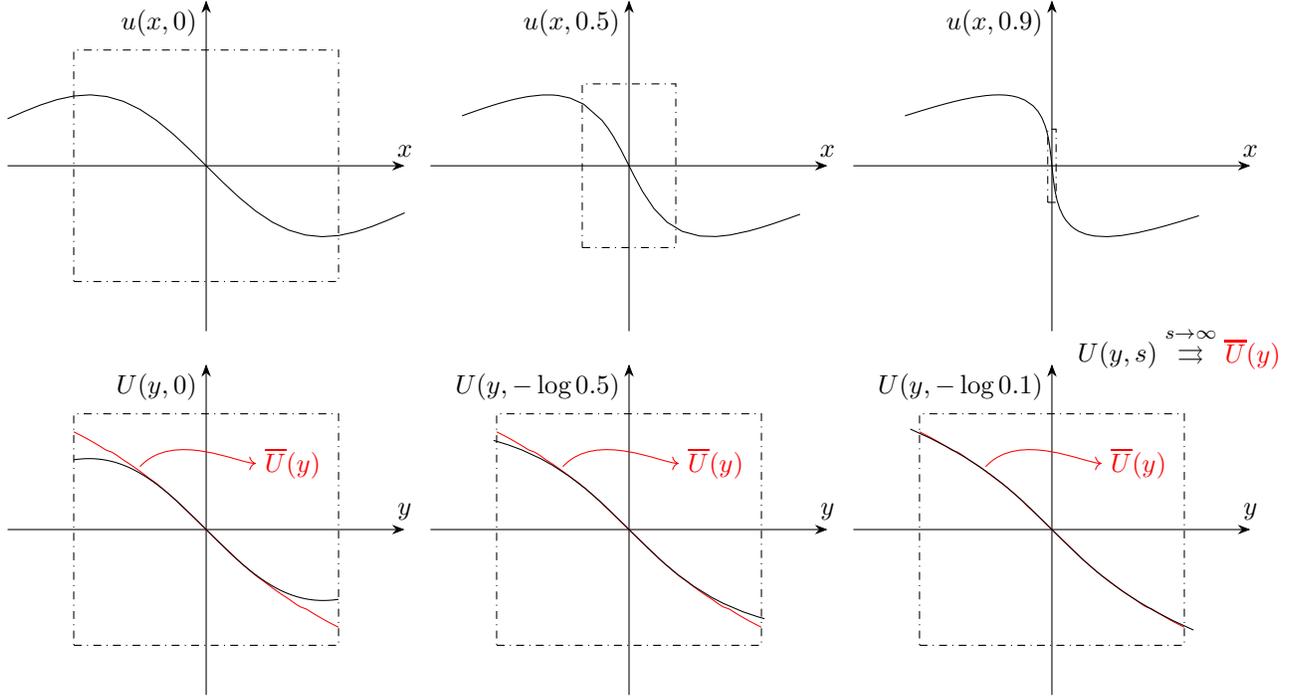
\begin{tikzpicture}[scale=2.2]
	\draw[-Stealth] (-1.2,0)--(1.2,0) node[above] {$x$};
	\draw[-Stealth] (0,-1)--(0,1) node[below left] {$u(x,0)$};
	\draw[domain =-1.2:1.2] plot (\x,{-\x*exp(-(\x)^2)});
	
	\draw[dash dot](-0.8,-0.7) rectangle (0.8,0.7);
	
	\draw[yshift=-2.2cm][-Stealth] (-1.2,0)--(1.2,0) node[above] {$y$};
	\draw[yshift=-2.2cm][-Stealth] (0,-1)--(0,1) node[below left] {$U(y,0)$};
	\draw[yshift=-2.2cm][->][color=red] (-0.4,0.38)..controls(-0.2,0.6) and (0.2,0.4)..(0.3,0.4) node[right]{$\ovl U(y)$};
	\draw[yshift=-2.2cm][domain =-0.8:0.8][smooth,color=red] plot[samples=50] (\x,{(-\x/2+(1/27+(\x)^2/4)^(1/2))^(1/3)-(\x/2+(1/27+(\x)^2/4)^(1/2))^(1/3)});
	\draw[yshift=-2.2cm][domain =-0.8:0.8] plot[samples=100] (\x,{-\x*exp(-(\x)^2)});
	\draw[yshift=-2.2cm][dash dot](-0.8,-0.7) rectangle (0.8,0.7);
\end{tikzpicture} 
\begin{tikzpicture}[scale=2.2]
	\draw[-Stealth] (-1.2,0)--(1.2,0) node[above] {$x$};
	\draw[-Stealth] (0,-1)--(0,1) node[below left] {$u(x,0.5)$};
	\draw (-1.00897761,0.30204478
	)--(-0.984594728,0.310810543
	)--(-0.960261711,0.319476578
	)--(-0.935991496,0.328017007
	)--(-0.911797411,0.336405178
	)--(-0.887693131,0.344613737
	)--(-0.839810203,0.360379595
	)--(-0.792457518,0.375084964
	)--(-0.722680057,0.394639887
	)--(-0.677168674,0.405662652
	)--(-0.610702401,0.418595199
	)--(-0.526015623,0.427968753
	)--(-0.46587807,0.428243859
	)--(-0.427546956,0.424906089
	)--(-0.355372838,0.409254325
	)--(-0.27386368,0.37227264
	)--(-0.162910322,0.274179356
	)--(-0.126716702,0.226566596
	)--(-0.092869267,0.174261466
	)--(-0.050497508,0.099004983
	)--(0,0
	)--(0.030107806,-0.059784388
	)--(0.082022008,-0.155955984
	)--(0.150556768,-0.258886464
	)--(0.229571242,-0.340857516
	)--(0.321600707,-0.396798586
	)--(0.427546956,-0.424906089
	)--(0.526015623,-0.427968753
	)--(0.610702401,-0.418595199
	)--(0.722680057,-0.394639887
	)--(0.792457518,-0.375084964
	)--(0.863692644,-0.352614712
	)--(0.960261711,-0.319476578
	)--(1.033397838,-0.293204324
	);
	\draw[dash dot](-0.2828,-0.49497) rectangle (0.2828,0.49497);
	
	\draw[yshift=-2.2cm][-Stealth] (-1.2,0)--(1.2,0) node[above] {$y$};
	\draw[yshift=-2.2cm][-Stealth] (0,-1)--(0,1) node[below left] {$U(y,-\log0.5)$};
	\draw[yshift=-2.2cm][->][color=red] (-0.4,0.38)..controls(-0.2,0.6) and (0.2,0.4)..(0.3,0.4) node[right]{$\ovl U(y)$};
	\draw[yshift=-2.2cm][domain =-0.8:0.8][smooth,color=red] plot[samples=50] (\x,{(-\x/2+(1/27+(\x)^2/4)^(1/2))^(1/3)-(\x/2+(1/27+(\x)^2/4)^(1/2))^(1/3)});
	\draw[yshift=-2.2cm] (-0.818513421,0.539131599
	)--(-0.731778485,0.51272945
	)--(-0.649325529,0.482045321
	)--(-0.533324394,0.428340829
	)--(-0.460779974,0.387748163
	)--(-0.358408957,0.320413553
	)--(-0.262673953,0.246442929
	)--(-0.172131878,0.167279377
	)--(-0.11385885,0.11241532
	)--(0,0
	)--(0.11385885,-0.11241532
	)--(0.201832718,-0.194147079
	)--(0.293933134,-0.271752291
	)--(0.391730219,-0.343660833
	)--(0.496595588,-0.408501091
	)--(0.5710017,-0.447232065
	)--(0.690024929,-0.497914463
	)--(0.774603461,-0.526473017
	)--(0.818513421,-0.539131599
	);
	\draw[yshift=-2.2cm][dash dot](-0.8,-0.7) rectangle (0.8,0.7);
\end{tikzpicture}
\begin{tikzpicture}[scale=2.2]
	\draw[-Stealth] (-1.2,0)--(1.2,0) node[above] {$x$};
	\draw[-Stealth] (0,-1)--(0,1) node[below left] {$u(x,0.9)$};
	\draw (-0.888159698,0.30204478
	)--(-0.804784693,0.328017007
	)--(-0.749847637,0.344613737
	)--(-0.695658365,0.360379595
	)--(-0.642423532,0.375084964
	)--(-0.590354648,0.388494835
	)--(-0.539664966,0.40037226
	)--(-0.466677978,0.414802246
	)--(-0.39795691,0.424492323
	)--(-0.354828122,0.427968753
	)--(-0.314045372,0.428838476
	)--(-0.275755404,0.42693844
	)--(-0.240080662,0.422132598
	)--(-0.207115922,0.414315642
	)--(-0.162881273,0.396798586
	)--(-0.113700976,0.362554471
	)--(-0.08398482,0.328905756
	)--(-0.060031497,0.288853892
	)--(-0.041295585,0.243004906
	)--(-0.031354983,0.209605574
	)--(-0.02316468,0.174261466
	)--(-0.016445555,0.137282716
	)--(-0.00619405,0.059784388
	)--(0,0
	)--(0.004057554,-0.039936051
	)--(0.010895515,-0.099004983
	)--(0.019639614,-0.155955984
	)--(0.031354983,-0.209605574
	)--(0.047002182,-0.258886464
	)--(0.067405566,-0.302882705
	)--(0.08398482,-0.328905756
	)--(0.113700976,-0.362554471
	)--(0.136898752,-0.381223609
	)--(0.162881273,-0.396798586
	)--(0.207115922,-0.414315642
	)--(0.25758452,-0.424906089
	)--(0.334133793,-0.42874023
	)--(0.420349455,-0.421833939
	)--(0.514903613,-0.405662652
	)--(0.616230021,-0.381966644
	)--(0.695658365,-0.360379595
	)--(0.77723534,-0.336405178
	)--(0.888159698,-0.30204478
	);
	\draw[dash dot](-0.025298,-0.22136) rectangle (0.025298,0.22136);
	
	\draw[yshift=-2.2cm][-Stealth] (-1.2,0)--(1.2,0) node[above] {$y$};
	\draw[yshift=-2.2cm][-Stealth] (0,-1)--(0,1) node[below left] {$U(y,-\log0.1)$};
	\draw[yshift=-2.2cm][->][color=red] (-0.4,0.38)..controls(-0.2,0.6) and (0.2,0.4)..(0.3,0.4) node[right]{$\ovl U(y)$};
	\draw[yshift=-2.2cm][domain =-0.8:0.8][smooth,color=red] plot[samples=50] (\x,{(-\x/2+(1/27+(\x)^2/4)^(1/2))^(1/3)-(\x/2+(1/27+(\x)^2/4)^(1/2))^(1/3)});
	\draw[yshift=-2.2cm][dash dot](-0.8,-0.7) rectangle (0.8,0.7);
	\draw[yshift=-2.2cm] (-0.855645957,0.607656596
	)--(-0.732531513,0.551063142
	)--(-0.621059131,0.493176125
	)--(-0.52005412,0.434126067
	)--(-0.428300661,0.374048059
	)--(-0.344546436,0.313081247
	)--(-0.267507458,0.251368297
	)--(-0.195873075,0.189054836
	)--(-0.128311122,0.126288882
	)--(-0.063473192,0.06322026
	)--(0,0
	)--(0.063473192,-0.06322026
	)--(0.128311122,-0.126288882
	)--(0.195873075,-0.189054836
	)--(0.267507458,-0.251368297
	)--(0.344546436,-0.313081247
	)--(0.428300661,-0.374048059
	)--(0.52005412,-0.434126067
	)--(0.621059131,-0.493176125
	)--(0.732531513,-0.551063142
	)--(0.855645957,-0.607656596
	);
	
	\draw(0.1,-1.1) node[right] {$U(y,s)\overset{s\rightarrow\infty}{\rightrightarrows}$};
	\draw[color=red](0.99,-1.14) node[right]{$\ovl U(y)$};
\end{tikzpicture} 

We now prove the convergence. Firstly, from (\ref{upper bound for Dxu, appendix}) and the self-similar transformation we have that
\begin{equation}
	\|\partial_yU(\cdot,s)\|_{L^\infty}\le1.
	\label{estimate of DU, appendix}
\end{equation}

From chain rule we can deduce that $U(s,y)$ satisfies
\begin{equation}
	\left\{
	\begin{aligned}
		&\left(\partial_s-\frac{1}{2}\right)U+\left(\frac{3}{2}y+U\right)\partial_yU=0\\
		&U(y,0)=U_0(y)=u_0(y)=-ye^{-y^2}.
		\label{evolution of U, appendix}
	\end{aligned}
	\right.
\end{equation}
Ignoring the time-dependent term in the above equation, we have
\begin{equation}
	-\frac{1}{2}W+\left(\frac{3}{2}y+W\right)\partial_yW=0.
	\label{equation for profile, appendix}
\end{equation}
which is called the self-similar Burgers equation. Using ODE techniques we can find a first integral of this equation: $y=-W_C(y)-C W_C(y)^{3}$. If we impose the constraint $W_C'''(0)=6=u_0'''(0)$, we have $C=1$. Thus we select $\ovl U$ to be the function that implicitly determined by the identity $y=-\ovl U(y)-\ovl U(y)^{3}$, the solution of this cubic equation is
\begin{equation}
	\ovl U(y)=\left(-\frac{y}{2}+\left(\frac{1}{27}+\frac{y^2}{4}\right)^{\frac{1}{2}}\right)^{\frac{1}{3}}-\left(\frac{y}{2}+\left(\frac{1}{27}+\frac{y^2}{4}\right)^{\frac{1}{2}}\right)^{\frac{1}{3}}.
\end{equation}
One can verify that $\ovl U(0)=U_0(0)$, $\ovl U'(0)=U_0'(0)$, $\ovl U''(0)=U_0''(0)$, and $\ovl U'''(0)=U_0'''(0)$. Thus we can check by the above explicit expression of $\ovl U$ that
\begin{equation}
	\left|\ovl U(y)-U_0(y)\right|=\left|\ovl U(y)+ye^{-y^2}\right|\le My^4.
	\label{estimate of initial data of tilde U, appendix}
\end{equation} 
holds for some $M>0$, and $-1\le\ovl U_y\le0$.

Now we are ready to prove (\ref{proposition of convergence, appendix}). Define $\widetilde{U}(y)=U(y)-\ovl U(y)$, then subtract (\ref{equation for profile, appendix}) from (\ref{evolution of U, appendix}) and we have
\begin{equation}
	\left\{
	\begin{aligned}
		&\partial_s\widetilde{U}-\widetilde{D}\widetilde{U}+\left(\frac{3}{2}y+U\right)\widetilde{U}_y=0\\
		&\widetilde{D}=\frac{1}{2}-\ovl U_y\\
		&\widetilde{U}(y,0)=\widetilde{U}_0(y):=u_0(y)-\ovl U(y).
	\end{aligned}
	\right.
	\label{evolution of tilde U, appendix}
\end{equation}

Notice that (\ref{evolution of U, appendix}) is a transport equation. We define its Lagrange trajectories by
\begin{equation}
	\left\{
	\begin{aligned}
		&\frac{d}{ds}\Phi_{y_0}(s)=\left(\frac{3}{2}y+U\right)\circ\Phi_{y_0}(s)\\
		&\Phi_{y_0}(0)=y_0.
	\end{aligned}
	\right.
\end{equation}
From (\ref{estimate of DU, appendix}) we have $|U(y)|\le|y|$, and $\left(\frac{3}{2}y+U\right)\cdot y\ge\frac{1}{2}y^2$, thus
\begin{equation}
	\frac{1}{2}\frac{d}{ds}\left|\Phi_{y_0}(s)\right|^2=\Phi_{y_0}(s)\frac{d}{ds}\Phi_{y_0}(s)=\left[\left(\frac{3}{2}y+U\right)\cdot y\right]\circ\Phi_{y_0}(s)\ge\frac{1}{2}\left|\Phi_{y_0}(s)\right|^2.
\end{equation}
If $\Phi_{y_0}(s)=y$, from the above inequality we have $e^{-s}|y|^2\ge|y_0|^2$. Rewrite (\ref{evolution of tilde U, appendix}) in terms of Lagrange trajectories we have
\begin{equation}
	\frac{d}{ds}\widetilde{U}\circ\Phi_{y_0}(s)=\left(\frac{1}{2}-\ovl U_y\right)\circ\Phi_{y_0}(s)\cdot\widetilde{U}\circ\Phi_{y_0}(s).
\end{equation}
From $-1\le\ovl U_y\le0$ we have
\begin{equation}
	\frac{d}{ds}\left|\widetilde{U}\circ\Phi_{y_0}(s)\right|\le\frac{3}{2}\left|\widetilde{U}\circ\Phi_{y_0}(s)\right|.
	\label{estimate of tilde U along trajectories}
\end{equation}
thus we can conclude that
\begin{equation}
	\begin{aligned}
		\left|\widetilde{U}(y,s)\right|&=\left|\widetilde{U}\circ\Phi_{y_0}(s)\right|\\
		\overset{(\ref{estimate of tilde U along trajectories})}&{\le} e^{\frac{3}{2}s}\left|\widetilde{U}\circ\Phi_{y_0}(0)\right|\\
		&=e^{\frac{3}{2}s}\left|\widetilde{U}(y_0,0)\right|\\
		&\le e^{\frac{3}{2}s}My_0^4\\
		\overset{\ref{estimate of initial data of tilde U, appendix}}&{\le}Me^{-\frac{s}{2}}y^4.
	\end{aligned}
\end{equation}
From this inequality we know that $\widetilde{U}$ converge to 0 uniformly on any compact set, or equivalently it holds that $U\rightrightarrows\ovl U$ on any compact set. 

Though we prove the convergence in the case of a specific initial datum, the proof can be modified to apply to almost all initial data. In fact, take any $u_0\in C_c^\infty$, there exists a point $x_0\in\mathbb{R}$ and an integer $k\ge1$, such that $u_0'(x_0)=\inf_{x\in\mathbb{R}} u_0'(x)$, and $u_0^{(j)}(x_0)=0$ holds for $2\le j\le2k$, while $u_0^{(2k+1)}(x_0)>0$. In this case, a rescaled version of the solution $u$ will eventually converge to a solution $\ovl U$ of the self-similar Burgers equation $-\frac{1}{2k}\ovl U(y)+\left[(1+\frac{1}{2k})y+\ovl U(y)\right]\ovl U_y(y)=0$. In this sense, the self-similar Burgers equation plays a universal role in the blow-up of the Burgers equation. 

\section{Interpolation}
Here we state the interpolation inequalities that are used in this paper.
\begin{lemma}[Gagliardo-Nirenberg inequalities]
	Suppose $1\le q,r\le\infty$, $1\le p<\infty$, $j<m$ are non-negative integers, $\theta\in[0,1]$, and they satisfy the relations
	\begin{equation}
		\frac{1}{p}=\frac{j}{n}+\theta\left(\frac{1}{r}-\frac{m}{n}\right)+\frac{1-\theta}{q},\ \ \ \ \ \frac{j}{m}\le\theta\le 1.
	\end{equation}
	Then $\|D^ju\|_{L^p(\mathbb{R}^n)}\le C\|D^mu\|_{L^r(\mathbb{R}^n)}^\theta\|u\|_{L^q(\mathbb{R}^n)}^{1-\theta}$ holds for any $u\in L^q(\mathbb{R}^n)$ such that $D^mu\in L^r(\mathbb{R}^n)$, with two exceptional cases: 
	
	(1)if $j=0$, $q=\infty$ and $rm<n$, then an additional assumption is needed: either $u$ tends to 0 at infinity or $u\in L^s(\mathbb{R}^n)$ for some finite value of $s$; 
	
	(2)if $r>1$ and $m-j-\frac{n}{r}$ is a non-negative integer, then the additional assumption $\frac{j}{m}\le\theta<1$ is needed. 
\end{lemma}

A mainly used special case in this paper is that
\begin{equation}
	\|D^j\varphi\|_{L^{\frac{2m}{j}}(\mathbb{R}^n)}\lesssim\|\varphi\|_{\dot H^m(\mathbb{R}^n)}^{\frac{j}{m}}\|\varphi\|_{L^\infty(\mathbb{R}^n)}^{1-\frac{j}{m}}.
	\label{special case 1 of lemma A.1}
\end{equation}
holds for any $u\in\dot H^m(\mathbb{R}^n)\cap L^\infty(\mathbb{R}^n)$. 

\begin{lemma}\label{interpolation lemma of L2 norm of product}Suppose $k\ge 4$, $0\le l\le k-3$ are integers, $q\in(4,2(k-1)]$, then
	\begin{equation}
		\left\|D^{2+l}\phi D^{k-1-l}\varphi\right\|_{L^2(\mathbb{R}^2)}\lesssim_{k,q}\|D^k\phi\|_{L^2(\mathbb{R}^2)}^a\|D^2\phi\|_{L^q(\mathbb{R}^2)}^{1-a}\|D^k\varphi\|_{L^2(\mathbb{R}^2)}^b\|D^2\varphi\|_{L^q(\mathbb{R}^2)}^{1-b},
		\label{interpolation of L2 norm of product}
	\end{equation}
	holds for any $\phi,\varphi\in\dot{H}^k(\mathbb{R}^2)\cap\dot{W}^{2,q}(\mathbb{R}^2)$, where $a$, $b$ are given by
	\begin{equation}
		a=\frac{\frac{1}{q}-\frac{1}{p}+\frac{l}{1}}{\frac{k}{2}+\frac{1}{q}-\frac{3}{2}},\ b=\frac{\frac{1}{q}-\frac{1}{2}+\frac{1}{p}+\frac{k-3-l}{2}}{\frac{k}{2}+\frac{1}{q}-\frac{3}{2}},
	\end{equation}
	and $p=\frac{2q(k-3)}{(q-3)l+2(k-3)}$. Moreover, we have that $a+b=1-\frac{\frac{1}{2}-\frac{1}{q}}{\frac{k-3}{2}+\frac{1}{q}}\in(0,1)$ is independent of $l$. 
\end{lemma}

\bibliographystyle{plain}
\bibliography{reference}

\begin{thebibliography}{10}

\bibitem{abbrescia2022emergence}
Leo Abbrescia and Jared Speck.
\newblock {The emergence of the singular boundary from the crease in $3D$
  compressible Euler flow}.
\newblock {\em arXiv preprint arXiv:2207.07107}, 2022.

\bibitem{alinhac1999blowup2}
Serge Alinhac.
\newblock Blowup of small data solutions for a class of quasilinear wave
  equations in two space dimensions. {II}.
\newblock {\em Acta Math.}, 182(1):1--23, 1999.

\bibitem{alinhac1999blowup}
Serge Alinhac.
\newblock Blowup of small data solutions for a quasilinear wave equation in two
  space dimensions.
\newblock {\em Ann. of Math. (2)}, 149(1):97--127, 1999.

\bibitem{an2020low}
Xinliang An, Haoyang Chen, and Silu Yin.
\newblock {Low regularity ill-posedness for elastic waves driven by shock
  formation}.
\newblock {\em American Journal of Mathematics}, to appear, 2020.
\newblock {\em arXiv preprint arXiv:2003.03195}.

\bibitem{an2021low}
Xinliang An, Haoyang Chen, and Silu Yin.
\newblock {Low regularity ill-posedness and shock formation for 3D ideal
  compressible MHD}.
\newblock {\em arXiv preprint arXiv:2110.10647}, 2021.

\bibitem{an2022h}
Xinliang An, Haoyang Chen, and Silu Yin.
\newblock {$H^{\frac{11}{4}}(\mathbb{R}^2)$ ill-posedness for 2D Elastic Wave
  system}.
\newblock {\em arXiv preprint arXiv:2206.14012}, 2022.

\bibitem{an2022low}
Xinliang An, Haoyang Chen, and Silu Yin.
\newblock Low regularity ill-posedness for non-strictly hyperbolic systems in
  three dimensions.
\newblock {\em J. Math. Phys.}, 63(5):Paper No. 051503, 16, 2022.

\bibitem{an2022cauchy}
Xinliang An, Haoyang Chen, and Silu Yin.
\newblock {The Cauchy problems for the 2D compressible Euler equations and
  ideal MHD system are ill-posed in $H^{\frac{7}{4}}(\mathbb{R}^2)$}.
\newblock {\em arXiv preprint arXiv:2206.14003}, 2022.

\bibitem{buckmaster2022simultaneous}
Tristan Buckmaster, Theodore~D. Drivas, Steve Shkoller, and Vlad Vicol.
\newblock Simultaneous development of shocks and cusps for 2{D} {E}uler with
  azimuthal symmetry from smooth data.
\newblock {\em Ann. PDE}, 8(2):Paper No. 26, 199, 2022.

\bibitem{buckmaster2022-2d-unstableformation}
Tristan Buckmaster and Sameer Iyer.
\newblock Formation of unstable shocks for 2{D} isentropic compressible
  {E}uler.
\newblock {\em Comm. Math. Phys.}, 389(1):197--271, 2022.

\bibitem{buckmaster2020shock}
Tristan Buckmaster, Steve Shkoller, and Vlad Vicol.
\newblock {Shock formation and vorticity creation for 3d Euler}.
\newblock {\em Communications on Pure and Applied Mathematics}, to appear.
\newblock {\em arXiv preprint arXiv:2006.14789}.

\bibitem{buckmaster2019-2d-formation}
Tristan Buckmaster, Steve Shkoller, and Vlad Vicol.
\newblock Formation of shocks for 2{D} isentropic compressible {E}uler.
\newblock {\em Comm. Pure Appl. Math.}, 75(9):2069--2120, 2022.

\bibitem{buckmaster2022formation}
Tristan Buckmaster, Steve Shkoller, and Vlad Vicol.
\newblock {Formation of Points Shocks for 3D Euler}.
\newblock {\em Communications on Pure and Applied Mathematics}, to appear.
\newblock {\em arXiv preprint arXiv:1912.04429}.

\bibitem{christodoulou2007formation}
Demetrios Christodoulou.
\newblock {\em The formation of shocks in 3-dimensional fluids}.
\newblock EMS Monographs in Mathematics. European Mathematical Society (EMS),
  Z\"{u}rich, 2007.

\bibitem{christodoulou2019shock}
Demetrios Christodoulou.
\newblock {\em The shock development problem}.
\newblock EMS Monographs in Mathematics. European Mathematical Society (EMS),
  Z\"{u}rich, 2019.

\bibitem{christodoulou2016shock}
Demetrios Christodoulou and Andr\'{e} Lisibach.
\newblock Shock development in spherical symmetry.
\newblock {\em Ann. PDE}, 2(1):Art. 3, 246, 2016.

\bibitem{christodoulou2014compressible}
Demetrios Christodoulou and Shuang Miao.
\newblock {\em Compressible flow and {E}uler's equations}, volume~9 of {\em
  Surveys of Modern Mathematics}.
\newblock International Press, Somerville, MA; Higher Education Press, Beijing,
  2014.

\bibitem{collot2018singularity}
Charles Collot, Tej-Eddine Ghoul, and Nader Masmoudi.
\newblock Singularity formation for {B}urgers' equation with transverse
  viscosity.
\newblock {\em Ann. Sci. \'{E}c. Norm. Sup\'{e}r. (4)}, 55(4):1047--1133, 2022.

\bibitem{eggers2008role}
Jens Eggers and Marco~A. Fontelos.
\newblock The role of self-similarity in singularities of partial differential
  equations.
\newblock {\em Nonlinearity}, 22(1):R1--R44, 2009.

\bibitem{john1974formation}
Fritz John.
\newblock Formation of singularities in one-dimensional nonlinear wave
  propagation.
\newblock {\em Comm. Pure Appl. Math.}, 27:377--405, 1974.

\bibitem{liu1979development}
Tai~Ping Liu.
\newblock Development of singularities in the nonlinear waves for quasilinear
  hyperbolic partial differential equations.
\newblock {\em J. Differential Equations}, 33(1):92--111, 1979.

\bibitem{luk2018shock}
Jonathan Luk and Jared Speck.
\newblock Shock formation in solutions to the 2{D} compressible {E}uler
  equations in the presence of non-zero vorticity.
\newblock {\em Invent. Math.}, 214(1):1--169, 2018.

\bibitem{luk2021stability}
Jonathan Luk and Jared Speck.
\newblock {The stability of simple plane-symmetric shock formation for 3D
  compressible Euler flow with vorticity and entropy}.
\newblock {\em arXiv preprint arXiv:2107.03426}, 2021.

\bibitem{majda2012compressible}
A.~Majda.
\newblock {\em Compressible fluid flow and systems of conservation laws in
  several space variables}, volume~53 of {\em Applied Mathematical Sciences}.
\newblock Springer-Verlag, New York, 1984.

\bibitem{merle1996asymptotics}
Frank Merle.
\newblock Asymptotics for {$L^2$} minimal blow-up solutions of critical
  nonlinear {S}chr\"{o}dinger equation.
\newblock {\em Ann. Inst. H. Poincar\'{e} C Anal. Non Lin\'{e}aire},
  13(5):553--565, 1996.

\bibitem{merle2005blow}
Frank Merle and Pierre Raphael.
\newblock The blow-up dynamic and upper bound on the blow-up rate for critical
  nonlinear {S}chr\"{o}dinger equation.
\newblock {\em Ann. of Math. (2)}, 161(1):157--222, 2005.

\bibitem{merle2020strongly}
Frank Merle, Pierre Rapha\"{e}l, and Jeremie Szeftel.
\newblock On strongly anisotropic type {I} blowup.
\newblock {\em Int. Math. Res. Not. IMRN}, (2):541--606, 2020.

\bibitem{merle1997stability}
Frank Merle and Hatem Zaag.
\newblock Stability of the blow-up profile for equations of the type
  {$u_t=\Delta u+|u|^{p-1}u$}.
\newblock {\em Duke Math. J.}, 86(1):143--195, 1997.

\bibitem{miao2017formation}
Shuang Miao and Pin Yu.
\newblock On the formation of shocks for quasilinear wave equations.
\newblock {\em Invent. Math.}, 207(2):697--831, 2017.

\bibitem{pasqualotto2022gradient}
Sung-Jin Oh and Federico Pasqualotto.
\newblock {Gradient blow-up for dispersive and dissipative perturbations of the
  Burgers equation}.
\newblock {\em arXiv preprint arXiv:2107.07172}, 2021.

\bibitem{qiu2021shock}
Yiya Qiu and Lifeng Zhao.
\newblock {Shock Formation of 3D Euler-Poisson System for Electron Fluid with
  Steady Ion Background}.
\newblock {\em arXiv preprint arXiv:2108.09972}, 2021.

\bibitem{riemann1858uber}
B~Riemann.
\newblock {Uber die Fortpflanzung ebener Luftwellen von endlicher
  Schwingungsweite, Gottingen Abbandlungen}.
\newblock {\em Bd. VIII, S}, 43:157, 1858.

\bibitem{sideris1985formation}
Thomas~C. Sideris.
\newblock Formation of singularities in three-dimensional compressible fluids.
\newblock {\em Comm. Math. Phys.}, 101(4):475--485, 1985.

\bibitem{speck2016shock}
Jared Speck.
\newblock {\em Shock formation in small-data solutions to 3{D} quasilinear wave
  equations}, volume 214 of {\em Mathematical Surveys and Monographs}.
\newblock American Mathematical Society, Providence, RI, 2016.

\bibitem{speck2016stable}
Jared Speck, Gustav Holzegel, Jonathan Luk, and Willie Wong.
\newblock Stable shock formation for nearly simple outgoing plane symmetric
  waves.
\newblock {\em Ann. PDE}, 2(2):Art. 10, 198, 2016.

\bibitem{yang2021shock}
Ruoxuan Yang.
\newblock Shock formation of the {B}urgers-{H}ilbert equation.
\newblock {\em SIAM J. Math. Anal.}, 53(5):5756--5802, 2021.

\end{thebibliography}

\end{document}